\documentclass[a4paper,12pt,oneside,reqno]{amsart}
\usepackage{hyperref}
\usepackage[headinclude,DIV13]{typearea}
\areaset{15.1cm}{25.0cm}
\usepackage{amsfonts,amssymb,amsmath,amsthm,bbm}
\usepackage{cancel} 
\usepackage[latin1] {inputenc}
\usepackage{graphicx, psfrag}

\newtheorem{theorem}{Theorem}[section]
\newtheorem{lemma}[theorem]{Lemma}
\newtheorem{proposition}[theorem]{Proposition}
\newtheorem{corollary}[theorem]{Corollary}

\newtheorem{example}{Example}

\theoremstyle{definition}
\newtheorem{definition}[theorem]{Definition}

\theoremstyle{remark}
\newtheorem*{remark}{Remark}

\def\paragraph#1{\noindent \textbf{#1}}

\numberwithin{equation}{section}

\def\d{\mathrm{d}}

\def\<{\langle}
\def\>{\rangle}
\def\a{\alpha}

\def\e{\epsilon}

\def\k{\kappa}

\def\R{{\Bbb R}}  
\def\N{{\Bbb N}}  
\def\Z{{\Bbb Z}}  
\def\Q{{\Bbb Q}}  

\def\H{{\Bbb H}}

\let\cal=\mathcal

 \def \k {{\kappa}}
 
 \def \e {{\varepsilon}}

 \def \d {{\delta}}
 \def \a {{\alpha}}

 \def \ba {\begin{array}}
 \def \ea {\end{array}}

 \def \cC {{\cal C}}
 \def \cD {{\cal D}}
 
 \def \cF {{\cal F}}
 
 \def \cH {{\cal H}}
 
 \def \cL {{\cal L}}
 \def \cM {{\cal M}}
 \def \cN {{\cal N}}

 \def \cR {{\cal R}}
 \def \cS {{\cal S}}
 \def \cU {{\cal U}}

 \newcommand{\be}{\begin{equation}}
 \newcommand{\ee}{\end{equation}}

\newcommand{\bea}{\begin{eqnarray}}
 \newcommand{\eea}{\end{eqnarray}}
\def\TH(#1){\label{#1}}\def\thv(#1){\ref{#1}}
\def\Eq(#1){\label{#1}}\def\eqv(#1){(\ref{#1})}

 \def \1{\mathbbm{1}}

\begin{document}

 \title[ Jarn{\'\i}k-type Inequalites]
{ Jarn{\'\i}k-type Inequalities}

\author[S. Weil]{Steffen Weil}
\address{
Steffen Weil\\
School of Mathematical Sciences\\
Tel Aviv University\\
Tel Aviv 69978\\
Israel
}
\email{steffen.weil@math.uzh.ch}


\begin{abstract} 
It is well known due to Jarn{\'\i}k \cite{Jarnik} that the set $\textbf{Bad}_\R^1$ of badly approximable numbers is of Hausdorff dimension one.
If $\textbf{Bad}_\R^1( c)$ denotes the subset of $x\in \textbf{Bad}_\R^1$ for which the approximation constant $c(x) \geq c$,
then Jarn{\'\i}k was in fact more precise and gave non-trivial lower and upper bounds on the Hausdorff dimension of $\textbf{Bad}_\R^1(c)$ in terms of the parameter $c>0$.
Our aim is to determine simple conditions on a framework which allow to extend 'Jarn{\'\i}k's inequality' to further examples.
For many dynamical examples, these extensions are related to the Hausdorff dimension of the set of orbits which avoid a suitable given neighborhood of an obstacle.
Among the applications, we discuss the set $\textbf{Bad}_{\R^n}^{\bar r}$ of badly approximable vectors in $\R^n$ with weights $\bar r$,
the set of orbits in the Bernoulli-shift which avoid a neighborhood of a periodic orbit,
the set of geodesics in the hyperbolic space $\H^n$ which avoid a suitable collection of convex sets,
and the set of orbits of a toral endomorphism which avoid neighborhoods of a separated set.
\end{abstract} 

\maketitle

\section{Introduction and Main Results}

\subsection{Introduction}
An irrational number $x\in \R$ is called \emph{badly approximable} if there exists a positive constant $c = c(x) >0$, called \emph{approximation constant}, %
such that 
\be
\label{1}
	\lvert x- \frac{p}{q} \rvert \geq \frac{c}{q^2}
\ee
for all $p\in \Z$ and $q \in \N$; we may set $c(x) = \inf_{(p,q) \in \Z \times \N} q^2 \lvert  x- \frac{p}{q} \vert$.
The set \textbf{Bad}$_\R^1$ of badly approximable numbers is a Lebesgue null-set, 
yet it is well known due to Jarn{\'\i}k \cite{Jarnik} that \textbf{Bad}$_\R^1$ is of Hausdorff dimension one.
Note that a positive irrational number $x \in \R$ is badly approximable if and only if 
the entries $a_n \in \N$ of the continued fraction expansion $x= [a_0; a_1,a_2, \dots]$ of $x$ are bounded by some integer $N\in \N$.
More precisely, let $M_N$ denote the set of irrational numbers for which the entries of the continued fraction expansion are bounded by $N$,
and let $\textbf{Bad}_\R^1(c)$ denote the set of badly approximable numbers with $c(x) \geq c$.
Then,
\be
\nonumber
	M_{N} \subset \textbf{Bad}_\R^1(\tfrac{1}{N+2}) \subset M_{N+2}.
\footnote{
This can be seen by the following.
If $p_n/q_n$ are the approximates given by the continued fraction expansion of $x$, then
$	\tfrac{1}{(a_{n+1}+2)q_n^2} < \lvert x - \tfrac{p_n}{q_n} \rvert < \tfrac{1}{a_{n+1} q_n^2}$.
Moreover, if $\lvert x - p/q \rvert < 1/(2q^2)$, then $p/q=p_n/q_n$ for a suitable $n$.
}
\ee
Using this correspondence, Jarn{\'\i}k was in fact more precise and gave non-trivial lower and upper bounds 
on the Hausdorff dimension of the set of badly approximable numbers with approximation constants bounded  from below.

\begin{theorem}[\cite{Jarnik}, Satz $4$]
For $N > 8$, we have
\be
\label{Jarnik}
	1- \frac{4}{N \log(2) } \leq \emph{dim}(M_N) \leq 1 - \frac{1}{ 8N \log( N)}.
\ee
Here and in the following, '\emph{dim}' stands for the Hausdorff dimension.
\end{theorem}

\noindent In particular, inequality \eqref{Jarnik}, which we call \emph{Jarn{\'\i}k's inequality}, 
implies Jarn{\'\i}k's theorem on full Hausdorff dimension of  \textbf{Bad}$_\R^1$.
Various authors continued the study of the set $M_N$, see for example Shallit \cite{Shallit} (and references therein) for a survey, or Hensley \cite{Hensley}  who determined the asymptotics, as $N \to \infty$, of dim$(M_N)$ up to second order.  The  motivation of this paper, however, is the relation of Jarn{\'\i}k's inequality to the following dynamical question.

Let $X=(X,d)$ be a metric space and $T:  X \to X$ a continuous transformation. 
Let $\cal{O} \in X$  be an \emph{obstacle}.%
\footnote{ Instead of a point, we may alternatively consider an object, such as a topological end of $X$ or a set in $X$, providing suitable  neighborhoods in $X$ (see the examples below).}
For a subset $S \subset X$ we obtain the quadruple $\cD= (X, T, \cal{O}, S)$ and consider the set 
\be
\nonumber
	\textbf{Bad}_{\cal{D}} \equiv \{ x \in S: c(x) = \inf_{n\in \N_0} d( T^n(x), \cal{O})  > 0  \}%
	\footnote{ Here and in the following, $\N_0$ denotes the natural numbers including $0$.} 
\ee
of points in $S$ for which the orbit avoids some (open) ball around  $\cal{O}$ of radius $c$ where $c=c(x)$ depends on $x$.
Notice that when $\mu$ is an ergodic Borel measure with respect to $T$ and $\cal{O}$ lies in the support of $\mu$ then $\textbf{Bad}_{\cal{D}}$ is a $\mu$-null set.
The set $\textbf{Bad}_{\cal{D}}$ has been studied for several examples with different techniques, for instance via Schmidt's game and its winning sets (see for instance \cite{Dani2, Schmidt}).
As a result various qualitative properties such as full Hausdorff dimension (that is the dimension of $S$), a property of winning sets  in a reasonably nice setting, have been achieved.
Our goal is to determine \emph{quantitative} results on the dimension of the set
\be
\nonumber
	\textbf{Bad}_{\cal{D}} (c) \equiv \{ x \in S: T^n(x) \not \in B(\cal{O},c) \text{ for all } n \in \N_0 \}%
\ee
for a given small  $c>0$.
It is worth pointing out that if the dimension of $\textbf{Bad}_{\cal{D}}(c)$ is less than the one of $S$, then in many cases $\textbf{Bad}_{\cal{D}} (c)$ cannot be a winning set for Schmidt's game.

For example, recall the following correspondence which is part of Dani's correspondence.
Let $\H^2 / SL(2, \Z)$ be the modular surface, which is a hyperbolic orbifold with a cusp; for details, we refer to Section \ref{GeodesicFlow}.
Let $H_0$ be the maximal standard cusp neighborhood and denote by $H_t \subset H_0$ the standard cusp neighborhood at \emph{height} $t \geq 0$. 
The set of complete 'cuspidal' geodesics $\gamma$ with $\gamma(0) \in \partial H_0$, $\gamma(-t) \in H_t$ (hence starting from the cusp)
can be identified with the set $[0,1)$ via the endpoint $\tilde \gamma(\infty) \in [0,1)$ of a suitable lift $\tilde \gamma$ of $\gamma$, 
starting from $\infty$. 
We say that $\gamma$ is \emph{bounded} 
with height $t=t(\gamma)$ if $\gamma\lvert_{\R^+}$ avoids the cusp neighborhood $H_t$.
Again, $\gamma$ is bounded if and only if 
$x=\tilde\gamma(\infty) \in [0,1)$ is a badly approximable number
and a small height $t(\gamma)$ corresponds to a large approximation constant $c(x)$.
Jarn{\'\i}k's inequality \eqref{Jarnik} thereby gives non-trivial bounds, in terms of  the height $t$, on the Hausdorff dimension of the set of cuspidal geodesics in the modular surface  avoiding a small given cusp neighborhood $H_t$.

While Kristensen, Thorn and Velani \cite{KristensenEtAl} extended Jarn{\'\i}k's result on full Hausdorff dimension to a more general setting,
our intention is to determine simple conditions on a framework which enables to extend  Jarn{\'\i}k's inequality to further examples
 - we call such inequalities   \emph{Jarn{\'\i}k-type inequalities}.
We remark that implicitly in the proof of \cite{KristensenEtAl} (as well as in proofs of many other works)
a lower bound on the Hausdorff dimension of a given set of badly approximable points with a lower bound on the approximation constant  can be determined.
However, the bound is not stated explicitly and we are furthermore interested in  precise asymptotics of the dimension.

\subsection{A sample of the main results}
Among the applications in Section \ref{Apps}, we now present several Jarn{\'\i}k-type inequalities in their simplest settings.
For $n\geq 1$, let \textbf{Bad}$_{\R^n}^{n}$ be the set of points $\bar x \in \R^n$ 
for which there exists a positive constant $c(\bar x)>0$ such that 
the distances (say in the supremum-norm) from $\bar x$ to all rational vectors satisfy
\be
\nonumber
\label{Bedinung}
 	 \lVert \bar x- \frac{\bar p}{q} \rVert \geq \frac{c(\bar x)}{q^{1+1/n}},
\ee
for every $q\in \N$ and $\bar p  \in \Z^n$.
The set \textbf{Bad}$_{\R^1}^1$ is the classical set of badly approximable numbers 
and  \textbf{Bad}$_{\R^n}^n$ is called the set of badly approximable vectors.
For $c>0$, let moreover \textbf{Bad}$_{\R^n}^n(c)$ be the subset of $\bar x \in \textbf{Bad}_{\R^n}^n$  
with approximation constant $c(\bar x) \geq c $.

\begin{theorem} 
\label{ThmA}
There exist positive 
 constants $k_l$, $k_u$ and $t_0\geq 0$, depending only on $n$, such that
for all $t>t_0$ we have
\bea
\nonumber
			n  -   \frac{  k_l  }{ t \cdot e^{ 1/(2n)t}   }
	\leq 	\emph{dim(\textbf{Bad}}_{\R^n}^n(e^{-t/n} )) \leq 
			n  - \frac{ k_u  }{t \cdot e^{ (n+1)t} }. 
\eea
\end{theorem}

\noindent
Note that $\textbf{Bad}_{\R^n}^n$  is  a Schmidt-winning set, see \cite{Schmidt} (and even HAW-winning, \cite{BroderickEtAl}). 
For $n=1$ and large $t=\log(N)$ an inequality similar to Jarn{\'\i}k's inequality \eqref{Jarnik} is recovered.
More generally, we will also consider the set  \textbf{Bad}$_{\R^n}^{\bar r}$ of badly approximable vectors with weight vector $\bar r$,
as well as intersections of \textbf{Bad}$_{\R^n}^{\bar r}$ with suitable 'diffuse' sets
which are, more precisely, supports of absolutely decaying measures; see Section \ref{BadRn}.
Note that  Broderick and Kleinbock \cite{BroderickKleinbock} recently extended the result for $\textbf{Bad}_{\R^n}^n$ to the set of badly approximable matrices.
For the case of  $\textbf{Bad}_{\R^n}^n$ their bounds are similar but their upper bound is sharper.
However, our proof to determine the upper bound  follows from an axiomatic approach which applies to many examples. 

In the following we study the Hausdorff dimension of orbits of a dynamical system which avoid a small given neighborhood of an obstacle.
Our first example is an application of Theorem \ref{ThmA}. 
Let $\cL_{n+1}$ be the space  of unimodular lattices in $\R^{n+1}$ which is a non-compact space with one 'thin' end%
 \footnote{ By a thin end we mean a topological end for which the injectivity radius tends to zero along every sequence that leaves every compact subset of  $\cL_{n+1}$. }
that we view as the obstacle;
we refer to \cite{KWBAVectors} for details and background.
When $\Lambda \in \cL_{n+1}$ is given by $ g\Z^{n+1}$ for some $g \in SL_n(\R)$, 
let $\Delta(\Lambda) \equiv \min\{ \lVert gv \rVert : v \in \Z^{n+1}, v \neq 0\}$.
For small $\e>0$, let $\cL_{n+1}(\e) \equiv \{ \Lambda \in \cL_{n+1}: \Delta(\Lambda) \leq \e\}$ which is a neighborhood of the end,
in particular, $\cL_{n+1}(\e)^C$ is compact.
Consider the one-parameter semigroup $F^+\equiv \{g_t: t \geq 0\}$, where $g_t\equiv \text{diag}(e^{t}, \dots, e^{t}, e^{-nt})$, acting on $\cL_{n+1}$ by left-multiplication, that is $g_t \Lambda = g_t g \Z^{n+1}$ for $\Lambda = g\Z^{n+1}$.
Moreover, for $\bar x\in \R^n$ consider the unimodular lattice
\be
\nonumber
	\Lambda_{\bar x} \equiv 
	\bigl(
	 \begin{smallmatrix}
  		I_{n\times n} & \bar x \\  0 & 1  \end{smallmatrix}
	\bigr)\Z^{n+1} = \{ (q \bar x - \bar p, q) : (\bar p, q) \in \Z^n \times \Z\}.
\ee
The 'Dani correspondence' states that $\bar x \in \textbf{Bad}_{\R^n}^n$ if and only if the trajectory $F^+ \Lambda_x = \{g_t \Lambda_{\bar x}  : t\geq0\}$ is bounded in $\cL_{n+1}$, or in other words, it avoids some neighborhood $\cL_{n+1}(\e)$ for some $\e=\e(\Lambda_{\bar x})$; a similar result is true for \textbf{Bad}$_{\R^n}^{\bar r}$ if we consider an adjusted semigroup $F_{\bar r}^+$.
More precisely, a computation (see \cite{BroderickKleinbock}, Lemma $3.1$) shows $\e(\Lambda_{\bar x}) = c(\bar x)^{\frac{n}{n+1}}$, where $c(\bar x)$ denotes the approximation constant of $\bar x$.
Theorem \ref{ThmA} then shows:

\begin{corollary}
There are  positive constants $\tilde k_l$, $\tilde k_u>0$ such that for $t\geq \tilde t_0 $ sufficiently large we have
\bea
\nonumber
			n  -   \frac{ \tilde k_l  }{ t \cdot e^{ \frac{1}{2n} t}   }
	\leq 	\emph{dim}(\{ \bar x \in \R^n : g_s \Lambda_{\bar x} \not\in \cL_{n+1}(e^{- t/(n+1)}) \text{ for all } s\geq 0  \}) \leq 
			n  - \frac{\tilde k_u  }{t \cdot e^{ (n+1) t} }.
\eea
\end{corollary}

Let us now present the main results of Section \ref{GeodesicFlow} in the simplest setting; see Section \ref{GeodesicFlow} for details and generalizations.
In the following let $M=\H^n/\Gamma$ be a complete $(n+1)$-dimensional finite volume hyperbolic manifold.
For a point $o \in M$ let $SM_o$ be the $n$-dimensional unit tangent sphere of $M$ at $o$.
Identify a vector $v \in SM_o$ with the unique geodesic $\gamma_v: \R_{>0} \to M$, called a ray, starting at $o$ such that $\dot \gamma_v(0) = v$.

First assume that $M=(M,d)$ has precisely one cusp, which is in particular a thin end, 
that we choose as the obstacle.
Let $H_0$ be a sufficiently small standard cusp neighborhood and let $H_t \subset H_0$ be the standard cusp neighborhood at height $t=d(H_t, H_0)$.
Fix a base point $o \in M - H_0$ in the compact part of $M$. 
For $t\geq 0$ define the set of rays $\gamma_v$, $v \in SM_o$, which avoid the cusp neighborhood $H_t$ (and stay in the compact part $H_t^C$) by
\be
\nonumber
	\textbf{Bad}_{M,H_0, o}(t) \equiv \{v \in SM_o : \gamma_v( s ) \not \in H_t \text{ for all } s\geq 0 \}.
\ee

\begin{theorem}
\label{ThmB}
There exist positive constants $k_l$, $k_u$ and a height $t_0\geq 0$, depending on $M$ and the choices of $H_0$ and $o$, such that
for all $t>t_0$ we have
\bea
\nonumber
 	n - \frac{k_l }{t \cdot e^{ n/2 t}}\leq \emph{ dim(\textbf{Bad}}_{M,H_0, o}(t) ) \leq n -   \frac{ k_u }{t \cdot e^{2 n t}}.
\eea
\end{theorem}

\noindent The set of 'bounded' rays follows to be of full Hausdorff dimension, earlier shown by \cite{Patterson,StratmannDAKleinian}, 
and is even an absolute winning set, see \cite{McMullen}.

Now let $M$ be compact and choose a closed geodesic $\alpha$  in $M$ as the obstacle.
Fix $\e_0>0$ sufficiently small with respect to $\a$ and consider the closed $\e_0$-neighborhood $\cN_{\e_0}(\a)$  of $\a$ in $M$.
Let $o \in M$.
Given a vector $v \in SM_o$ define the \emph{penetration length} of $\gamma_v$ at time $t\geq0$ by
$\cL_v( t)= 0$ if $\gamma_v(t) \not \in \cN_{\e_0}(\a)$ and otherwise by
$\cL_v( t) \equiv \ell(I)$,
where $\ell(I)$ denotes the length of the maximal connected interval $I\subset \R^+$ such that  $t \in I$ and  $\gamma_v(s) \in \cN_{\e_0}(\a)$ for all $s\in I$.
Note that when $\gamma_v$ has bounded penetration lengths in the neighborhood $\cN_{\e_0}(\a)$ of $\alpha$ in $M$
then $\dot \gamma_v$ avoids a small neighborhood of $\dot \alpha$ (depending on the penetration lengths) in the unit tangent bundle $SM$ of $M$.
Hence, for a given length $L>0$ define 
\be
\nonumber
	\textbf{Bad}_{M,\cN_{\e_0}(\a), o}(L)  \equiv \{v \in SM_o : \cL_v(t) \leq L \text{ for all } t\geq 0 \}.
\ee

\begin{theorem}
\label{ThmC}
There exist positive constants $k_l$, $k_u >0$ and a length $L_0\geq 0$, depending on $M$, $\a$ and the choices of $\e$ and $o$, such that
for all lengths $L>L_0$ we have
\be
\nonumber
 	n - \frac{ k_l }{L \cdot  e^{n/2 L}}\leq \emph{dim}(\textbf{Bad}_{M,\cN_{\e_0}(\a), o}(L) ) \leq n -  \frac{k_u }{L \cdot e^{n L}}.
\ee
\end{theorem}

\noindent  Again, the set of 'bounded' rays follows to be of full Hausdorff dimension, earlier shown by \cite{Patterson,StratmannDAKleinian}, 
and is even an absolute winning set, see \cite{Weil2}.

Further Jarn\'ik-type inequalities (and generalizations of the above ones) will be obtained in Section \ref{Apps}.
In particular we moreover consider the set of words in the Bernoulli shift which avoid a periodic word (Section \ref{Bernoulli}),
and the set of orbits of toral endomorphisms which avoid separated sets of $\R^n$ (Section \ref{ToralEndo}).

\subsection{Further remarks} 
A given set \textbf{Bad} of badly approximable elements (or bounded orbits) in $X$ as above,  defines a \emph{spectrum} $\cS \equiv \{c(x) : x \in \textbf{Bad} \} $  in terms of the approximation constants $c(x)$.
In the case of \textbf{Bad}$_{\R}^1$, $\cS$ is the classical Markoff spectrum for which various properties are known, see \cite{CusickFlahive}.
Define the function $\frak{D}$ on $\cS$ by
\be
\nonumber
	\frak{D} : \cS \to [0, \text{dim}(X)],\ \ \ \ c\mapsto \text{dim}(\textbf{Bad} (c) ),
\ee
where $\textbf{Bad} (c) \equiv \{x \in \textbf{Bad}: c(x) \geq c\}$.
In a similar fashion it is possible to define further suitable functions on $\cS$, such as $\frak{D}_0(c) \equiv \text{dim}(\{x \in \textbf{Bad} : c(x) = c\})$.
Clearly the function $\frak{D}$ is non-increasing and $\frak{D}_0(c) \leq \frak{D}(c)$.
Our results above establish non-trivial estimates for the function $\frak{D}$ in the corresponding context
which in turn give further information about the spectrum.
Both functions deserve further study and provide many questions,
such as what is the set of discontinuities, the domain of positivity, or the asymptotics of $\frak{D}(c)$ and $\frak{D}_0(c)$ as $c \to 0$.
\\

\emph{Outline of the paper.}
In Section \ref{SectionBounds}, we introduce the framework and conditions in an axiomatic fashion which lead to the lower and upper bound on the Hausdorff dimension of a set of badly approximable points with respect to a given lower bound on the approximation constant (see Sections \ref{SectionFramework} and  \ref{SectionResonantSetsAndBounds} respectively).
In Section \ref{SectionMeasure}  we derive the required conditions from 'local measure conditions' (and separation conditions), which both concern the parameter space as well as the structure and distribution of the resonant sets.

In Section \ref{Apps}, we apply the deduced bounds to the set of badly approximable vectors with weights (Section \ref{BadRn}),
to the set of words in the Bernoulli shift which avoid a periodic word (Section \ref{Bernoulli}),
to the set of geodesics in a geometrically finite hyperbolic manifold which are bounded with respect to a suitable collection convex sets (Section \ref{GeodesicFlow}),
and to the set of orbits of toral endomorphisms which avoid separated sets of $\R^n$ (Section \ref{ToralEndo}).
\\

 \emph{Acknowledgments.}
The author is grateful to Barak Weiss for helpful suggestions which led to further results.
He thanks an anonymous referee for numerous valuable suggestions and remarks, improving the paper considerably. 
This research has been partially supported by the Swiss National Science Foundation, Project 135091, and the ERC starter grant DLGAPS 279893.


\section{The Geometry of Parameter Spaces and the Axiomatic Approach}
\label{SectionBounds}

The idea of the axiomatic approach and the required conditions are simple, yet hidden below technicalities. 
We therefore want to roughly explain it for the basic example \textbf{Bad}$_{\R}^1$, the set of badly approximable numbers (see Section \ref{BadRn}).
For $r>0$, let $R(r) \equiv \{ p/q \in\Q : \tfrac{1}{q^2}\geq r \}$. 
Fix a sufficiently large parameter $c>0$.
For the lower bound, we start with any closed metric ball $B_1= B(x, 1)$.
Now, given a closed metric ball $B=B_{1i_2 \dots i_k}$ of radius $r=e^{-2 k c}$ at the $k$.th step, 
we consider  the 'relevant set' $\Delta_l(k,c) =\bigcup_{p/q \in R(r \cdot l_*)} B(p/q, e^{-2c}r)$. 
The constant $l_*=3$ guarantees that at most one of the balls $B(p/q, e^{-2c}r)$ with $p/q\in R(r \cdot l_*)$ can intersect $B$.
Hence, with respect to the Lebesgue measure $\mu$, 
\be
\label{Decaying0}
 	\mu( B \cap \bigcup_{p/q \in R(r \cdot l_*)} B(\tfrac{p}{q}, e^{-2c}r) ) \leq  e^{-2c} \mu(B) \equiv \tau_l(c) \cdot \mu(B).
\ee
Up to further separation constants, we can find disjoint balls  $B_{1i_2 \dots i_k i_{k+1}}$ of radius $e^{-2c}r$ contained in $B$ and in the complement of $\Delta_l(k,c)$.
The number of these balls can be estimated from below in terms of $\tau_l(c)$.
Thus, step by step, we construct a treelike collection of 'sub-covers' of the set \textbf{Bad}$_{\R}^1( e^{-2 \tilde c})$ with $\tilde c$ related to $c$.
This will yield a lower bound on the Hausdorff dimension of \textbf{Bad}$_{\R}^1(e^{- 2\tilde c})$ in terms of $\tau_l(c)$.
 
For the upper bound, 
 given again a closed metric ball $B=B_{1i_2 \dots i_k}$ of radius $r_k= u_*^ke^{-4 k c}$ at the $k$.th step, 
we consider  the 'relevant set' $\Delta_u(k,c)= \bigcup_{p/q \in R(r_k \cdot u_c)} B(p/q, \tfrac{e^{-2c}}{q^2})$.
The parameter $u_c= u_* e^{-2c}$ guarantees that either $B$ is contained in a set  $B(p/q, \tfrac{e^{-2c}}{q^2})$  with $p/q \in R(r_{k-1} \cdot u_c)$ 
or that there exists a point $p/q \in R(r_k \cdot u_c)$ with $B(p/q, e^{-4c}r_k) \subset B$.
In both cases,
\be
\label{Dirichlet0}
 	\mu( B \cap \bigcup_{p/q \in R(r_k \cdot u_c)} B(\tfrac{p}{q}, \tfrac{e^{-2c}}{q^2}) ) \geq  e^{-4c} \mu(B) \equiv \tau_u(c) \cdot \mu(B).
\ee
Again, up to further separation constants, we can find closed balls  $B_{1i_2 \dots i_k i_{k+1}}$ of radius $u_*e^{-4c}r_k$ covering the complement of  $\Delta_u(k,c)$ in $B$,
for which the number can be estimated from above in terms of $\tau_u(c)$.
Thus, step by step, we construct a treelike collection of covers of the set \textbf{Bad}$_{\R}^1( e^{-2c})\cap B_1$.
This will yield an upper bound on the Hausdorff dimension of \textbf{Bad}$_{\R}^1(e^{- 2c})$ in terms of $\tau_u(c)$.
 
For our axiomatic approach, we will in fact assume the conditions \eqref{Decaying0} and \eqref{Dirichlet0} as well as separation conditions
 and construct  treelike collections of 'sub-covers' and covers respectively as above.

\begin{remark}
Our setting and axioms are similar to the \emph{local ubiquity} setup of Beresnevich, Dickinson and Velani \cite{BeresnevichEtAl}.
In particular, our main conditions \eqref{Decaying0} and \eqref{Dirichlet0} (as well as \eqref{Decaying} and \eqref{Dirichlet} respectively) are similar to their \emph{intersection conditions}.
However, their methods served the purpose of determining the Hausdorff dimension of the complementary set, that is the set of well-approximable points and of 'limsup sets' in general.
\end{remark}


\subsection{The general framework.}
\label{SectionFramework}
We first introduce the setting of this section that is based on the notion of \cite{KleinbockWeiss}
and was adapted in the author's earlier work \cite{Weil2}.
However, some of the following terminology differs from these works.

The following setting will be considered throughout Section \ref{SectionBounds}.
Let $(\bar X,d)$ be a proper metric space.
Fix $t_* \in \R \cup \{- \infty\}$ and define the parameter space $\bar \Omega \equiv \bar X \times (t_*, \infty)$, the set of \emph{formal balls} in $\bar X$.
Let $\cal{C}(\bar X)$ be the set of nonempty compact subsets of $\bar X$.
Assume that there exists a function 
\be
\nonumber
	\bar \psi : \bar \Omega \to \cal{C}(\bar X)
\ee
which is \emph{monotonic},  that is, for all $(x,t)\in \bar \Omega$ and $s\geq 0$ we have
\be
\label{Mono}
	\bar \psi(x,t+s) \subset \bar \psi(x, t).
\ee

\begin{example}
For instance, since $\bar X$ is proper, set $t_* = - \infty$.
Given $\sigma>0$, the \emph{standard function} $\bar B_{\sigma}$ is given by 
\be
\label{Standard}
	\bar B_{\sigma} (x,t) \equiv B(x, e^{- \sigma t})
\ee
which is a monotonic function, where $B(x,r) \equiv \{y\in \bar X: d(x,y)\leq r\} \in \cal{C}(\bar X)$ 
for $x\in \bar X$, $r>0$.

A vector $\bar \sigma=(\sigma_1, \dots, \sigma_n) \in \R^n_{>0}$ determines the monotonic \emph{rectangle function}
\be
\label{Rectangular}
	\bar R_{\bar \sigma} ( (x_1, \dots, x_n),t) \equiv B_{\sigma_1}(x_1, t) \times \dots \times B_{\sigma_n}(x_n, t) 
\ee
on $\bar \Omega= (\bar X \times \dots \times \bar X) \times \R$.

If $\bar X=\R$ and $\bar \sigma=(\sigma, \dots, \sigma)$ 
we write $\bar Q_{\sigma}= \bar R_{\bar \sigma} $ where $\bar Q_{\sigma}( (x_1, \dots, x_n),t)$ is a cube in $\R^n$ of radius $2e^{- \sigma t}$ centered at 
$(x_1, \dots, x_n)$. We call $\bar Q_{\sigma}$ the \emph{cube function}.
\end{example}

For a subset $Y\subset \bar X$ and  $t>t_*$, we call $(Y,t) \equiv \{(y, t) : y\in Y\} $ \emph{formal neighborhood},
and define $\cal{P}= \cal{P}(\bar X) \times (t_*, \infty) $ to be the set of formal neighborhoods.
Define the \emph{$\bar \psi$-neighborhood} of $(Y,t) \in \cal{P}$ by
\be
\nonumber
	\cN(Y,t) = \cN_{\bar \psi}(Y,t) \equiv \bigcup_{y \in Y}\bar \psi (y,t). 
\ee
Note that $ \bar \psi(Y,t + s ) \subset \bar \psi(Y,t)$ for all $s \geq 0$ by monotonicity \eqref{Mono}.

In many applications, we are interested in badly approximable points of a closed subset $X$ of $\bar X$
 which is, with the induced metric, a complete metric space.
However, we do not require the resonant sets to be contained in $X$ but in $\bar X$.
Therefore, let also $\Omega= X \times (t_*, \infty) \subset \bar \Omega$.
The monotonic function $\bar \psi$ induces the monotonic function $\psi : \Omega \to \cal{C}(X)$, 
defined by
\be
\nonumber
	\psi(\omega) \equiv \bar \psi(\omega) \cap X, \ \ \ \omega \in \Omega.
\ee

Now let $\mu$ be a locally finite Borel measure on $\bar X$ and notice that $\bar \psi(\omega)$ is a Borel set for $\omega \in \bar \Omega$.
We say that  $(\Omega, \psi, \mu)$ satisfies a \emph{power law with respect to the parameters $(\tau, c_1, c_2)$} (short a $\tau$-power law or $(\tau, c_1, c_2)$-power law),
where  $\tau>0$, $c_{2}\geq c_{1}>0$,
if supp$(\mu)=X$ and
\be
\label{PowerLaw}
	c_1 e^{- \tau  t } \leq \mu( \psi(x,t) ) \leq c_2 e^{- \tau t}
\ee
for all formal balls $(x,t) \in \Omega$.
This extends the classical notion for the case $\psi=B_1$.
Note  that the exponent $\tau$ from \eqref{PowerLaw} might differ for two different parameter spaces $(\Omega, \psi_1, \mu)$ and  $(\Omega, \psi_2, \mu)$.

The \emph{lower pointwise dimension} of $\mu$ at $x\in $ supp$(\mu)$ is defined by%
\be
\nonumber
	d_{\mu}(x) \equiv \liminf_{r \to 0} \frac{\log( \mu(B(x,r) )}{\log r}.
\ee
When $(\Omega, B_1, \mu)$  satisfies a $\tau$-power law then $d_{\mu}(x)= \tau$ for any $x \in X$.
For a nonempty subset $Z$ of $X \cap \text{supp}(\mu)$, define $d_{\mu}(Z) \equiv \inf_{x \in Z}d_{\mu}(x)$.

\subsection{The family of resonant sets, conditions and dimension estimates}
\label{SectionResonantSetsAndBounds}
Let $s_*\in \R$ with $s_* >t_*$.
Consider a countable family of subsets $R_n$ in $\bar X$, $n \in \N$, indexed by the natural numbers, which are called \emph{resonant sets};
while the concept of resonant sets comes from \cite{BeresnevichEtAl} we remark that our assumptions below slightly differ. 
To each resonant set $R_n$ we assign a \emph{size} $s_{n}\geq s_*$ (also called a \emph{height}).
Denote this family by
\be	
\nonumber
\label{Family}
	\cal{F}= \{(R_{n}, s_{n}): n \in \N\}.
\ee
Require that the family $\cal{F}$ is  \emph{increasing} and \emph{discrete}, that is,
\begin{itemize}
\item[${[I]}$] $R_n \subset R_{n+1}$ and $s_n \leq s_{n+1}$ for every $n \in \N$, and
\item[${[D]}$] $ \lvert \{ n \in \N : s_{n} \leq t\} \rvert < \infty$ for all $t>s_1$. 
\end{itemize}
Each pair $(R_n,s_n)$ gives a $\bar \psi$-neighborhood of $(R_n, s_n)$ with, by monotonicity of $\bar \psi$, 
\be
\nonumber
\label{Contraction}
	\cN(R_{n}, s_{n} + c) \subset \cN(R_{n}, s_{n} ), \ \ \ c\geq 0.
\ee
We then define the set of \emph{badly approximable points} with respect to $\cal{F}$ by 
\be
\nonumber
	\textbf{Bad}_X^{\bar \psi}(\cal{F})= \{x\in X : \exists \ c =c(x)< \infty \text{ such that } x  \not \in \bigcup_{n \in \N}\cN(R_n, s_n + c)  \},
\ee
or simply by $\textbf{Bad}(\cal{F})$ if there is no confusion about the parameter spaces under consideration.
The constant $c(x)\equiv \inf\{c \in \R : x \not  \in \bigcup_{n \in\N}  \cN(R_n, s_n + c)  \}$ 
is called the \emph{approximation constant} of $x\in\textbf{Bad}(\cal{F})$.
In the following, we are interested in the subset 
\be
\nonumber
	\textbf{Bad}(\cal{F}, c) \equiv \{x\in X : x  \not \in \bigcup_{n \in \N} \cN(R_n, s_n + c)  \}.
\ee

Using $[I]$ and $[D]$, we define the \emph{relevant resonant set} given the parameter $t\geq s_1$, 
\be
\nonumber
	\cR(t) \equiv \bigcup_{ s_n \leq t} R_n = R_{n_t},
\ee
where $n_t \in \N$ is the largest integer such $s_n \leq t$, and we call $s_{n_t}$  the \emph{relevant size}. 


\subsubsection{Dimension estimates}
We now present the main conditions and result of Section \ref{SectionBounds}.

Fix a constant $d_*\geq0$ for later purpose and assume we are given a parameter $c  \geq d_*$.
For the lower bound, let $l=l_c\geq 0$ and
for $k\in \N_0$ define $t_k\equiv s_1 + kc + l$ and 
\be
	L_{k}(c) = L_{k}^{\bar \psi}(c)  \equiv   \bigcap_{i=1}^k \cN(\cR(t_i - l), t_i + c - d_* )^C .
\ee

Let $\mu$ be a locally finite Borel measure on $\bar X$ for which we assume the following.
\begin{itemize}
\item[{[$\mu>0$]}] 
The measure $\mu$ is positive on $\psi$-balls, that is, for all $\omega \in \Omega$ we have 
\be
\label{PositiveMeasure}
	\mu(\psi(\omega))>0.
\ee
\item[{[$\tau(c)$]}]
There exists a constant $\tau(c) >0$ as follows:
given a formal ball $\omega =(x, t_k) \in\Omega$, $k \in \N_0$, with $x \in L_{k-1}(c)$ 
there is a collection $\cC_{l,c}(\omega)$ of formal balls $\omega_{i}  =(x_i, t_{k+1}) \in \Omega$
satisfying 
\be
\label{Disjoint}
	\psi(\omega_{i}) \subset \psi( \omega) - \cN( \cR(t_k - l ) ,t_k + c) 
\ee
where $x_{i} \in L_{k}(c)$, such that $\psi(\omega_i)$ are essentially-disjoint (that is, $\mu(\psi(\omega_i) \cap \psi(\omega_j))=0$ for $i\neq j$),
and moreover,
\be
\label{Schranke}
	\mu \big(\bigcup_{\omega_i \in \cC_{l,c}(\omega)} \psi(\omega_i) \big) \geq \tau(c)\cdot\mu(\psi(\omega))
\ee
\end{itemize}

For the upper bound, let $u=u_c \geq 0$ and for $k \in \N_0$
define $\bar t_k = s_1 + k(c+u) - u $ and
\bea
\nonumber
	U_k(c)=U_k^{\bar \psi}(c)  &\equiv&  \bigcap_{s_n \leq \bar t_k + u} \cN(R_n, s_n + c )^C.
\eea

\begin{itemize}
\item[{[N(c)]}]
There exists a constant $N(c) \geq 0$ as follows:
given a formal ball $\omega =(x, \bar t_k ) \in\Omega$, $k \in \N_0$, with $x \in U_{k-1}(c)$
there exists a collection $\cC_{u,c}(\omega)$ of formal balls $\omega_{i}  =(x_i, \bar t_{k+1} ) \in \Omega$ with $x_i \in U_{k}(c)$ 
satisfying 
\be
\label{Covering}
	\psi(\omega) -  \bigcup_{s_n \leq \bar t_k + u} \cN(R_n, s_n + c ) \subset \bigcup_{\omega_i \in \cC_{u,c}(\omega)} \psi( \omega_i) 
\ee
and with cardinality
\be
\nonumber
	\lvert \cC_{u,c}(\omega) \rvert \leq N(c).
\ee
\end{itemize}

Finally, for both bounds we require that the diameter of $\psi$-balls is bounded.

\begin{itemize}
\item[{[$\sigma$]}]
There exist positive constants $c_{\sigma}$ and $\sigma$ such that 
for all $\omega=(x,t) \in \Omega$, the diameter of $\psi(\omega)$ is bounded by
\be
\label{Diam}
	\text{diam}(\psi(x,t)) \leq c_{\sigma} e^{-\sigma t}.
\ee
\end{itemize}

\begin{remark}
As will be evident from the proof, it suffices to require $[N(c)]$ for all $k \geq k_0$ for some $k_0 \in \N$ and obtain an upper bound for intersections $\textbf{Bad}(\cal{F}, c ) \cap \psi(x, \bar t_{k_0})$, see Section \ref{SectionProofOfThmBounds}.
Likewise, by requiring the existence of a formal ball $\omega=(x, t_{k_0})$ with $\psi(\omega) \subset L_k(c)$,
it suffices to require $[\tau(c)]$ for all $k \geq k_0$. 
\end{remark}

Note that the conditions $[\mu>0]$ and $[\sigma]$ depend only on the parameter space $(\Omega, \psi, \mu)$,
whereas $[\tau(c)]$ and $[N(c)]$ depend also on the family $\cF$. 
Further discussion of the latter conditions will take place in the next Section \ref{SectionMeasure}.

The following theorem is the main result of this section. 
Its proof is skipped to Section \ref{SectionProofOfThmBounds} below.
The above conditions are used to inductively construct tree-like collections of covers and 'sub-covers'%
\footnote{ By a collection of sub-covers we mean a collection of sets which give rise to a limit set contained in $\textbf{Bad}(\cal{F}, c)$, see Section \ref{SectionProofOfThmBounds} below. }
 of $\textbf{Bad}(\cal{F}, c)$.
We point out that conditions similar to $[\mu>0]$, $[\tau(c)]$ and  $[\sigma]$ were already used in \cite{KleinbockWeiss} to derive a dimension estimate.

\begin{theorem}
\label{ThmBounds}
Let $(\Omega, \psi)$ be a parameter space and $\cal{F}$ be a family as above satisfying $[\sigma]$. 
\begin{itemize}
\item[{[LB]}]
If $\mu$ is a locally finite Borel-measure satisfying $[\mu>0]$ and $[\tau(c)]$ then 
\be
	\text{dim}(\textbf{Bad}(\cal{F}, 2c  + l_c)) \geq   d_\mu \big(\textbf{Bad}(\cal{F}, 2c  + l_c) \big) 
	- \frac{ \lvert \log(\tau(c)) \rvert  }{ \sigma c }.
\ee
\item[{[UB]}] If $[N(c)]$ is satisfied, then for any formal ball $\omega=(x, \bar t_0) \in \Omega$, 
\be
	\text{dim}(\textbf{Bad}(\cal{F}, c ) \cap \psi(\omega)) \leq
	  \frac{\log(  N(c) )}{\sigma ( c + u_c)}.
\ee
\end{itemize}
\end{theorem}

For the upper bound, recall that by the countable stability of the Hausdorff dimension, 
given a countable collection $\cU = \{ U_n \}_{n\in \N}$ of sets $U_n \subset X$ which cover $X$, we have
\be
\nonumber
	\text{dim}(\textbf{Bad}(\cal{F}, c )) = \text{dim}(\cup_n \textbf{Bad}(\cal{F}, c ) \cap U_n) \leq \sup_{n \in \N} \text{dim}(\textbf{Bad}(\cal{F}, c ) \cap U_n) .
\ee


\subsection{Deriving $[\tau(c)]$ and $[N(c)]$ from 'local measure conditions'}
\label{SectionMeasure}

Before we present the proof Theorem \ref{ThmBounds}, we discuss how to derive conditions $[\tau(c)]$ and $[N(c)]$ from local measure (and separation) conditions.
First we treat the case of a general parameter space $(\Omega, \psi)$ and then consider the special case that $X$ is the Euclidean space and $\psi=Q_\sigma$ (the cube function) or $\psi=R_{\bar \sigma}$ (the rectangle function)
in order to obtain sharper estimates.

In the following, let $(\Omega, \psi)$ be a parameter space and $\cal{F}$ be a family as given above.
The following 'separation' and 'contraction' conditions depend on the parameter space $(\Omega, \psi)$ 
as well as on the 'local structure' of the resonant sets $R_n$.

Fix $d_*\geq 0$. Then consider the conditions $[d_*]$ and $[d_*, \cF]$ and notice that we choose the same constant $d_* = d_*(\psi, \cF)$ for both of them to avoid technicalities, 
but considering a dependency in terms of a parameter $c>0$ would  yield sharper bounds below.

\begin{itemize}
\item[{[$d_*$]}]
$(\Omega, \psi)$ is called \emph{$d_*$-contracting} if for all $(y,t) \in \bar \Omega$ and $x \in X$,
\bea
\label{Separating1}
	&& x \in \bar \psi(y,t + d_*)  \implies \bar \psi(x, t+d_*) \subset \bar \psi(y,t),
	 \\ \nonumber
	&& x\not \in \bar \psi(y,t  ) \implies \psi(x,t + d_*) \cap \bar  \psi(y, t + d_*) = \emptyset.
\eea 
\item[{[$d_*, \cF$]}]
Moreover, require that $(\Omega, \psi)$ is \emph{$d_*$-separating with respect to $\cal{F}$}, that is,  
for all resonant sets $R_n \subset \bar X$, $t\geq t_*$, and for all $x \in X$, 
\bea
\label{Separating2}
	&& x\not \in \cN(R_n,t  ) \implies \psi(x,t + d_*) \cap  \cN(R_n, t + d_*) = \emptyset.
\eea
\end{itemize}

\begin{example}
Clearly, the standard function $ B_{\sigma}(x,t) \equiv B(x, e^{- \sigma t})$ is $d_*$-contracting for 
\be
\nonumber
	d_* = \log(2)/\sigma. 
\ee
\end{example}

\subsubsection{The general case}
Given the parameter $c>0$ recall the definitions of $t_k$, $\bar t_k$ and $L_k(c)$, $U_k(c)$ respectively.
Let $\mu$ be a locally finite Borel measure on $\bar X$, 
assume in addition that for every resonant set $R_n$ and $t\geq s_n$, $\bar \psi(R_n, t)$ is a Borel set.
Suppose that $(\Omega, \psi, \mu)$ satisfies the following conditions.

\begin{itemize}
\item[{[$k_c, \bar k_c$]}]
There are positive constants $\bar k_c, k_c$ such that 
for all formal balls $\omega=(x, t_k) \in \Omega$ with $x \in L_{k-1}(c )$ and $y\in \psi(\omega) \cap L_k(c)$, 
\be
\label{LowerLaw}
		k_c \ \mu(\psi(x,t_k)) \leq \mu(\psi(y, t_{k+1}))  \leq \mu(\psi(y,t_{k+1}-d_*)) \leq \bar k_c \  \mu(\psi(x,t_k + d_*)).
\ee

\item[{[$K_c$]}]
There is a positive constant $K_c$ such that 
for all formal balls $(x, \bar t_k -d_*) \in \Omega$ with $x \in U_{k-1}(c)$
and $y \in \psi(x, \bar t_k) \cap U_k(c)$,
\be
\label{UpperLaw}
	\mu( \psi(y, \bar t_{k+1} + d_*) ) \geq K_c \cdot   \mu( \psi(x, \bar t_k - d_*) ).
\ee
\end{itemize}

\begin{example}
When $(\Omega, \psi, \mu)$ satisfies a power law with respect to the parameters $(\tau, c_1, c_2)$,
then $[k_c, \bar k_c]$ and $[K_c]$ hold, independently from the conditions that $x \in L_{k-1}(c )$ and $y\in \psi(\omega) \cap L_k(c)$ or $x \in U_{k-1}(c)$
and $y \in \psi(x, \bar t_k) \cap U_k(c)$ respectively, with
\bea
\label{Inequalities}
	k_c = \tfrac{c_1}{c_2}e^{- \tau c},  \ \ \ \bar k_c = \tfrac{c_2}{c_1} e^{-\tau( c - 2d_*) } , \ \ \ K_c= \tfrac{c_1}{c_2} e^{- \tau( c + u_c + 2d_*) }.
\eea
\end{example}

The following two local intersection conditions are the crucial conditions that need to be verified in applications and are therefore presented as definitions. 
We again point out the similarity to the intersection conditions of \cite{BeresnevichEtAl}.

The concept of (absolutely) decaying measures was introduced in \cite{LindenstraussEtAl} and we adapted it to our setting in \cite{Weil2} with respect to a given family $\cal{F}$.

\begin{definition}
\label{DefDecaying}
 $(\Omega, \psi, \mu)$ is called%
\footnote{ To be more precise we should call this condition 'absolutely $\tau_l(c)$-decaying' rather than $\tau_l(c)$-decaying according to \cite{LindenstraussEtAl}. For the sake of simplicity we omit the term 'absolutely'.}
\emph{$\tau_l(c)$-decaying with respect to $\cal{F}$ and the parameters $(c,l_c)$}, where $\tau_l(c)<1$,
if all formal balls $\omega=(x, t_k+ d_*) \in \Omega$ with $x\in L_{k-1}(c )$ we have
\be
\label{Decaying}
	\mu(\psi(\omega) \cap  \cN(\cR( t_k -l_c), t_k + c - d_*) ) \leq \tau_l(c) \cdot \mu(\psi(\omega)).
\ee
\end{definition}

\begin{remark}
For $c \geq d_*$, the condition that  $x\in L_{k-1}(c )$ 
implies that $\psi(x, t_k)$ is disjoint to $ \cN(\cR( t_{k-1} -l_c), t_{k-1} + c) \supset \cN(\cR( t_{k-1} -l_c), t_k  +c- d_*)$ by \eqref{Separating2}.
Hence it would  suffice to consider the set $\cR(t_k -l_c, c) \equiv \cR( t_{k} -l_c) - \cR( t_k -l_c - c )$ in \eqref{Decaying}.
Note that also the proof of Lemma \ref{NotIn} will work if we only consider the sets $\cR(t_k - l_c, c)$.
\end{remark}

Note that we called the next condition 'Dirichlet' since \eqref{Dirichlet} will follow from Dirichlet-type results in the applications.

\begin{definition}
\label{DefDirichlet}
 $(\Omega, \psi, \mu)$ is called \emph{$\tau_u(c)$-Dirichlet with respect to $\cal{F}$ and the parameters $(c,u_c)$}, where $\tau_u(c)>0$,
if for all formal balls $\omega= (x, \bar t_k -d_*) \in \Omega$ with $x \in U_{k-1}(c)$ 
we have
\be
\label{Dirichlet}
	\mu \big( \psi(\omega) \cap \bigcup_{s_n\leq  \bar t_k + u_c } \cN( R_n, s_n+  c  + d_*) \big) \geq \tau_u(c) \cdot \mu( \psi(\omega)).
\ee
\end{definition}

The above conditions can be transferred  into the conditions $[\tau(c)]$ and $[N(c)]$.

\begin{proposition}
\label{PropositionMeasure}
Given $(\Omega, \psi, \mu)$ and $\cal{F}$ as above satisfying $[\mu>0]$, $[d_*]$ and $[d_*, \cF]$.
\begin{itemize}
\item[1.]
If $(\Omega, \psi, \mu)$ satisfies $[k_c, \bar k_c]$ and is $\tau_l(c)$-decaying with respect to $\cal{F}$ and the parameters $(c,l_c)$
then $[\tau(c)]$ is satisfied with
\be
	\tau(c)= \frac{(1-\tau_l(c)) k_c }{2\bar k_c}.
\ee
\item[2.]
If $(\Omega, \psi, \mu)$ satisfies $[K_c]$ and is $\tau_u(c)$-Dirichlet with respect to $\cal{F}$ and the parameters $(c,u_c)$
then $[N(c)]$ is satisfied with

\be
	 N(c) \leq \frac{(1- \tau_u(c))}{K_c}.
\ee
\end{itemize}
\end{proposition}

\begin{proof}
We start with the first assertion.
Given the formal ball $\omega = (x, t_k)\in \Omega$ where $x \in L_{k-1}(c)$, 
assume that we have $m\geq 0$ formal balls $\omega_{i} = (x_i, t_{k+1}) \in \Omega$, $x_i \in L_k(c)$,
for which \eqref{Disjoint} is satisfied and such that  the $\psi$-balls $ \psi (\omega_i )$ are disjoint. 
We apply  \eqref{Decaying}  on the formal ball $\omega_0 \equiv (x, t_k + d_*) \in \Omega$ 
and use \eqref{LowerLaw} so that we  obtain
\bea
\nonumber
	&& \mu \big(  \psi(\omega_0 ) - \cN(\cR(t_k - l), t_k + c - d_* )) 
		- \bigcup_{ i=1}^m \psi( x_i, t_{k+1} - d_*)) \big)
	\\ \nonumber
	&=& \mu ( \psi(\omega_0 )) - \mu \big(\psi(\omega_0) \cap \big(  \cN(\cR(t_k - l), t_k + c -  d_*) \cup 
		 \bigcup^m_{ i=1} \psi( x_i, t_{k+1} -  d_*) \big) \big) 
	\\ \nonumber
	&\geq&
		(1- \tau_l(c) - m \cdot \bar k_c) \mu(\psi(\omega_0)).
\eea
As long as $m< (1- \tau_l(c))\bar k_c^{-1}$, since $\mu(\psi(\omega_0))>0$ by \eqref{PositiveMeasure},
there exists a  point 
\be
\nonumber
	x' \in  \psi(\omega_0 ) - \cN(\cR(t_k - l), t_k + c - d_*) 
		- \bigcup_{i=1}^m \psi( x_i, t_{k+1} - d_*),
\ee
in particular, $x' \in L_{k}(c)$.
Define $\omega_{m+1} \equiv (x', t_k +  c  ) \in \Omega$.
By \eqref{Separating1} and \eqref{Separating2} we know that $\psi( \omega_{m+1} )$ is disjoint from  both, 
$\cup_{i=1}^m \psi(\omega_i)$ 
as well as $\cN(\cR(t_k- l) , t_k +c)$.
Moreover, by \eqref{Separating1} we have that 
$\psi( \omega_{m+1} ) \subset \psi(\omega)$.
Iterating this argument until $m+1\geq (1- \tau_l(c))\bar k_c^{-1}$, set $\cC_l(\omega)=\{\omega_i : i=1, \dots, m\}$
for which we see by \eqref{LowerLaw} that 
\bea
\nonumber
	\mu \big(\bigcup_{i=1}^m \psi(\omega_{i}) \big)
	&\geq& m \cdot k_c \cdot \mu(\psi(\omega))
	\\ \nonumber
	&\geq& \frac{m+1}{2} k_c \cdot \mu(\psi(\omega))	
	\geq \frac{(1-\tau_l(c)) k_c }{2\bar k_c} \cdot \mu(\psi(\omega)),
	\eea
which concludes the first part.

For the second part, let now $\omega=(x, \bar t_k) \in \Omega$ with $x \in U_{k-1}(c)$.
Suppose that the we are already given $m\geq 0$ points 
$x_i  \in \psi(\omega) \cap U_{k}(c)$
such that $\psi(x_i, \bar t_{k+1} + d_*)$ are disjoint.
Note that if there exists $x'\in \psi(\omega) \cap U_{k}(c)$ with $x' \not \in \cup \psi(x_{i}, \bar t_{k+1})$,
then $\psi(x', \bar t_{k+1}+ d_*)$ is disjoint to $\cup \psi(x_{i}, \bar t_{k+1} + d_* )$ by \eqref{Separating1}
and we set $\omega_{m+1} \equiv (x', \bar t_{k+1}) \in \Omega$.
So let $m$ be the maximal number with respect to this property (which is finite as seen below) and
let $\cC_{u,c}(\omega)=\{\omega_i : i=1, \dots m\}$ be the collection of these formal balls
such that 
\be
\nonumber
	\psi(\omega) - \bigcup_{s_n \leq \bar t_k + u} \cN(R_n, s_n + c )= \psi(\omega) \cap U_{k}(c) \subset \bigcup_{i=1}^m \psi(\omega_i).
\ee
Moreover, by \eqref{Separating2}  and since 
\be
\nonumber
	\bar t_{k+1}  +d_* = \bar t_{k} + (c+u) +d_* \geq s_n + c +d_* ,
\ee
for all $s_n \leq \bar t_k + u$,
the $\psi$-balls $\psi(x_i, \bar t_{k+1}  + d_*)$  are disjoint to $\bar \psi(R_n, s_n  + c  + d_*) $ when $s_n \leq \bar t_k + u$.
In addition, they are contained in $\psi( x_i, \bar t_k - d_*)$  by \eqref{Separating1}.
Hence,  \eqref{UpperLaw} and \eqref{Dirichlet} applied to the formal ball $\omega_0 = (x, \bar t_k -d_*)$ imply
\bea
	\nonumber
 	\mu( \psi( \omega_0)) 
	&\geq& \mu( \psi(\omega_0) \cap \bigcup_{s_n \leq \bar t_k + u} \cN(R_n, s_k + c + d_*)) +
	 \sum_{i=1}^{m} \mu( \psi(x_i, \bar t_{k+1} + d_* ))
	 \\ 	\nonumber
	& \geq & (\tau_u(c) + m \cdot K_c)  \mu( \psi(\omega_0)).
\eea
Using \eqref{PositiveMeasure}, we get 
\be
\nonumber
	 \lvert \cC_{u,c}(\omega) \rvert = m \leq \frac{(1-\tau_u(c))}{K_c},
\ee
which establishes the second assertion.
\end{proof}


\subsubsection{The cube function}
\label{SectionQube}
Lor $\sigma >0$, let $B_{\sigma}(x,t) \equiv B(x, e^{-\sigma t})$. 
Then $\mu$ satisfies a power law.
However, even in this case, the resulting dimension estimates might not be sharp
because the constants $k_c$, $\bar k_c$ and $K_c$ depend sensitively on the separation constant $d_*$.
Note that in the applications this may lead to an upper bound on the dimension which exceeds the one of $\R^n$, hence is trivial.
However, we may sharpen the bounds of Proposition \ref{PropositionMeasure} by modifying the above arguments and shifting the separation constants into $\tau_l(c)$ and $\tau_u(c)$.

Assume in the following that we are given the parameter space $(\Omega, Q_\sigma, \mu)$, where $\Omega=\R^n\times \R$ and $\mu$ denotes the Lebesgue measure on $\R^n$.
Recall that the monotonic \emph{cube function} $Q_\sigma$ on $\Omega $ is given by
\be
\nonumber
	Q_{\sigma}(x,t) \equiv B(x_1,e^{-\sigma t}) \times \cdots \times B(x_n,e^{-\sigma t}),
\ee
which denotes the $n$-dimensional cube of edge length $2 e^{-\sigma t }$ with center $x=(x_1, \dots, x_n)$. 
The crucial point is that, given any cube $Q\equiv Q_{\sigma}(x,t) \subset \R^n$ and $s=\log(m)/\sigma$ for some $m\in \N$, 
we can  find a partition into $m^{n}=e^{\sigma s}$ cubes $Q_i\equiv Q_{\sigma}(x_i,t + s)$ satisfying
\bea
\label{CubeCover}
		\mu(Q_i \cap Q_j) &=& 0 \ \ \ \text{ for $i\neq j$, and } 
\\ \nonumber
		Q &=& \bigcup_i Q_i .
\eea

Up to increasing $d_*$, assume that $d_*\geq \log(2)/\sigma$, where $Q_\sigma$ is clearly $\log(2)/\sigma$-contracting.
Also adjust the Definitions \ref{DefDecaying} and \ref{DefDirichlet} as follows (where the constants $l_c$ and $u_c$ and hence the times $t_k$, $\bar t_k$ remain fixed).
Modify \eqref{Decaying} and require that for all formal balls $\omega=(x, t_k) \in \Omega$ with $x\in L_{k-1}(c)$ we have
\be
\label{Decaying2}
	\mu(Q_{\sigma}(\omega) \cap  \cN(\cR( t_k - l_c), t_k + c -  2d_*) ) \leq  \tau_l(c) \cdot \mu(Q_{\sigma}(\omega)).
\ee
In addition modify \eqref{Dirichlet} as well as the definition of $U_k(c)$ and require that for all formal balls $\omega= (x, \bar t_k ) \in \Omega$ with $x \in U_{k-1}(c)
\equiv \bigcap_{s_n \leq  \bar t_{k-1}  + u_c - 2d_* } Q_{\sigma}(R_n, s_n + c + d_*)^C $ 
we have
\be
\label{Dirichlet2}
	\mu(Q_{\sigma}(\omega) \cap \bigcup_{s_n\leq \bar t_k + u_c - 2d_*} \cN( R_n, s_n+  c  + d_*)) \geq  \tau_u(c) \cdot \mu(Q_{\sigma}(\omega)).
\ee

The above partition then strengthens Proposition \ref{PropositionMeasure} to the following.

\begin{proposition}
\label{PropositionCube}
Given $(\Omega, Q_\sigma, \mu)$ and $\cal{F}$ as above satisfying $[d_*, \cF]$.
\begin{itemize}
\item[1.]
Let $c=\log(m)/\sigma$ for some $m \in \N$.
If \eqref{Decaying2} holds,
then $[\tau(c)]$ is satisfied with
\be
\nonumber
	\tau(c)= 1-\tau_l(c).
\ee
\item[2.]
Let $c+u_c=\log(m)/\sigma$ for some $m \in \N$.
If \eqref{Dirichlet2} holds,
then $[N(c)]$ is satisfied with
\be
\nonumber
	 N(c) \leq (1-\tau_u(c)) \cdot e^{n\sigma(c+u_c)}.
\ee
\end{itemize}
\end{proposition}

\begin{remark}
Note that the restriction to $c=\log(m)/\sigma$ will not be a severe one in the applications, since for sufficiently large $c>0$ 
we can choose a $\bar c = \log(m)/\sigma$ with $\bar c\leq c$ and obtain a lower bound with respect to $\bar c$.
The defect can again be shifted to a multiplicative constant in $\tau(c)$.
Similar applies for the upper bound.
\end{remark}

\begin{remark}
The improvement (which we will notice in the applications) of $\tau(c)$ and $N(c)$ relies on the partition \eqref{CubeCover} of cubes.
This is no longer possible in general, not even for subsets such as fractals of the Euclidean space.
Then again, for instance regular Cantor sets or the Sierpinski Carpet admit a similar partition and hence a possible improvement of the above constants; see also Example \ref{Bernoulli}.
However, for these examples the delicate point seems to be to obtain non-trivial parameters $\tau_l(c)$ and $\tau_u(c)$ in the conditions \eqref{Decaying2} and \eqref{Dirichlet2} respectively.  
\end{remark}

\begin{proof}[Proof of Proposition \ref{PropositionCube}]
For the first assertion, if $Q=Q_{\sigma}(x, t_k)$ is a given cube with $x \in L_{k-1}(c)$, 
let $\cC_{l,c}(Q)$ be precisely the collection of cubes $Q_{i} = Q_{\sigma}(x_i, t_k +c)$ (rather than formal balls) of the partition of $Q$ as in \eqref{CubeCover},
which intersect
\be
\nonumber
	Q\cap \cN( \cR(t_k - l_c), t_k+  c - 2d_*  )^C,
\ee
and hence cover $Q\cap \cN( \cR(t_k -  l_c), t_k+  c - 2d_*  )^C$.
Suppose $\bar Q$ as above   intersects $Q\cap \cN( \cR(t_k - l_c), t_k+  c - 2d_*  )^C$ in a point $y$.
Then $\bar Q \subset Q_{\sigma}(y,  t_{k+1} - d_*)$ because $d_* \geq \log(2)/\sigma$, and, 
since $y \not \in  \cN( \cR(t_k - l_c), t_k+  c -  d_* )$, the supset is disjoint to 
\be
\nonumber
	\cN( \cR(t_k -  l_c), t_k +  c  - d_*) \supset \cN( \cR(t_k - l_c), t_k +  c  )
\ee
by $(d_*2)$.
In particular, $x_i \in Q(y, t_k)$.
The first assertion now follows from \eqref{Decaying2}, showing
\be
\nonumber
	\mu \big(\bigcup_{Q_i \in \cC_{l,c}(Q)} Q_i \big)
	 \geq (1- \tau_l(c)) \cdot \mu(Q) = \tau_l(c) \cdot \mu(Q).
\ee

For the second assertion,
if $Q=Q_{\sigma}(x, \bar t_k)$ is a given cube with $x \in U_{k-1}(c)$ 
let $\cC_{u,c}(x, \bar t_k)$  be precisely the cubes $Q_i = Q_{\sigma}(x_i,\bar t_{k+1})$ (rather than formal balls) of the partition of $Q$ as in \eqref{CubeCover},
which intersect the set
\be
\nonumber
	W\equiv Q\cap \bigcap_{s_n\leq \bar t_k + u_c  - 2d_* } \cN( R_n, s_n+  c )^C,
\ee
and hence cover $W$.
Let $\bar Q = Q_i$ be such  a cube which intersects $W$ in a point $y$.
Then $\bar Q \subset Q_{\sigma}(y, \bar t_{k+1} -d_*)$.
Moreover, when $s_n \leq \bar t_k + u_c -2 d_*$, we have
\be
\nonumber
	\bar t_{k+1} - d_* = \bar t_k+(c+ u_c) - d_*    \geq s_n + c  + d_* 
\ee 
so that $Q_{\sigma}(y, \bar t_{k+1} - d_* ) \subset Q_{\sigma}(y, s_n + c +d_*)$ where the supset is disjoint to 
$\cN( R_n, s_n+  c + d_* )$ by $(d_*2)$.
In particular, $x_i \in U_k(c)$.
Using that $\mu( Q_{\sigma}(y, t+s ) ) = e^{- n\sigma s}   \mu( Q_{\sigma}(x, t ))$
and applying \eqref{Dirichlet2} we get (for any $j$)
\be
\nonumber
	\lvert \cC_{u,c}(x, \bar t_k) \rvert =\frac{ \mu(\cup_{Q_i \in \cC_{u,c}(x, \bar t_k)} Q_i)}{\mu(Q_j)}
	\leq (1-\tau_u(c)) e^{n\sigma(c+u_c)},
\ee
finishing the proof.
\end{proof}


For later purpose, we shortly discuss how to obtain the conditions  considered above, given the respective conditions for the parameter space $(\Omega, B_\sigma, \mu)$ and the family $\cal{F}$.
Note that for all formal balls $(x,t) \in \Omega$ we have 
\be
\label{Comparison}
	Q_{\sigma}(x, t + \sqrt{n}/\sigma ) \subset B_{\sigma}(x,t) \subset Q_{\sigma}(x,t ) .
\ee
Thus, $\textbf{Bad}_{\R^n}^{Q_{\sigma}}(\cal{F}, c  ) \subset \textbf{Bad}_{\R^n}^{B_{\sigma}}(\cal{F}, c ) \subset \textbf{Bad}_{\R^n}^{Q_{\sigma}}(\cal{F}, c + \sqrt{n}/\sigma )$.
Moreover, if $(\Omega, B_{\sigma})$ is $d_*$-separating with respect to $\cal{F}$,  let $\bar d_* \equiv d_* + \sqrt{n}/\sigma$
and we see that $(\Omega, Q_{\sigma})$ is at least $\bar d_*$-separating with respect to $\cal{F}$.

We shall also show the following technical lemma
that can be used to establish the conditions \eqref{Decaying2} (respectively \eqref{Dirichlet2}) when $(\Omega, B_\sigma, \mu)$ is $\tau_l(c)$-decaying (respectively $\tau_u(c)$-Dirichlet) with respect to $\cal{F}$, at least in special cases.

\begin{lemma} 
\label{TildeDecayingLemma}
Let $a \equiv 2\sqrt{n}/\sigma + 3 d_*$.
If $(\Omega, B_\sigma, \mu)$ is $\tau_l(c)$-decaying with respect to $\cal{F}$ and the parameters $(c,l_c)$, independent on the condition that $x \in L_k^{B_\sigma}(c)$,
then 
$(\Omega, Q_\sigma, \mu)$ satisfies \eqref{Decaying2} for the parameters $(c, l_c + a )$ 
 and
\be
\nonumber
	\tilde \tau_l(c) \leq e^{n\sigma( a - d_*)} \tau_l(c).
\ee

Assume that $(\Omega, B_\sigma, \mu)$ is $\tau_u(c+a)$-Dirichlet with respect to $\cal{F}$ and the parameters $(c+a,u_c+ a)$
and independent on the condition that $t=\bar t_k$, that is,
whenever $(x, t-d_*) \in \Omega$ with $x \in \bigcap_{s_n\leq t+ (u_c + a)} \cN_{B_{\sigma}}( R_n, s_n+  (c+a) )^C $
then 
\be
\label{Bedingung}
	\mu \big(B_{\sigma}(\omega) \cap \bigcup_{s_n\leq t + (u_c +a)} \cN_{B_{\sigma}}( R_n, s_n+  (c  +a) + d_*) \big) 
	\geq \tau_u(c+a) \cdot \mu( B_{\sigma}(\omega)).
\ee
Then
$(\Omega, Q_\sigma, \mu)$ satisfies \eqref{Dirichlet2} for the parameters $(c, u_c + 2a)$ and
\be
\nonumber
	\tilde \tau_u(c) \geq e^{-n\sigma( a + 2d_*+ \sqrt{n}/\sigma)} \tau_u(c+a) .
\ee
\end{lemma}

\begin{proof}
The proof of the first assertion is straight-forward using \eqref{Comparison} above.
We therefore omit the proof and rather show the second assertion.

First, for $t\leq \bar t_{k-1}  + a- 2d_*$ (where $\bar t_{k-1}$ is with respect to $(c, u_c + 2a)$) we have
\bea
\nonumber
	U_{k-1}^{Q_{\sigma}}(c) 
	&=& \bigcap_{s_n \leq \bar t_{k-1} + (u_c+2a) - 2\bar d_*} \cN_{Q_{\sigma}}(R_n, s_n + c +  d_*)^C
	\subset \bigcap_{s_n\leq t + (u_c+a)} \cN_{B_{\sigma}}( R_n, s_n+  (c+ a)  )^C , 
\eea
where we used $\bar t_{k-1} + u_c + 2a - 2d_* \geq  t + u_c+a$ and \eqref{Comparison} using $a \geq d_*$.
Similarly, 
\be
\nonumber
	B_{\sigma}(x, t_0-d_*) \cap \bigcup_{s_n\leq t_0 + u_{c} +a} \cN_{B_{\sigma}}( R_n, s_n+  c +a  + d_*)  
	\subset Q_{\sigma}(x,  \bar t_k ) \cap \bigcup_{s_n\leq \bar t_k + u_c +2a - 2d_*} \cN_{Q_{\sigma}}( R_n, s_n+  c  + d_*),
\ee
for $t_0=\bar t_k+a- 2d_*$ since $t_0-d_* \geq \bar t_k$.
Hence, we obtain for any $(x, \bar t_k)$ with $x\in U^{Q_{\sigma}}_{k-1}(c)$
\bea
	\nonumber
	\mu \big(Q_{\sigma}(x,  \bar t_k ) \cap \bigcup_{s_n\leq \bar t_k + u_c + 2a - d_*} \cN_{Q_{\sigma}}( R_n, s_n+  c  + \bar d_*) \big) 
	&\geq& \tau_u(c+a) \cdot \mu( B_{\sigma}(x, t_0))
	\\ \nonumber
	&\geq& \tau_u(c+a) \cdot \mu( Q_{\sigma}(x, t_0 + \sqrt{n}/\sigma)) 
	\\ \nonumber
	&=& \tau_u(c+a) e^{-n\sigma( a + 2d_*+ \sqrt{n}/\sigma)}\cdot \mu( Q_{\sigma}(x, \bar t_k )),
\eea
 by \eqref{Bedingung}, 
proving the lemma.
\end{proof}

\subsubsection{A short discussion for the rectangle function}
As mentioned above, the improvement of the previous section relied on the partition of the cubes.
Since for the rectangle function this is only possible in special situations, we keep the following discussion short.

Consider the parameter space $(\Omega, R_{\bar \sigma}, \mu)$, where $\Omega=\R^n\times \R$ and $\mu$ is the Lebesgue measure on $\R^n$.
Recall that the 'rectangle' function $R_{\bar \sigma}$ on $\Omega $ for a vector $\bar \sigma=(\sigma_1, \ldots, \sigma_n) \in \R^n_{>0}$ is given by
\be
\nonumber
	R_{\sigma}(x,t) \equiv B(x_1,e^{-\sigma_1 t}) \times \cdots \times B(x_n,e^{-\sigma_1 t}).
\ee
Assume that each $\sigma_i \in \Q$. Write $\sigma_i= p_i/q$ with the same denominator $q \in \N$ for every $\sigma_i$.
Then similarly to the cube function, given $s=q \log(m)$ for $m\in \N$, 
any rectangle $R=R_{\bar \sigma}(x,t) \subset \R^n$ can be partitioned into $m^{p_1} \dots m^{p_n} = e^{(\sum_i \sigma_i) s}$ rectangles $R_i\equiv R_{\bar \sigma}(x_i,t + s)$ satisfying
$\cup R_i = R$ and $\mu(R_i \cap R_j)= 0$ for $i\neq j$.

Let $d_*\geq \log(2)/\min_i\{ \sigma_i\}$ so that $(\Omega, R_{\bar \sigma})$ is  $d_*$-contracting and $[d_*, \cF]$ is satisfied.
Modify \eqref{Decaying2} and \eqref{Dirichlet2} with respect to the rectangle function,
that is to the following conditions.
For all formal balls $\omega=(x, t_k) \in \Omega$ with $x\in L_{k-1}(c)$ we have
\be
\label{Decaying3}
	\mu(R_{\bar \sigma}(\omega) \cap  \cN(\cR( t_k - l_c), t_k + c -  2d_*) ) \leq  \tau_l(c) \cdot \mu(R_{\bar \sigma}(\omega)).
\ee
Moreover, for all formal balls $\omega= (x, \bar t_k ) \in \Omega$ with $x \in U_{k-1}(c)$ (defined as in \eqref{Dirichlet2} with respect to $R_{\bar \sigma}$)
we have
\be
\label{Dirichlet3}
	\mu(R_{\bar \sigma}(\omega) \cap \bigcup_{s_n\leq \bar t_k + u_c - 2d_*} \cN( R_n, s_n+  c  + d_*)) \geq  \tau_u(c) \cdot \mu(R_{\bar \sigma}(\omega)).
\ee

The above partition then strengthens Proposition \ref{PropositionMeasure} to the following.

\begin{proposition}
\label{PropositionRectangle}
Given $(\Omega, R_{\bar \sigma}, \mu)$ with $\bar \sigma \in \Q^n_{>0}$ and $\cal{F}$ as above satisfying $[d_*, \cF]$.
\begin{itemize}
\item[1.]
Let $c=q\log(m)$ for some $m \in \N$.
If \eqref{Decaying3} holds,
then $[\tau(c)]$ is satisfied with
\be
\nonumber
	\tau(c)= 1-\tau_l(c).
\ee
\item[2.]
Let $c+u_c=q\log(m)$ for some $m \in \N$.
If \eqref{Dirichlet3} holds,
then $[N(c)]$ is satisfied with
\be
\nonumber
	 N(c) \leq (1-\tau_u(c)) \cdot e^{(\sum_i \sigma_i)(c+u_c)} .
\ee
\end{itemize}
\end{proposition}

\begin{proof}
The proof follows from the proof of Proposition \ref{PropositionCube}
by replacing the cube function $Q_{\sigma}$ with the rectangle function $R_{\bar \sigma}$.
\end{proof}

\subsubsection{Dirichlet and decaying measures}
In the applications several situations appear frequently which can be translated into the abstract conditions \eqref{ContainedInSpheres} and \eqref{ContainsSpheres} below. 
We translate these conditions in turn to the ones required previously. 

Let $\cal{S} \equiv \{ S \subset \bar X \}$ be a given collection of nonempty Borel sets.
For instance, consider $\cal{S}$ to be the collection of metric spheres $S(x,t) \equiv \{y \in \bar X : d(x,y) =e^{-t}\}$ in $\bar X$,
or the set of affine hyperplanes in the Euclidean space $\R^n$.
Assume moreover, that  $\cN_{\bar \psi}(S, t)$ is a Borel-set for all $t>t_*$ and $S \in \cal{S}$.

For the lower bound, given a locally finite Borel measure $\mu$ on $X$, 
$(\Omega, \psi, \mu)$ is said to be \emph{absolutely $(c_{\delta}, \delta)$-decaying with respect to $\cal{S}$}
if for all $(x,t) \in \Omega$ and for all $S \in \cal{S}$ and $s\geq 0$ we have
\be
\label{SphereDecaying}
	\mu(\psi(x,t) \cap \cN(S, t + s)) \leq c_{\delta} e^{- \delta s} \mu(\psi(x,t)).
\ee

\begin{remark}
When $\cal{S}$ denotes the set of affine hyperplanes in the Euclidean space $\R^n$ and $\psi=B_1$ is the standard function,
then \eqref{SphereDecaying} corresponds to the original notion of absolutely decaying measures, see \cite{LindenstraussEtAl}.
\end{remark}

Moreover, we say that an increasing discrete family $\cal{F}$ is \emph{locally contained in $\cal{S}$} (with respect to $(\bar \Omega, \bar \psi)$)
if there exists $ l_*\geq 0$
and a number $n_*\in \N$
such that for all $(x,t) \in \Omega$ we have
\be
\label{ContainedInSpheres}
	\bar \psi(x, t+ l_* ) \cap \cR(t) \subset \bigcup_{i=1}^{n_*} S_i
\ee
is contained in at most $n_*$ sets $S_i$ of $\cal{S}$.

\begin{remark}
Notice that Condition \eqref{ContainedInSpheres} also makes sense for a given parameter $c>0$.
In this case, we consider $l_c$, $n_c$ and $\cR(t,c) = \cR(t) - \cR(t-c)$  depending on $c$ rather than $l_*$, $n_*$ and $\cR(t)$ respectively.
\end{remark}

We say that $(\Omega, \bar \psi)$ is \emph{$d_*$-separating} if for all formal balls $(x,t) \in \Omega$
and for any set $M$ disjoint to $\bar \psi(x,t)$, we have
\be
\label{Separating}
		\bar \psi(x,t + d_*) \cap \cN(M, t + d_*) = \emptyset.
\ee
Clearly, the standard function $B_{\sigma}$ is $\log(3)/\sigma$-separating in a proper metric space $\bar X$.

\begin{proposition} 
\label{SphereDecayingMeasure}
Let $(\Omega, \psi, \mu)$ be $d_*$-separating and let  $\cal{F}$ be locally contained in $\cal{S}$.
Then, if $(\Omega, \bar \psi)$  is absolutely $(c_{\delta}, \delta)$-decaying with respect to $\cal{S}$, it is  $\tau_l(c)$-decaying with respect to $\cal{F}$ and the parameters $(c, l_* + d_*)$ 
 for all $c\geq 2d_*$ such that $\tau_l(c) <1$,
where
\be
\nonumber
	\tau_l(c) = n_* c_{\delta} e^{-\delta(c-2d_*)}.
\ee
\end{proposition}

\begin{proof}
Fix $c\geq 2d_*$.
Given $\omega=(x,t + l_* +2 d_* ) \in \Omega$ and $l_*$, $n_*\in \N$ as well as $S_1, \dots, S_{n_*}$ from the definition of  \eqref{ContainedInSpheres},
we claim that 
\be
\nonumber
	\psi(\omega) \cap \cN( \cR(t), t  + l_* +d_*+  (c - d_*)) \subset \psi( \omega) \cap \bigcup_{i=1}^{n_*} \cN(S_i, t  + l_* + (c - d_*)).
\ee
To see this, let $M$ be the set  $R(t) -  \cup S_i$ which is disjoint to $\bar \psi(x, t + l_*)$ by   \eqref{ContainedInSpheres}.
By monotonicity of $\bar \psi$, we have 
\be
\nonumber
	\bar \psi(x, t+l_* +2 d_*) \subset \bar \psi(x, t+l_* +d_*)
\ee
which, by  \eqref{Separating}, is disjoint to 
\be
\nonumber
	\cN(M, t + l_* + d_*) \supset \cN(M, t+ l_* +c-d_*),
\ee
for $c\geq 2d_*$  again by monotonicity of $\bar \psi$.
This shows the above claim.

Set $l_c = l_* +d_*$ so that  $\omega = (x, t+ l_c + d_*) \in \Omega$.
Finally, the claim and \eqref{SphereDecaying} imply
\bea
\nonumber
	\mu(  \psi( \omega) \cap \cN( \cR(t), t  + l_c+  (c - d_*))
	&\leq& \mu(  \psi( \omega) \cap \bigcup_{i=1}^{n_*} \cN(S_i, t  + l_c + d_* + (c -2 d_*) ) )
	\\ \nonumber	
	&\leq& n_* c_{\delta} e^{-\delta(c-2d_*)} \mu(  \psi( \omega)),
\eea
which shows that $\mu$ is $\tau_l(c) $-decaying with respect to $\cal{F}$ and the parameters $(c, l_* + d_*)$. 
\end{proof}

As a special case, let $\bar \psi = \bar B_{\sigma}$ be the standard function and $\bar X$ be a proper metric space.
Recall that $d_* \leq \log(3)/\sigma$,
and assume that for all distinct points $x$, $ y \in R_n$ we have
\be
\label{Distinct}
	d(x,y) > \bar c \cdot e^{- \sigma s_n},
\ee
for some constant $\bar c>0$.

\begin{lemma} 
\label{DistinctLemma}
Let $(\Omega, \psi, \mu)$ satisfy a power law with respect to  the parameters $(\tau, c_{1}, c_{2})$.
If \eqref{Distinct} is satisfied, 
then $\mu$ is  $\tau_l(c)$-decaying with respect to $\cal{F}$, where $\tau_l(c)= \tfrac{c_{2}}{c_{1}} e^{\tau(c - 2d_*)} $,
 for all $c \geq 2d_*$ and $l_c=  -\log(\bar c)/\sigma + d_* + \log(2)$.
\end{lemma}

\begin{proof}
Let $l_* =  - \log(\bar c)/\sigma + \log(2)$.
Given a formal ball $(x, t + l_* ) \in \Omega$, at most one point $y \in R(t)$ can lie in $B(x, e^{-\sigma (t+l_*)})$.
In fact, for distinct $y$ and $y' \in R_{n_t}$ (where $n_t \in \N$ was the largest integer such that $s_n \leq t$),
\eqref{Distinct} implies 
\be
\nonumber
	d(y,y') > e^{-\sigma (s_n + \log(\bar c)/\sigma )} \geq 2 e^{- \sigma(t + l_*)}.
\ee
Hence, $\cal{F}$ is locally contained in the set $\cal{S} \equiv \{y \in R_n : n \in \N\}$ with $n_*=1$.
Since $\mu$ satisfies the power law, it is $(\tfrac{c_{2}}{c_{1}}, \tau)$-decaying with respect to $\cal{S}$ and $\bar B_{\sigma}$.
The proof follows from Proposition \ref{SphereDecayingMeasure}.
\end{proof}


Analogously, for the upper bound and a possibly different collection of Borel sets $\cal{S}$, 
for a locally finite Borel measure $\mu$ on $X$, $(\Omega, \psi, \mu)$ is called \emph{$(c_{\delta}, \delta)$-Dirichlet with respect to $\cal{S}$}
if for all $\omega=(x,t ) \in \Omega$, for all $S \in \cal{S}$ such that $S \cap \bar \psi(\omega) \neq \emptyset$ and $s\geq 0$ we have
\be
\label{SphereIncreasing}
	\mu(\psi(\omega) \cap \cN(S, t + s)) \geq c_{\delta} e^{- \delta s} \mu(\psi(\omega)).
\ee

We say that the family $\cal{F}$  \emph{locally contains $\cal{S}$} (with respect to $(\Omega, \psi)$) 
if there exists $u_* \geq 0$  such that
for all formal balls $\omega=(x, t - u_*) \in \Omega$ there exists $S \in \cal{S}$ with
\be
\label{ContainsSpheres}
	 \bar \psi(\omega) \cap  S \subset   \cR(t) .
\ee

\begin{proposition} 
\label{DirichletProposition}
If $\cal{F}$ locally contains $\cal{S}$ and $(\Omega, \psi,\mu)$ is $(c_{\delta}, \delta)$-Dirichlet with respect to $\cal{S}$,
then $(\Omega, \psi, \mu)$ is $\tau_u(c)$-Dirichlet with respect to $\cal{F}$ and the parameters $(c, u_*)$, where $\tau_u(c)\geq c_{\delta} e^{-\delta( c + d_*)}$.

In the special case when $\cal{F}$ locally contains $\cal{S}$, where $\cal{S}$ consists of subsets of $X$, 
and $(\Omega, \psi,\mu)$ satisfies a power law with respect to the parameters $(\tau, c_{1}, c_{2})$,
we have that $(\Omega, \psi, \mu)$ is $\tau_u(c)$-Dirichlet with respect to $\cal{F}$ and the parameters $(c, u_* )$, where 
\be
\nonumber
	\tau_u(c)\geq \tfrac{c_{1}}{c_{2}} e^{-\tau(c+2d_* + u_*)}.
\ee 
\end{proposition}

\begin{proof}
The first statement is readily checked.
For the second one, let $\omega= (x, t- u_*)\in \Omega$ and $S \in \cal{S}$ such that $S\cap \bar \psi(\omega) \subset R_{n_t} \cap \bar \psi(\omega)$.
Let $y \in S\cap \bar \psi(\omega)$.
By monotonicity of $\bar \psi$ and  \eqref{Separating1},
$\bar \psi( y, t + c +d_*)\subset \bar \psi( y, t - u_* + c ) \subset  \bar \psi(x , t - u_* -  d_*)$. 
Hence, for $\omega_0 = (x,t  - u_* -d_*) \in \Omega$  we see that
\bea
\nonumber
	\mu( \psi(\omega_0) \cap \bigcup_{s_n \leq t} \cN(R_n, s_n + c + d_*)) 
	&\geq& \mu(\psi(\omega_0) \cap \bar \psi( y, t + c + d_* ) ) 
	\\ \nonumber
	&\geq& \mu( \psi( y,t + c +d_* ) ) 
	\geq   \tfrac{c_{1}}{c_{2}} e^{-\tau(c+2d_* + u_*)} \mu(\psi(\omega_0)),
\eea
which shows the second claim.
\end{proof}


\subsection{Proof of Theorem \ref{ThmBounds}.}
\label{SectionProofOfThmBounds}
We now prove the main theorem of this section.
For the lower bound of dim(\textbf{Bad}$(\cal{F}, 2c + l_c))$, using $[\tau(c)]$,
we inductively construct a strongly treelike family of sets such that its limit set, $A_{\infty}$,
is a subset of \textbf{Bad}$(\cal{F}, 2c + l_c)$.
Using the method of \cite{KWBAVectors,KleinbockWeiss} (which is a generalization of the ones of \cite{McMullen2, Urbanski}),
based on the 'Mass Distribution Principle', we derive a lower bound of dim$(A_{\infty})$.

For the upper bound of dim(\textbf{Bad}$(\cal{F}, c ) \cap \psi(\omega_0))$, 
we construct a sequence $\cU_k$ of covers of \textbf{Bad}$(\cal{F}, c ) \cap \psi(\omega_0)$ 
with uniform bounds on the diameters converging to zero.
The idea is to inductively use Condition $[N(c)]$ in order to 'refine' a given cover $\cU_k$ to a cover $\cU_{k+1}$ by smaller $\psi$-balls and to given an upper bound on the cardinality of the new cover.

\subsubsection{Proof of the lower bound [LB]}
\label{LB}
Recall that for $c \geq d_*$ and  $l_c\geq 0$, 
for $k\geq 0$, we defined $t_k\equiv s_1 + kc + l_c$ and  
\be
\nonumber
	L_{k}(c) = L^{\bar \psi}_k(c) \equiv   \bigcap_{i=1}^k \cN(R(t_i - l_c), t_i + c )^C .
\ee

Let  $\omega_0 = (x_0, t_0) \in \Omega$ be a formal ball and set $L_{-1}(c) = \bar X$ (and $i_0=0$).
Construct a strongly treelike family $\cal{A}$ of subsets of $X\cap \psi(\omega_0)$ relative to $\mu$ as follows.
Let $\cal{A}_0 = \{ \psi(\omega_0) \}$.
Assume we are given the subfamily $\cal{A}_k$ at the $k$.th step and a set $\psi(\omega_{i_0 \dots i_k}) \in \cal{A}_k$, where $\omega_{i_0 \dots i_k} = (x_{i_0 \dots i_k}, t_k) \in \Omega$ with $x_{i_0 \dots i_k} \in L_{k-1}(c)$.
Then Condition $[\tau(c)]$ provides a collection $\cC_{l,c}( \omega_{i_0 \dots i_k} )$ of formal balls $\omega_{i_0 \dots i_k i_{k+1}} = (x_{i_0 \dots i_k i_{k+1}}, t_{k+1}) \in \Omega$ with 
$x_{i_0 \dots i_k i_{k+1}} \in L_{k}(c)$ such that the sets
 $\psi(\omega_{i_0\dots i_k i_{k+1}}) \subset \psi(\omega_{i_0 \dots i_k})$
 are  essentially-disjoint (relative to $\mu$) and disjoint to $\cN( R(t_k - l_c), t_k + c)$ and satisfy \eqref{Schranke}.
We therefore define  
\be
\nonumber
	\cal{A}_{k+1} = \cup_{i_0 \dots i_k} \{ \psi(\omega_{i_0 \dots i_k i_{k+1}}) : \omega_{i_0 \dots i_k i_{k+1}} \in\cC_{l,c}( \omega_{i_0 \dots i_k} ) \} ,
\ee
where the indices $i_0 \ldots i_{k}$ run over all indices  from the previous construction.

If $\cal{A}$ (a countable family of compact subsets of $X$) denotes the union of these sub collections $\cal{A}_k$, $k\in \N$, the following properties are satisfied
with respect to $\mu$:
\begin{itemize}
\item[(TL0)] $\mu(A) > 0$ for all $A\in \cal{A}$ (by $[\mu>0]$),
\item[(TL1)] for all $k\in \N$, for all $A, B\in \cal{A}_n$, either $A=B$ or $\mu(A \cap B) = 0$,
\item[(TL2)] for all $k\in \N_{\geq 2}$, for all $B\in \cal{A}_k$, there exists $A\in \cal{A}_{k-1}$ such that $B \subset A$,
\item[(TL3)] for all $k\in \N$, for all $A\in \cal{A}_k$, there exists $B\in \cal{A}_{k+1}$ such that $B \subset A$.
\end{itemize}
We can therefore define $\cup \cal{A}_k = \cup_{A \in \cal{A}_k} A$ 
and obtain a decreasing sequence of nonempty compact subsets 
$X \supset \cup \cal{A}_1 \supset \cup \cal{A}_2 \supset \cup \cal{A}_3 \supset \dots$.
Since $X$ is complete, the limit set
\be
\nonumber
	A_{\infty} \equiv \bigcap_{k\in \N} \cup \cal{A}_k
\ee
is nonempty. 
Define moreover the $k$.th stage diameter $d_k(\cal{A}) \equiv \max_{A\in \cal{A}_k} \text{diam}(A)$, 
which by $[\sigma]$ satisfies $d_k(\cal{A})\leq c_{\sigma}e^{-\sigma t_k}$, and hence
\begin{itemize}
\item[(STL)] $ \lim_{k \to \infty} d_k(\cal{A})  = 0$.
\end{itemize}
Finally, by \eqref{Schranke}, we obtain a lower bound for the $k$.th stage 'density of children'
\be
\label{Delta}
	\Delta_k(\cal{A}) \equiv \min_{B\in \cal{A}_k} \frac{\mu(\cup{A}_{k+1} \cap B) }{\mu(B)} 
	\geq \tau(c)
\ee
of $\cal{A}$.
This gives a lower bound on the Hausdorff dimension of $A_{\infty}$.

\begin{lemma}
\label{LowerBoundLemma}
If $\cal{A}$ as above satisfies (TL0-3) and (STL), then
\be
\nonumber
 	\text{dim}(A_{\infty}) \geq 
			 \inf_{x_0 \in A_{\infty}}d_{\mu}(x_0) - \frac{| \log( \tau(c) )| } {\sigma c}.
\ee
\end{lemma}

\begin{proof}
In \cite{KWBAVectors}, Lemma 2.5 (which is stated for $\bar X=\R^n$ but also true for general complete metric spaces, see \cite{KleinbockWeiss})
a measure $\nu$  is constructed for which the support equals  $A_{\infty}$.
Moreover, $\nu$ satisfies for every $x \in A_{\infty}$ that
\bea
\nonumber
	d_{\nu}(x) &\geq& d_{\mu}(x) -  \limsup_{k\to \infty} \frac{\sum_{i=1}^k  \log( \Delta_i(\cal{A})) }{ \log(d_{k}(\cal{A})) }
	\\ \nonumber
	&\geq& \inf_{x_0 \in A_{\infty}}d_{\mu}(x_0) - \limsup_{k \to \infty} 
		\frac{ k \log(\tau(c)) } 		
				{\log( c_{\sigma} e^{- \sigma t_k})}
	\geq \inf_{x_0 \in A_{\infty}}d_{\mu}(x_0) - \frac{\lvert  \log(\tau(c)) \rvert} {\sigma c},
\eea
where we used \eqref{Diam} and \eqref{Delta}.
For every open set $U \subset  \bar X$ with $\nu(U)>0$, let
\be
\nonumber
	d_{\nu}(U) \equiv \inf_{x \in U \cap \text{ supp}(\mu)} d_{\nu}(x),
\ee
which is known to be a lower bound for the Hausdorff dimension of supp$(\nu) \cap U = A_{\infty} \cap U$ (see \cite{Falconer}, Proposition 4.9 (a)).
Setting $U=\bar X$ shows the claim.
\end{proof}

We establish our lower bound by showing the following Lemma.

\begin{lemma} 
\label{NotIn}
The limit set $A_{\infty} \subset \psi(\omega_0) \cap \emph{\textbf{Bad}}(\cal{F}, 2c  + l_c)$; hence, \emph{dim(\textbf{Bad}}$(\cal{F}, 2c + l_c)) \geq \emph{dim}(A_{\infty})$.
\end{lemma}

\begin{proof}
Let $x_\infty \in A_{\infty}$.
Let $\{i_0\dots i_k\}_{k\in \N}$ be a sequence such that $x_0 \in  \cap_{k\in \N} \psi(\omega_{i_0 \dots i_k})$.
Note that since the sets $\psi(\omega_{i_0 \dots i_k})$ are only essentially-disjoint relative to $\mu$, 
the sequence $\{i_0\dots i_k\}_{k\in \N}$ might not be unique. 

Assume that $x_0 \in \cN(R_m, s_m)$ for some $m \in \N$ 
(if no such $m$ exists, then the claim already follows).
Choose the integer  $k \geq 0$ such that $s_m + l_c \in [t_k, t_{k+1})$.
By construction, $x_\infty \in \psi(\omega_{i_0 \dots i_{k+2}})$ which is disjoint to $\cN(\cR(t_{k+1} - l_c), t_{k+1} + c)$ by \eqref{Disjoint}.
Since $t_k - l_c \leq s_m < t_{k+1} - l_c$ we have $R_m \subset \cR(t_{k+1} - l_c)$ and 
\bea
\nonumber
	x_\infty &\notin& \cN(\cR(t_{k+1}- l_c), t_{k+1} + c) 	
				\\\nonumber
				&=& \cN(\cR(t_{k+1} - l_c), t_{k} - l_c  + 2c + l_c)
				\supset \cN(R_m, s_m + (2c + l_c)),
\eea
by monotonicity of $\bar \psi$. This shows that $x_\infty \in  \textbf{Bad}(\cal{F}, 2c + l_c)$.
\end{proof}


\subsubsection{Proof of the upper bound [UB]}
\label{UB}
Recall that for $c>0$ and $u_c \geq 0$, we defined $\bar t_k = s_1 + k(c+u) - u $ for $k\geq0$
and
\bea
\nonumber
	U_k(c)  &\equiv&  \bigcap_{s_n \leq \bar t_k + u} \cN(R_n, s_n + c )^C .
\eea

Let $\omega_0 =(x_0, \bar t_0) \in \Omega$ be any formal ball.
Let $\cU_0= \{\psi(\omega_0)\}$ and set $U_{-1}(c)=X$ (and $i_0=0$).
Suppose we have already constructed the $k$.th step cover $\cU_k= \{\psi(\omega_{i_0\dots i_k})\}$ where $\omega_{i_0\dots i_k} = (x_{i_0\dots i_k}, \bar t_k)$ with $x_{i_0\dots i_k} \in U_{k-1}(c)$.
Condition $[N(c)]$ gives a collection $\cC_{u,c}(\omega_{i_0\dots i_k})$ of formal balls 
$\omega_{i_0 \dots i_k i_{k+1}}  =(x_{i_0 \dots i_k i_{k+1}}, \bar t_{k+1} ) \in \Omega$ with $x_{i_0 \dots i_k i_{k+1}} \in U_{k}(c)$ 
such that the sets $\psi(\omega_{i_0 \dots i_k i_{k+1}} )$ cover the set
\be
\label{LocalCover}
	U_k(c)\cap  \psi(\omega_{i_0\dots i_k})  = \psi(\omega_{i_0\dots i_k}) -  \bigcup_{s_n \leq \bar t_k + u} \cN(R_n, s_n + c ).
\ee
and the number of these sets is bounded above by $N(c)$.
Set
\be
\label{FinerCovering}
	\cU_{k+1} \equiv \cup_{i_0 \dots i_k}\{ \psi(\omega_{i_0 \dots i_k i_{k+1}} ) : \omega_{i_0 \dots i_k i_{k+1}} \in \cC_{u,c}( \omega_{i_0\dots i_k}) \}
\ee
where the indices $i_0, \dots, i_k$ run over all indices from the previous construction.

\begin{lemma}
\label{ConstructCover}
The collection $\cU_k$ is a cover of $\textbf{Bad}(\cal{F}, c ) \cap \psi(\omega_0)$ by sets of diameter at most $e^{- \sigma \bar t_k}$ with
\be
\nonumber
	\lvert \cU_k \rvert \leq  N(c)^{k-1}.
\ee
\end{lemma}

\begin{proof}
First note that for every $k \in \N_0$,
\bea
\nonumber
	U_k(c)  &=&  \bigcap_{s_n \leq \bar t_k + u} \cN(R_n, s_n + c )^C 
		\supset \bigcap_{n \in \N} \cN(R_n, s_n+ c )^C 
		= \textbf{Bad}(\cal{F}, c ) .
\eea
Since $ \textbf{Bad}(\cal{F}, c ) \subset U_k(c) $ and by \eqref{LocalCover}, we have  
\bea
\nonumber
	\textbf{Bad}(\cal{F}, c ) \cap ( \cup_{U \in \cU_k} U) &\subset& \textbf{Bad}(\cal{F}, c ) \cap U_k(c) \cap ( \cup_{U \in \cU_k} U) 
	 \\ \nonumber
	&=&\textbf{Bad}(\cal{F}, c )  \cap \cup_{U \in \cU_k} ( U_k(c) \cap U) 
	\subset \textbf{Bad}(\cal{F}, c )  \cap (\cup_{U \in \cU_{k+1}}   U)
\eea
This shows that $\cU_k$ is a cover of $\textbf{Bad}(\cal{F}, c ) \cap \psi(\omega_0)$, since
\bea
\nonumber
	\textbf{Bad}(\cal{F}, c ) \cap \psi(\omega_0) 
	&\subset & \textbf{Bad}(\cal{F}, c ) \cap ( \cup_{U \in \cU_1} U) \subset \ldots \subset  \textbf{Bad}(\cal{F}, c ) \cap ( \cup_{U \in \cU_k} U) \subset \bigcup_{U \in \cU_k}U.
\eea
Moreover, it is finite  and bounded by $N_k \equiv \lvert \cU_{k} \rvert $ with
\be
\nonumber
	N_k  \leq N(c) N_{k-1}\leq\ldots \leq N(c)^{k-1} N_0 = N(c)^{k-1} 
\ee
and the sets of $\cU_k$ are of diameter at most $c_{\sigma} e^{-\sigma \bar t_{k}}$ by $[\sigma]$.
This finishes the proof.
\end{proof}

Finally, if $\cH^s$ denotes the $s$-dimensional Hausdorff-measure on $X$, then
\be
\nonumber
	\cH^s( \textbf{Bad}(\cal{F}, c ) \cap \psi(\omega_0) ) \leq \lim_{\delta \to 0}\cH_{\delta}^s( \textbf{Bad}(\cal{F}, c ) \cap \psi(\omega_0) ) 
	\leq \liminf_{k \to \infty} \  \lvert \cU_{k} \rvert \cdot (c_{\sigma} e^{-\sigma \bar t_{k}})^s,
\ee
which is finite whenever $s> \liminf_{k\to \infty}  \tfrac{- \log(\lvert \cU_{k} \rvert )}{\log(c_{\sigma} e^{-\sigma \bar t_{k}})}$ and in particular for
\be
\label{Schranke3}
	s> \frac{\log(N(c))}{\sigma ( c + u_c)} \geq \liminf_{k \to \infty} \frac{-\log( \lvert \cU_{k} \rvert )}{\log(c_{\sigma} e^{-\sigma \bar t_{k}})} .
\ee
This establishes the upper bound of the theorem, 
\be
\nonumber
	\text{dim}(\textbf{Bad}(\cal{F}, c ) \cap \psi(\omega_0)) \leq
	  \frac{\log(N(c))}{\sigma ( c + u_c)}.
\ee


\subsubsection{A final remark on the upper bound.}
Let $X=\R^n$ in the following.
In the case that $\bar \psi$ is not the standard function $B_{\sigma}$, but for instance the rectangle function $R_{\bar \sigma}$, we may give a different upper bound in the following situation.

\begin{lemma}
\label{UBFormula2}
Assume that conditions $[\sigma]$ and $[N(c)]$ hold.
If there is $\theta>0$ and $\hat \sigma \geq \sigma$ such that any $\psi$-ball $\psi(x,t)$ can be covered by $N_{\psi}(t) \leq 2 e^{\theta t}$ cubes $Q_{\hat \sigma}(x_i, t)$,
then
\be
\nonumber
	\emph{dim}(\textbf{Bad}(\cal{F}, c ) \cap \psi(\omega_0)) \leq \frac{\theta}{\hat \sigma } +   \frac{\log(  N(c) )}{\hat \sigma ( c + u_c)}.
\ee
\end{lemma}

\begin{proof}
Let $\psi(\omega) \in \cU_k$, $\omega=(x, \bar t_k)$, be any $\psi$-ball from the cover $\cU_k$ constructed in \eqref{FinerCovering}.
Cover $\psi(\omega)$ by $N_{\psi}(t) $ cubes $Q_{\bar \sigma}(x_i, t)$
and denote the collection of these cubes by $\bar \cU_k$.
The collection $\bar \cU_k$ by sets of diameter $e^{- \bar \sigma \bar t_k}$ and of cardinality at most $N_{\psi}(\bar t_k) \cdot \lvert \cU_k \rvert$
is a cover of $\textbf{Bad}(\cal{F}, c )$.
The proof follows by the arguments of the previous section.
\end{proof}

For later purpose, we state the following elementary Lemma. Its proof is a simple volume argument and omitted.

\begin{lemma} 
\label{RefineCover}
Let $\hat \sigma \equiv \max_i\{\sigma_i\}$, where $\bar \sigma\in\R^n_{>0}$.
Then any rectangle $R_{\bar \sigma}(x, t)$ can be covered by $N_{\bar \sigma}(t)$ cubes $Q_{\hat \sigma}(x_i,t)$, where
\be
\nonumber
	N_{\bar \sigma}(t) \leq 2e^{(n \hat \sigma - \sum_i \sigma_i) t}.
\ee
\end{lemma}


\section{Applications}
\label{Apps}

\noindent 
Verifying the conditions of the axiomatic approach we determine upper and lower bounds on the Hausdorff dimension of $\textbf{Bad}(\cal{F}, c)$ of several examples.

\subsection{\textbf{Bad}$_{\R^n}^{\bar r}$}
\label{BadRn}
For $n\geq 1$, let $\bar r \in \R^n$ be a given weight vector with $r_1,\dots, r_n > 0$ satisfying $\sum r_i =1$.
We adjust the usual definition (which in particular differs from the one in the introduction)
 of badly approximable vectors  $\textbf{Bad}_{\R^n}^{\bar r}$ with weight $\bar r$ to a more convenient one for our situation. 
A vector $\bar x = (x_1,\dots ,x_n)\in \R^n$ belongs to $\textbf{Bad}_{\R^n}^{\bar r}$ if 
there exists a positive constant $c(\bar x)>0$ such that
for every $q\in \N$ and $\bar p = (p_1,\dots, p_n) \in \Z^n$
there is some $i\in \{1, \dots, n\}$ with
\be
\nonumber
\label{Bedinung}
 	\lvert x_i -\frac{ p_i}{q} \rvert  \geq \frac{c(\bar x)^{1+r_i}}{q^{1+r^{i}}}.
\ee
Note that for $\bar r= (\tfrac{1}{n}, \dots, \tfrac{1}{n})$ we write $\textbf{Bad}_{\R^n}^{n} \equiv \textbf{Bad}_{\R^n}^{\bar r}$ which agrees with the set of badly approximable vectors.
For $c>0$  let $\textbf{Bad}_{\R^n}^{\bar r}(e^{-c})$ be the subset of vectors $\bar x \in \textbf{Bad}_{\R^n}^{\bar r}$ with $c(\bar x) \geq e^{-c}$.

Define $\bar \sigma = (1+r_1, \ldots, 1+r_n)$.
As in \cite{KleinbockWeiss, Weil2}, we let let $\bar \Omega= \R^n \times \R$ 
and consider the rectangle function $\bar R_{\bar  \sigma} : \bar \Omega \to \cal{C}(\R^n)$, given by the rectangle determined by $\bar r$, that is, the product of metric balls
\bea
\nonumber
	\bar R_{\bar \sigma}( \bar x, t  ) \equiv B(x_1, e^{-(1+r_1)t}) \times \dots \times B(x_n, e^{-(1+r_n)t}) .
\eea
Denote by $r_+ \equiv \max\{r_i\}$ and by $r_- \equiv \min\{r_i\}>0$.
Clearly, we have diam$(\bar R_{\bar \sigma}(x,t)) \leq 2 e^{-(1+r_-)t}$, hence $\sigma=1+r_-$ in $[\sigma]$.

In the following, let $\cal{S}$ be the set of affine hyperplanes in $\R^n$.
Note that, if  $\mu$ denotes the Lebesgue-measure on $\R^n$,
it follows from \cite{LindenstraussEtAl} Lemma 9.1 (see also \cite{Weil2} for the case of rectangles), 
that $(\Omega, \bar R_{\bar \sigma}, \mu)$ is absolutely $(\delta, c_n)$-decaying with respect to $\cal{S}$  
for $\delta = 1+ \min \{r_1,\dots, r_n\}= 1 +r_-$  and some $c_n>0$.
Moreover, $(\Omega, \bar R_{\bar \sigma},\mu)$ satisfies a $(n+1)$-power law.
More precisely, for all $(x, t)\in \Omega$ we have
$\mu(\bar R_{\bar \sigma}(x,t)) = 2^n e^{- (n+1)t}$.
Define $d_*\equiv \tfrac{\log(3)}{1 +r_-} $ and for $c>0$, 
\bea
\nonumber
	l_* = \tfrac{\log(n!)}{n+1} + \tfrac{n}{n+1} \log(2) +  \tfrac{\log(3)}{1 + r_-}, \ \ \ \
	u_c = \tfrac{1}{r_- } \big(c +(1+r_+)2  \log(2) \big)  \equiv \tfrac{c}{r_-} + u_*.
\eea

\begin{theorem}
\label{BadN}
Let $X \subset \R^n$ be the support of a locally finite Borel measure $\mu$ on $\R^n$ such that  
$(\Omega, \bar R_{\bar \sigma}, \mu)$ satisfies a $(\tau, c_1, c_2)$-power law.
When $(\Omega, \bar R_{\bar \sigma}, \mu)$ is absolutely $(\delta, c_{\delta})$-decaying with respect to $\cal{S}$,
then for $c> \log(c_{\delta})/\delta + 2d_*$ we have
\bea
\nonumber
	\emph{dim} (\textbf{Bad}_{\R^n}^{\bar r }( \tfrac{1}{2} e^{-(2c +l_*)} ) \cap X) 
	\geq d_{\mu}(X)  - \frac{  \log( 2) +  2\log( \tfrac{c_2}{c_1})  +2 \tau d_* + \lvert \log(1- c_{\delta} e^{2\delta d_*} e^{-\delta c})  \rvert  }
	{ (1+r_-) c }.
\eea
Moreover, when $\bar r \in \Q^n_{>0}$, there exist $c_0>0$ depending only on $n$ and $\bar r$ such that for $c>c_0$ we have
\bea
\nonumber
	\emph{dim} (\textbf{Bad}_{\R^n}^n(\tfrac{1}{2} e^{-(2c +  l_* + d_*)} ))
	&\geq &
	 n  - \frac{ \lvert   \log(1- 18c_n e^{- (1+r_-)c }) \rvert  }{  (1+r_-)c },
\eea
\bea
\nonumber
	\emph{dim}(\textbf{Bad}_{\R^n}^n(e^{-c} ))
	& \leq &  n - \frac{ \lvert   \log(1-  \tfrac{1}{8}3^{-(n+1)(1+r_-) }e^{-(n+1)(c+u_c )} )\rvert}{(1+r_+)(c+u_c + 2d_*)}.
\eea
\end{theorem}

Theorem \ref{ThmA} follows by applying the Taylor expansion and pointing out the modification in the definition from the introduction (which results in a transformation of the exponents).

\begin{remark}
As pointed out by the referee, given a fractal $F \subset \R^n$, non-trivial upper bounds for  $\text{dim}(\textbf{Bad}_{\R^n}^n(e^{-c} )\cap F)$ of the intersection 
(that is dimensions strictly less than $\text{dim}(F)$) would imply that the Hausdorff-measure on $F$ of $\textbf{Bad}_{\R^n}^n(e^{-c} )\cap F$ is null.
This in turn would extend a recent result of Einsiedler, Fishman, Shapira \cite{EinsiedlerEtAl}. 
Our results, unfortunately, give no answers to this question. 
However, recall Section \ref{SectionQube} for a remark on possible extensions to fractals such as the Sierpinksi carpet or regular Cantor sets.
\end{remark}

\subsubsection{Proof of Theorem \ref{BadN}}
For $k\in \N$ we define the set of rational vectors 
\be
\nonumber
	R_{k} \equiv \{ \bar p/q \in \Q^n: \bar p \in \Z^n, 0<q \leq k \}
\ee
as resonant set and define its size by $s_{k} \equiv \log(k+1)$. 
The family $\cal{F}=(R_k, s_k)$ is increasing and discrete.
Moreover, since $\cR(t)$ is a discrete set for all $t\geq s_1$
 and since $\bar R_{\bar \sigma}$ is a product of metric balls, 
it is readily checked that $(\Omega, \bar R_{\bar \sigma})$ is $d_*$-separating with respect to $\cal{F}$.
Also $(\Omega, \bar R_{\bar \sigma})$ is $d_*$-separating and $d_*$-contracting. 

For the lower bound, we shall show the following.

\begin{proposition}
\label{BadNDecaying}
Let $c> \log(c_{\delta})/\delta + 3d_*$.
Then  $(\Omega, \bar R_{\bar \sigma}, \mu)$ is $\tau_l(c)$-decaying with respect to $\cal{F}$  and the parameters $(c, l_*)$,
 where $\tau_l(c) =  c_{\delta} e^{-\delta(c-2d_*)}$.
 Moreover, when $\mu$ denotes the Lebesgue measure, 
 then $(\Omega, \bar R_{\bar \sigma}, \mu)$ satisfies \eqref{Decaying3} for the parameters $(c, l_*+d_*)$ with $\tau_l(c) =  c_{n} e^{(1+r_-)d_*}e^{-(1+r_-)c}$.
\end{proposition}

\begin{proof}
Choose any $\bar l_* >  \log(n!)/(n+1) + n/(n+1) \log(2)  $.
Note that for a formal ball $\omega=(\bar x, t + \bar l_* )$ such that $ t \leq s_{k}$ the sidelights $\rho_i$ of the box
$\bar R_{\bar \sigma}(\omega)$ satisfy
\bea
\nonumber
	\rho_1 \cdots \rho_n &=& 2 e^{-(1+r^1)(t + l_*)} \cdots 2 e^{-(1+r^n)(t+l_*)}
	\\
	\nonumber
	 &<& 2^n e^{-(1+n)s_k - \log(n!) - n \log(2) }  \leq \tfrac{1}{ n! (k+1)^{n+1}}.
\eea
We now use the following version of the 'Simplex Lemma' due to Davenport and Schmidt where
the version of this lemma can be found in \cite{KristensenEtAl}, Lemma 4.

\begin{lemma}
Let $D\subset \R^n$ be a box of side lengths $\rho_1, \dots, \rho_n$ such that  $\rho_1 \dots \rho_n < 1/(n! (k+1)^{n+1} )$.
Then there exists an affine hyperplane $L$ such that $R_{k} \cap D \subset L$.
\end{lemma}

\noindent This shows that $\cal{F}$ is locally contained in the collection of affine hyperplanes $\cal{S}$ with $n_*=1$.
Since $(\Omega,\bar R_{\bar \sigma},\mu)$ is absolutely $(\delta, c_{\delta})$-decaying with respect to $\cal{S}$, 
if follows from Proposition \ref{SphereDecayingMeasure} that
 $(\Omega, \bar R_{\bar \sigma}, \mu)$ is $\tau_l(c)$-decaying with respect to $\cal{F}$  for all $c> 2d_*$ where $l_c \equiv l_* = \bar l_* + d_*$
 and $\tau_l(c) =  c_{\delta} e^{-\delta(c-2d_*)}$ with $\tau_l(c) <1$.
Remarking that   $\tau_l(c)<1$ for $c> \log(c_{\delta})/\delta + 2d_*$ shows the first claim.

Let now $\mu$ denote the Lebesgue measure which is absolutely $(1+r_-, c_n)$-decaying with respect to $\cal{S}$.  
Set $\bar l_c = \bar l_* + 2d_*$. 
Then by the above, given any  
$\omega=(x, t + \bar l_c) \in \Omega$, it is readily shown that for $c> 3d_* + \log(c_n)/(1+r_-)$,
\be
\nonumber
	\mu(\bar R_{\bar \sigma}(\omega) \cap  \cN(\cR( t), t + \bar l_c+ c -  2d_*) ) \leq c_n e^{- (1+r_-)(c-2d_*) } \cdot \mu(\bar R_{\bar \sigma}(\omega)),
\ee
which shows the second claim and  finishes the proof.
\end{proof}

For the upper bound, we need the following.

\begin{proposition}
\label{BadNDirichlet}
Let  $\mu$ be the Lebesgue measure on $\R^n$.
 Then $(\R^n \times \R, \bar R_{\bar \sigma}, \mu)$ is $\tau_u(c)$-Dirichlet with respect to $\cal{F}$ 
 and the parameters $(c, u_c)$, where $\tau_u(c)= \tfrac{c_{1}}{c_{2}} e^{-2 \tau d_*}\cdot e^{- \tau (c + u_c)}$.
Moreover, $(\R^n \times \R, \bar R_{\bar \sigma}, \mu)$ satisfies \eqref{Dirichlet3} for the parameters $(c, u_c+2d_*)$
with $\tau_u(c) = \tfrac{1}{4} e^{-(n+1)d_*}e^{-(n+1)(c+u_c )} $.
\end{proposition}

\begin{proof}
Note that using the pigeon-hole principle as for the classical Dirichlet Theorem,
the following Lemma can be shown; its proof is readily checked.

\begin{lemma}
Let $x\in\R^n$. For every $N\in\N$ there exists a vector $(p_1, \dots, p_n) \in \Z^n$ and $1\leq q \leq N$ such that,
for $i=1, \dots, n$, we have
\be	
\nonumber
	\lvert x_i - \frac{p_i}{q} \rvert \leq \frac{1}{qN^{r_i}}.
\ee
\end{lemma}
\noindent Define $u_c = c/r_- + u_*$ as given above and let $\bar x\in X = \R^n$.
Given $t>0$, let  $N \in \N$ be the maximal integer such that $\log(N) \leq t + u_c$, hence $t + u_c \leq  \log(N) +\log(2)$.
Let $\bar p$, $q$ as in the above lemma, given $N$. 
In the case when $\log(q) \leq  t + u_c- (c + u_c)$, we have 
\bea
\nonumber
     			\lvert x_i - \frac{p_i}{q} \rvert 
			&\leq& \frac{1}{qN^{r_i}}  
			\\ \nonumber
		   	&\leq & e^{-\log(q) - r_i (t + u_c  - \log(2))} 
			\\ \nonumber
		   	&\leq& e^{-(1+r_i)\log(q) - r_i (c + u_c  - \log(2)) } 
			\\ \nonumber
		 	&\leq& e^{- (1+r_i)(\log(q+1)  + c  ) } = e^{- (1+r_i)(s_q +c )},
\eea
for every $i=1, \dots, n$, where we used in the last inequality that
\bea
\nonumber
	r_i( u_c - \log(2)) &=& \tfrac{r_i}{r_-} (c  + (1+ r_+)  \log(2)) 
	\\ \nonumber
	&\geq& c + (1+r_i)( \log(q+1) - \log(q)).
\eea
This shows
\be
\nonumber
		\bar x \in \bar R_{\bar \sigma}(\bar p/q, s_q +c ) \subset \bigcup_{s_q \leq  t + u_c- (c + u_c)} \cN(R_q, s_q + c ).
\ee
Hence, we may assume  $\log(q) > t + u_c - (c+u_c)$ and obtain that
for every $i=1, \dots, n$,
\bea
\nonumber
     			\lvert x_i - \frac{p_i}{q} \rvert 
		   	&\leq & e^{-\log(q) - r_i (t+u_c)} 
			\\ \nonumber
		   	&\leq& e^{-(1+r_i)(t+u_c) + (c + u_c)) } 
		 	\leq  e^{- (1+r_i)t},
\eea
since $r_i u_c \geq c$.
This yields $\bar p/q \in \bar R_{\bar \sigma}(\bar x, t)$. 
Thus $\cF$ locally contains the set of rationals, see \eqref{ContainsSpheres}.
Since  by assumption $X=\R^n$, it is readily shown (by replacing $u_*$ with $u_c$ in the proof of Proposition \ref{DirichletProposition}).
that $(\Omega, \bar R_{\bar \sigma},\mu)$ is $\tau_u(c)$-Dirichlet with respect to $\cal{F}$  and the parameters $(c, u_c)$,
where
$\tau_u(c)= \tfrac{c_{1}}{c_{2}} e^{-2 \tau d_*}\cdot e^{- \tau (c + u_c)}$.
This shows the first claim.

Now set $\bar u_c=u_c+2d_*$.
Given 
$\omega= (x, \bar t_k ) \in \Omega$ with $x \in U_{k-1}(c)$ (defined as in \eqref{Dirichlet2} with respect to $R_{\bar \sigma}$)
the above argument again shows that there is $\bar p/q \in \bar R_{\bar \sigma}(\bar x, t)$ with $\bar p/q \in \cR(\bar t_k +\bar u_c -2d_*)$.
Hence,
\be
\nonumber
	\mu(R_{\bar \sigma}(\omega) \cap \bigcup_{s_n\leq \bar t_k + \bar u_c - 2d_*} \cN( R_n, s_n+  c  + d_*)) \geq \tfrac{1}{4} e^{-(n+1)(c+u_c + d_*)} \cdot \mu(R_{\bar \sigma}(\omega)),
\ee
which shows the second claim and finishes the proof.
\end{proof}

Note moreover the following.

\begin{lemma}
Given $c>0$ we have   \textbf{Bad}$^{\bar R_{\bar \sigma}}_X(\cal{F}, c) \subset $ \textbf{Bad}$_{\R^n}^{\bar r }( \tfrac{1}{2} e^{-c} ) \cap X$
and $\textbf{Bad}_{\R^n}^{\bar r }( e^{-c} )\subset \textbf{Bad}_{\R^n}^{\bar R_{\bar \sigma}}(\cal{F} , c )$.
\end{lemma}

\begin{proof}
For the first inclusion,
if $\bar x \in $ \textbf{Bad}$_X^{\bar R_{\bar \sigma}}(\cal{F}, c)$,
then for every $\bar p/q$, where $\bar p=(p_1,\dots ,p_n) \in \Z^n$ and  $q\in \N$,
$\bar x \not \in \cN(R_{q}, s_{q} + c) \supset \bar R_{\bar \sigma}(\bar p/q, s_{q} + c  )$.
Hence, for some $i\in\{1, \dots, n\}$, we have
\be
\nonumber
	\lvert x_i -  p_i/q \rvert \geq  e^{-(1+r_i)(s_{q} +c)} 
	=  \frac{e^{-(1+ r_i) c}}{(q+1)^{1+r_i}} 
	\geq \frac{e^{-(1+ r_i) c}}{2 q^{1+r_i}},
\ee
hence $\bar x \in \textbf{Bad}_{\R^n}^{\bar r }( \tfrac{1}{2} e^{-c} ) \cap X$.
For the second inclusion, let $\bar x \in \textbf{Bad}_{\R^n}^{\bar r }(e^{-c} ) $.
Thus, for every $\bar p/q $ with $\bar p\in \Z^n$ and  $q\in \N$ there exists $i\in\{1, \dots, n\}$ such that 
\bea
\nonumber
	\lvert x_i -  \frac{p_i}{q} \rvert &\geq& \frac{e^{-(1+r_i)c}}{ q^{-(1+r_i)}} \geq  e^{-(1+r_i) (s_q + c) },
\eea
hence $\bar x \in  \textbf{Bad}_{\R^n}^{\bar R_{\bar \sigma}}(\cal{F} , c )$,
finishing the proof.
\end{proof}

Finally, using the above Propositions, we obtain the following.
First, Theorem \ref{ThmBounds} [LB], together with Proposition \ref{PropositionMeasure} and \eqref{Inequalities} (giving $[k_c, \bar k_c]$), establishes 
\bea
\nonumber
	\text{dim} (\textbf{Bad}_{\R^n}^{\bar r }(\tfrac{1}{2} e^{-(2c +l_*)} ) \cap X) 
	& \geq&\text{dim}(\textbf{Bad}_{X}^{\bar R_{\bar \sigma}}(\cal{F}, 2c + l_* +d_*)) 
	\\ \nonumber
	&\geq& d_{\mu}(X)  
	- \frac{ \lvert  \log(1-\tau_l(c)) + \log( k_c) - \log(2\bar k_c) \rvert  }
	{ (1+r_-) c } 
	\\ \nonumber
	&\geq& d_{\mu}(X)  
	- \frac{  \log( 2) +  2\log( \tfrac{c_2}{c_1})  +  2\tau d_*+ \lvert \log(1- c_{\delta} e^{2\delta d_*} e^{-\delta c})  \rvert  }
	{ (1+r_-) c },
\eea
which is the lower bound for the first part of the Theorem.

In the following, let $\mu$ be the Lebesgue measure so that $X=\R^n$ with $\bar r \in \Q^n_{>0}$ and hence $\bar \sigma \in \Q^n_{>0}$.
Suppose each $\sigma_i= p_i/q$ with the same denominator $q \in \N$.
Then for $c= q \log(m)$ for some $m \in \N$,
Theorem \ref{ThmBounds} [LB] together with Proposition \ref{PropositionRectangle} show
\bea
\nonumber
	\text{dim} (\textbf{Bad}_{\R^n}^n(\tfrac{1}{2} e^{-(2c +  l_*+d_*)} )) 
	&\geq &
	 n  - \frac{ \lvert   \log(1-  c_n e^{(1+r_-)2d_*} e^{- (1+r_-)c }) \rvert  }{  (1+r_-)c }.
\eea
For the upper bound, let $c+ \bar u_c=q\log(m)$ for some $m \in \N$.
Using Proposition \ref{PropositionRectangle}, Lemma \ref{UBFormula2} and Lemma \ref{RefineCover} (with $\hat \sigma= 1+r_+$ and $\theta=n\hat \sigma - \sum_i \sigma_i$)  show
\bea
\nonumber
	\text{dim}(\textbf{Bad}_{\R^n}^n(e^{-c} )) &\leq& \text{dim}(\textbf{Bad}_{\R^n}^{R_{\bar \sigma}}(\cal{F}, c ) ) 
	\\ \nonumber
	& \leq & \frac{n \hat \sigma - \sum_i \sigma_i}{\hat \sigma} + 
	  \frac{\log\big(  (1-  \tfrac{1}{4} e^{-(n+1)d_*}e^{-(n+1)(c+u_c )} ) \cdot e^{(\sum_i \sigma_i)(c+\bar u_c)} \big)}{\hat \sigma( c + \bar u_c)}
	\\ \nonumber
	& \leq &  n - \frac{ \lvert   \log(1-  \tfrac{1}{4}e^{-(n+1)d_*}e^{-(n+1)(c+u_c )} )\rvert}{(1+r_+)(c+u_c + 2d_*)}.
\eea
Restricting to $c>c_0$ sufficiently large, with $c_0$  depending only on $n$ and $\bar r$,
we obtain the bounds stated for Theorem \ref{BadN}, finishing the proof.


\subsection{The Bernoulli shift $\Sigma^+$}
\label{Bernoulli}

For $n\geq 1$, 
let $\Sigma^+=\{1,\dots ,n\}^{\N}$ be the set of one-sided sequences in symbols from $\{1,\dots ,n\}$.
Let $T$ denote the shift and let $d^+$ be the metric given by 
$d^+(w, \bar w)\equiv e^{-\min\{ i \geq 1: w(i) \neq \bar w(i) \}}$ for $w \neq \bar w$ and $d(w, w)\equiv0$.
Note that  $\text{dim}(\Sigma^+)=\log(n)$.
 
Fix a periodic word $\bar w \in \Sigma^+$ of  period $p\in \N$. For $c \in \N$, consider the set 
\be
\nonumber
	S_{\bar w}(c) = 
	\{ w\in \Sigma^+ : T^k w \not \in B(\bar w,  e^{-(c +1)}) \text{ for all } k \in \N\}.
\ee

\begin{theorem}
For every $c\in \N$ we have
\be
\nonumber
	\emph{dim}(S_{\bar w}(c))
		 \leq   \log(n)  -
	\frac{ \lvert \log(1 - n^{- c}) \rvert } {c},
\ee
as well as 
\be
\nonumber	
  \emph{dim}(S_{\bar w}(2c+p +1 )) \geq \log(n)  - \frac{ \lvert \log( 1- n^{- c}) \rvert}
	{ c }.
\ee
\end{theorem}

\begin{remark} 
Note that the Morse-Thue sequence $w$ in $\{0,1\}^{\N}$ is a particular example of a word in $S_{\bar w}(2p)$ for any periodic word $\bar w$ or period $p$.
In fact, $w$ does not contain any subword of the form $WWa$ where $a$ is the first letter of the subword $W$;
for details and more general words in $S_{\bar w}$, 
we refer to an earlier work of Schroeder and the author \cite{SchroederWeil}.
\end{remark}

\begin{proof}
For $k \in \N$ and  $ w_k \in \{1,..,n\}^k $,  let $\bar w_k \in \Sigma^+$ denote the word $\bar w_k=w_k \bar w$.
Consider the resonant sets $R_0 = \{\bar w\}$ and for $  k \in \N$
\be
\nonumber
	 R_{k} = \{\bar  w_l \in \Sigma^+ : w_l \in \{1,..,n\}^l, l\leq k  \}) \cup R_0, 
\ee 
which we give the size $s_k = k+1$.
The family $\cal{F} = ( R_k, s_k)$ is increasing and discrete.

Note that we let $\Omega= \Sigma^+ \times \N$ and consider the standard function $B_1$
for which we have $d_*=0$ for $[d_*]$ and $[d_*, \cF]$ and $\sigma=1$ for $[\sigma]$. 
Moreover, we have \textbf{Bad}$(\cal{F}, c) = S_{\bar w}(c)$.
To see this, note that $d^+(T^{k-1} w, \bar w) \leq e^{-(c+1)}$ if and only if 
\be
\nonumber
	w(k)\dots w(k+c)=\bar w(1)\dots \bar w(c).
\ee
Thus, for $w_k=w(1)\dots w(k)$ and $\bar w_k=w_k\bar w$ we have
$d^+(w, \bar w_k) \leq e^{-(k+c +1)}$ if and only if $w \in B(\bar w_k, e^{-(s_k +c)}) \subset \cN_{B_1}(R_k, s_k+c)$.

For the lower bound, let $\bar w_m$ and $\tilde w_m \in R_m$ be distinct.
By definition of $\bar w_m$ and $ \tilde w_m$ there exists $i \in \{1, \dots, m+p\}$ such that $\bar w_m(i) \neq \tilde w_m(i)$;
hence  
\be
\nonumber
	d^+(\bar w_m, \tilde w_m) \geq  e^{-(p+m+1)} =  e^{-p} e^{-s_m}
\ee
and we are given the special case \eqref{Distinct} with $\bar c=e^{-p}$.
Moreover, for the probability measure $\mu = \{1/n, \dots , 1/n\}^{\N}$, $(\Omega, B_1,\mu)$ satisfies 
\be
\nonumber
	 \mu(B(w, e^{-(t+1)})) = n^{-t} = n e^{- \log(n)(t+1)},
\ee
and hence a  $(\log(n),n,n)$-power law.
From Lemma \ref{DistinctLemma} we see that $(\Omega, B_1,\mu)$  is $(\log(n), 1)$-decaying with respect to $\cal{F}$ and the parameters $(c, p+1)$.
Applying Theorem \eqref{ThmBounds} [LB], together with Proposition \eqref{PropositionMeasure} and \eqref{Inequalities} (giving $k_c=\bar k_c=n^c$), 
we obtain 
\be
\nonumber	
  \text{dim}(S_{\bar w}(2c+p +1)) \geq \log(n)  - \frac{  \log(2) + \lvert \log( 1- n^{- c}) \rvert}	{ c  }.
\ee
However, note that  $(\Omega, B_1)$ satisfies a partition as in \eqref{CubeCover} for every $c\in \N$. 
Following the arguments of the proof of Proposition \ref{PropositionCube},
we can see that the constant '$\log(2)$' can be omitted.

For the upper bound, let $(w, s_k)= (w, k+1) \in \Omega$.
If $w_k\equiv w(1) \dots w(k)$, let $\bar w_k \equiv w_k\bar w\in R_k$ which lies in $B(w, e^{-s_k})$; 
hence, $R_k \cap B_1(w,s_k) \neq \emptyset$.
Thus,  Lemma \ref{DirichletProposition} shows that  $(\Omega, B_1,\mu)$  is $(\log(n),1)$-Dirichlet  with respect to $\cal{F}$ for $u_*=0$.
Theorem \eqref{ThmBounds} [UB], together with Proposition \eqref{PropositionMeasure} and \eqref{Inequalities} (giving $K_c=n^c$), yields
\be
\nonumber
	\text{dim}(S_{\bar w}(c))
		 \leq   \log(n)  -
	\frac{ \lvert \log(1 - n^{- c}) \rvert } {c},
\ee
finishing the proof.
\end{proof}


\subsection{The geodesic flow in $\H^{n+1}$}
\label{GeodesicFlow}

Although the following setting is even suitable for proper geodesic CAT(-1) metric spaces, 
we restrict to the real hyperbolic space $\H^{n+1}$.
The reason is, given a non-elementary geometrically finite Kleinian group $\Gamma$, the existence of a nice measure 
satisfying the \emph{Global Measure Formula} (see Theorem \ref{GMF}).
We start by  introducing the setting and a model of Diophantine approximation developed by Hersonsky, Paulin and Parkkonen in 
\cite{HersonskyPaulin,HersonskyPaulin2,Parkkonen2}, which allows a dynamical interpretation of badly approximable elements.

In the following, $\H^{n+1}$ denotes the $(n+1)$-dimensional real hyperbolic ball-model. 
For $o\in \H^{n+1}$, we define the \emph{visual metric} $d_o:  S^n \times S^n \to [0, \infty)$ at $o$ by  $d_o(\xi, \xi )\equiv 0$ and  
\be
\nonumber
	d_o(\xi, \eta) \equiv e^{-(\xi, \eta)_o},
\ee
for $\xi \neq \eta$,
where $(\cdot, \cdot)_o$ denotes the Gromov-product at $o$.
Note that if $o=0$ is the center of the ball $\H^{n+1}$ then the visual distance $d_0$ is bi-Lipschitz equivalent to the angle metric on the unit sphere $S^n$.
The boundary $S^n = \partial_{\infty}\H^{n+1}$ is a compact metric space with respect to $d_o$ and we will consider all metric balls to be with respect to $d_o$ in the following.

Let $\Gamma$ be a discrete subgroup of the isometry group $I(\H^{n+1})$ of $\H^{n+1}$.
Note that an isometry $\varphi$ of $\H^{n+1}$ extends to a homeomorphism of $S^n$. We denote the image of a set $S \subset S^n$ under $\varphi$ by $\varphi. S$.
The \emph{limit set} $\Lambda\Gamma$ of $\Gamma$ is given by the set $\overline{ \Gamma.o} \cap S^n$ (independent of $o$),
which is the set of all accumulation points of subsequences from $\Gamma.o \equiv \{\varphi(o): \varphi \in \Gamma\}$.
Recall that a subgroup $\Gamma_0 \subset \Gamma$ is called \emph{convex-cocompact} 
if $\Lambda \Gamma_0$ contains at least two points and the action of $\Gamma_0$ on the convex hull $\cal{C}\Gamma_0$ has compact quotient.
We call $\Gamma_0$ \emph{bounded parabolic} if $\Gamma_0$ is the maximal subgroup of $\Gamma$ stabilizing a parabolic fixed point $\xi_0 \in \Lambda\Gamma$
and $\Gamma_0$ acts cocompactly on $\Lambda\Gamma - \{\xi_0\}$.
Moreover, we call $\Gamma_0$  \emph{almost malnormal} 
if $\varphi. \Lambda\Gamma_0 \cap \Lambda \Gamma_0 
= \emptyset$  for every $\varphi \in \Gamma-\Gamma_0$.

Let $\Gamma$ be a non-elementary geometrically finite group.
We refer to \cite{Ratcliffe} for the following.
For the convex hull $\cal{C}\Gamma$ of $\Lambda\Gamma$, the  subset $\cal{C}\Gamma \cap \H^{n+1}$ of  $\H^{n+1}$  is  closed, convex and $\Gamma$-invariant.
The convex core $\cal{C}M \subset M$ of $M=\H^{n+1}/ \Gamma$ is the convex closed connected set
\be
\nonumber
	\cal{C}M \equiv (\cal{C}\Gamma \cap \H^{n+1})/ \Gamma = K \cup   \bigcup_i V_i,
\ee
which can be decomposed into a compact set $K$, and, unless $\Gamma$ is convex cocompact, 
finitely many open disjoint sets $V_i$ corresponding to the conjugacy classes of maximal parabolic subgroups of $\Gamma$ which are bounded parabolic and almost malnormal.
Moreover, if  $\pi $ denotes the projection to $M= \H^n/\Gamma$
we may assume that each $V_i=\pi(C_i)\cap \cC M$ is the projection of a horoball $C_i$ contained in the convex core $\cC M$, 
where the collection $\varphi(C_i)$, $\varphi \in \Gamma - $ Stab$_{\Gamma}(C_i)$, is disjoint.

We call the projection $V_i$ and every projection of a smaller horoball (in the convex core) contained in $V_i$ a standard cusp neighborhood.

%
\subsubsection{The setting.}
Let $\Gamma$ be a non-elementary geometrically finite group without elliptic elements as above
and $\Gamma_{i} \subset  \Gamma$, $i=1,2$, be an almost malnormal subgroup in $\Gamma$ of infinite index.
We treat the following two 'disjoint' cases simultaneously.
\begin{itemize}
\item[1.] There is  precisely one conjugacy class of a maximal parabolic subgroup $\Gamma_1$ of $\Gamma$.
	Let   $m$ be the rank of $\Gamma_1$ and let $C_1$ be a horoball based at the parabolic fixed point $\xi_0$ of $\Gamma_1$ as above.
\item[2.] Let $\Gamma$ be convex-cocompact  such that $\Lambda\Gamma \subset S^n$ is not contained in a finite union of spheres of $S^n$ of codimension at least $1$.
	Let $\Gamma_2$ be a convex-cocompact subgroup and  $C_2 = \cal{C}\Gamma_2$ be the convex hull of $\Gamma_2$ which is a hyperbolic subspace 
	(that is, $C_2$ is totally geodesic and isometric%
\footnote{ With respect to the induced metric on $C_2$.}
	 to the  hyperbolic space $\H^m$).
\end{itemize}

\begin{remark}
The requirements that there is only one parabolic subgroup in \emph{Case 1.} or that $\Gamma$ itself is convex-cocompact in \emph{Case 2.}
will be necessary for the Global Measure Formula.
More precisely, we need to control the 'depth of geodesic rays in the cuspidal end' which is not possible in  \emph{Case 2.} if $\cal{C}M$ is not compact.%
\footnote{ However, when $\Gamma$ is a lattice so that the measure will satisfy a power law, \emph{Case 2.} is possible but further arguments for Lemma \ref{DirichletCompact} are necessary. }
\end{remark}

Note that, since $\Gamma_i$ is almost malnormal, we have $\Gamma_i = $ Stab$_{\Gamma}(C_i)$ so that $\Gamma_i$ is determined by $C_i$.
In addition, $C_i$ is \emph{$(\e, T)$-immersed}, 
that is,  
for every $\e>0$ there exists $T=T(\e) \geq 0$ such that for all $\varphi \in \Gamma - \Gamma_{i}$ 
we have that diam$(\cal{N}_{\e}(C_i) \cap  \varphi(\cal{N}_{\e}(C_i)) \leq T$; see \cite{Parkkonen2}.
In the first case, we therefore assume, after shrinking $C_1$, that the images $\varphi ( C_1)$, $[\varphi] \in \Gamma/\Gamma_1$,
form a disjoint collection of horoballs.
For the second case, we let $\e=\delta_0$ and $T_0=T(2\delta_0)$ where $\delta_0$ is the constant such that $\H^{n+1}$ is a tripod-$\delta_0$-hyperbolic space.

\begin{example}
Clearly, if $M= \H^{n+1}/\Gamma$ is a finite volume hyperbolic manifold with exactly one cusp, then \emph{Case 1.} is satisfied with $m=n$. 
If $\Gamma$ is even cocompact, 
then every closed geodesic $\alpha$ in $M$ determines a subgroup $\Gamma_2$ as in \emph{Case 2.} and $C_2$ (a lift of $\alpha$) is one-dimensional.
Moreover, $T$ can be estimated in terms of the length of $\alpha$ and the length of a systole of $M$ (or with the injectivity radius 'along' $\alpha$).
\end{example}

\subsubsection{A model of Diophantine approximation and the main result.}
Given $\Gamma$, $\Gamma_i$, $i=1,2$, as above, 
we fix a base point $o\in \H^{n+1}$ such that $\pi(o) \in K$.
For technical reasons, we also fix a sufficiently large constant $t_0\geq 0$.
For the respective cases, $i=1,2$, denote the quadruple of data by
\be
\nonumber
  \cal{D}_i=(\Gamma, C_i, o, t_0).
\ee
For $r=[\varphi] \in\Gamma/\Gamma_{i}$ we define
\be
\nonumber
	D_i(r) = d(o, \varphi (C_i))
\ee
which does not depend on the choice of the representative $\varphi$ of $r$. 
Note that the set $\{D_i(r) : r\in   \Gamma/\Gamma_{i}\}$ is discrete and unbounded (see \cite{Parkkonen2,Weil2}); that is,
for every $D \geq 0$ there are only finitely many elements $r\in \Gamma/\Gamma_{i}$ 
such that $D_i(r) \leq D$ and there exists an $r\in \Gamma/\Gamma_{i}$ with $D_i(r)>D$.

Now, for $i=1,2$ and for $\xi \in \Lambda\Gamma - \Gamma. \Lambda\Gamma_{i}$ define the \emph{approximation constant}
\be
\nonumber
	c_i(\xi) = \inf_{r=[\varphi] \in \Gamma/\Gamma_{i} : \ D_i(r)>t_0} e^{D_i(r)} d_o(\xi, \varphi.\Lambda \Gamma_i),
\ee
If $c_i(\xi)=0$ then $\xi$ is  called \emph{well approximable}, 
otherwise it is called \emph{badly approximable} (with respect to $\cal{D}_i$).
Define the set of badly approximable limit points  by
\be
\nonumber
	\textbf{Bad}(\cal{D}_i) = \{\xi \in \Lambda\Gamma - \Gamma.\Lambda \Gamma_{i} : c_i(\xi)>0 \} \subset \Lambda\Gamma,
\ee
and $\textbf{Bad}(\cal{D}_i, e^{-c})$ the subset of elements for which $c_i(\xi) \geq e^{-c}$.

\begin{theorem}
\label{ThmGF}
Let $\delta$ be the Hausdorff dimension of $\Lambda\Gamma$. 
Then there exists $c_0>0$, constants $k_l, \bar k_l, k_u, \bar k_u, \tilde k_u >0$ determined in the following,  
and an exponent $\tau>0$ (from Theorem \ref{SU} below)
such that for all $c>c_0$ we have
\bea
\nonumber
	\delta  - \frac{  k_l + \lvert \log(1- \bar k_l\ e^{ - (2\delta-m) c/2 }) \rvert  }	{ c/2 - (\delta_0 + \log(2))}
		 \leq \emph{dim(\textbf{Bad}}(\cal{D}_1, e^{-c} )) 
	 \leq  (2\delta  - m ) - \frac{  \lvert \log(1- \bar k_u\ e^{- (3\delta   - m) c }   ) \rvert - \tilde k_u  }{2c + k_u},
\eea
as well as 
\bea
\nonumber
	  \delta 	- \frac{ k_l+ \lvert \log(1- \bar k_l\ e^{-\tau c/2}) \rvert  }{c/2 - (T_0+ \delta_0 + 2\log(3))} 
		  \leq  \emph{dim(\textbf{Bad}}(\cal{D}_2, e^{- c}))
	  \leq \delta -  \frac{   \lvert \log(1- \bar k_u\ e^{- \delta c}) \rvert - \tilde k_u  }{ c + k_u}.
\eea
\end{theorem}

\begin{remark}
It is well known (see \cite{Nicholls}) that $2\d \geq m$.
In fact, it follows from the lower and upper bound that $\delta \geq m$ in our case.
Therefore, the upper bound is only suitable for $c>0$ such that the right hand side is smaller than the trivial bound $\delta$.
For the second case, note that if $C_2$ is an axis, we can choose $\tau=\delta$.
We moreover expect that $\tau$ is dependent on the dimension of $C_2$ (and of course on $\delta$).
\end{remark}

In the special case when $\Gamma$ is of the \emph{first kind}, that  is $\Lambda\Gamma =S^n$ (for instance if $\Gamma$ is a lattice),
we can improve the above theorem to the following.
 
\begin{theorem}
\label{TheoremBoundedGeodesics}
Let again $\tau>0$ be the exponent of Theorem \ref{SU} below.
If in addition $\Lambda\Gamma=S^n$, then there exists $c_0>0$ and  constants $k_l, \bar k_l, k_u, \bar k_u%
\footnote{ The constants may differ from the ones in the proof. }
>0$, 
such that for all $c>c_0$ we have
\bea
\nonumber
	n  - \frac{   \lvert \log(1- \bar k_l \ e^{ - n c/2 }) \rvert  }	{ c/2 - k_l }
		 \leq  \emph{dim(\textbf{Bad}}(\cal{D}_1, e^{-c} )) 
	 \leq n - \frac{ \lvert \log(1- \bar k_u\ e^{-2n c } ) \rvert   }{2c + k_u},
\eea
as well as 
\bea
\nonumber
	 n 	- \frac{  \lvert \log(1- \bar k_l\ e^{-\tau c/2}) \rvert  }{c/2 - k_l} 
	  \leq  \emph{dim(\textbf{Bad}}(\cal{D}_2, e^{- c}))
	  \leq n -  \frac{\lvert \log(1- \bar k_u\ e^{-n c}) \rvert }{ c + k_u}.
\eea
\end{theorem}

As a corollary of Theorem \ref{TheoremBoundedGeodesics} we obtain Theorem \ref{ThmB} and Theorem \ref{ThmC}. 


\subsubsection{Proof of Theorem \ref{ThmB} and Theorem \ref{ThmC}}
\label{SectionProofBC}
Notice the following dynamical interpretation of the set $\textbf{Bad}(\cal{D}_i, e^{- c})$.

\begin{lemma}[ \cite{Weil2}, Lemma 3.16 for the context of CAT(-1)-spaces]
\label{Dynamical}
There exist positive constants $c_0$, $\kappa_0>0$ (we may assume $\k_0\geq1$) and $t_0\geq 0$  with the following property:
if $C_1$ is a horoball based at $\partial_{\infty}C = \eta \in \partial_{\infty}\H^n$ or $C_2$ is a hyperbolic subspace with $d(o, C_i) \geq t_0$, 
then for all $\xi \in \Lambda\Gamma$ and $c>c_0$ we have
\begin{itemize}
\item[1.] 	$\gamma_{o, \xi}([t, t + c]) \subset C_1$, 
\item[2.] 	$\gamma_{o, \xi}([t, t+c]) \subset \cal{N}_{\delta_0}(C_2)$, 
\end{itemize}
for some $t\geq d(o,C_i)$,
if and only if 
\begin{itemize}
\item[1.] 	$d_o(\xi, \eta) \leq \kappa_0\ e^{-c/2} \cdot e^{-d(o, C_1)}$,
\item[2.]	$d_o(\xi, \partial_{\infty}C_2) \leq \kappa_0\ e^{-c} \cdot e^{-d(o, C_2)}$.
\end{itemize}
\end{lemma}

\noindent Note moreover that the penetration length $c$ of a geodesic ray in a horoball $C$ differs from its 'height' $c/2$ only up to additive constants; see again \cite{Weil2}.

Recall the sets $\textbf{Bad}_{M,H_0, o}(t)$ and $\textbf{Bad}_{M,\cN_{\e}(\a), o}(L)$ considered in Theorem \ref{ThmB} and Theorem \ref{ThmC} for large constants $t$ and $L$.
By Lemma \ref{Dynamical} and the above discussion, they
lift to the sets $\textbf{Bad}(\cal{D}_1, e^{- \tilde t})$ and  $\textbf{Bad}(\cal{D}_2, e^{- \tilde L})$ respectively  for suitable data $\cD_1$ and $\cD_2$, where $\tilde t$, $\tilde L$ agree with $t$ and $L$ up to a  constant independent on $t$, $L$ (determined by Lemma \ref{Dynamical}).
Finally, we finish the proof by applying Theorem \ref{TheoremBoundedGeodesics} and the Taylor expansion. 


\subsubsection{Preparation: a measure on $\Lambda\Gamma$.}
\label{PattersonSullivanMeasure}
Let $o=0$ be the center so that the visual metric $d_o$ is bi-Lipschitz equivalent to the angle metric on the unit sphere $S^n$.
Hence, if $\Gamma$ is of the first kind, 
then  the Lebesgue measure on $S^n$ satisfies a power law with respect to the visual metric $d_o$ and the exponent $n$.
More generally, recall  that the \emph{critical exponent} of a discrete group $\Gamma \subset I(\H^{n+1})$ is given by
\be
\nonumber
	\delta(\Gamma) \equiv \inf \big\{s>0 : \sum_{\varphi \in \Gamma} e^{-s d(x, \varphi(x)) } < \infty \big\},
\ee
for any $x\in \H^{n+1}$.
If $\Gamma$ is non-elementary and discrete 
then the Hausdorff dimension of the conical limit set of $\Lambda\Gamma$ equals $\delta(\Gamma)$
and if  $\Gamma$ is moreover geometrically finite, then dim$(\Lambda\Gamma)= \delta(\Gamma)$ (see \cite{BishopJones}).

Moreover, associated to $\Gamma$, there is a canonical measure, the \emph{Patterson-Sullivan} measure $\mu_{\Gamma}$, 
which is a $\delta(\Gamma)$-conformal probability measure supported on $\Lambda\Gamma$. 
For a precise definition we refer to \cite{Nicholls}.
There are various results concerning the Patterson-Sullivan measure.
Here, we will make use of the following.

Let $\Gamma$ be a non-elementary geometrically finite Kleinian group as in \emph{Cases 1.} and \emph{2.} above.
Let moreover $D_0$ be the diameter of the compact part $K$ of the convex core $\cal{C}M$ of $M$.
For a limit point $\xi \in \Lambda \Gamma$, we let $\gamma_{o, \xi}$ be the unique geodesic ray starting in $o$ and asymptotic to $\xi$.
In \emph{Case 1.} define the \emph{depth} $ D_t(\xi)$ 
of the point $\gamma_{o, \xi}(t)$ in the collection of horoballs $\{ \varphi(C_1)\}_{\varphi \in \Gamma}$,
where  $D_t(\xi) \equiv 0$ if $\gamma_{o, \xi}(t)$ does not belong to $\cup_{\varphi\in \Gamma} \varphi(C_1)$,
and $D_t(\xi) \equiv d(\gamma_{o, \xi}(t), \partial \varphi(C_1))$ otherwise; in \emph{Case 2.} we simply set $D_t(\xi)=0$ for all $t>0$.

We need the following Lemma.

\begin{lemma} We have
\be
\nonumber
	D_t(\xi) \leq d( \gamma_{o, \xi}(t), \Gamma.o) \leq  D_t(\xi) + 4\log(1+ \sqrt{2}) + D_0.
\ee
\end{lemma}

\begin{proof}
By the arguments given below, the proof is obvious if $\Gamma$ is convex-cocompact (and hence the set $V$ is empty) and we may assume that we are given \emph{Case 1.}
Recall that the convex core $\cal{C}M=(\cal{C}\Gamma \cap \H^{n+1})/\Gamma$ consists of (the disjoint union of) the compact set $K$ and the set $V$ which we may assume to be the projection of $C_1 \cap \cal{C}\Gamma$.
Since $\cal{C}\Gamma$ is convex and $o\in \cal{C}\Gamma$,
for every limit point $\xi \in \cal{C}\Gamma$ the ray $\gamma_{o, \xi}(\R^+)$ is contained in $\cal{C}\Gamma$ and hence covered by lifts of $K$ and of $V$.
Since $\pi(o)\in K$, if $\gamma_{o, \xi}(t)  \in \cal{C}\Gamma - \cup_{\varphi} \varphi(C_1)$ for some $t>0$, then $d(\gamma_{o, \xi}(t), \Gamma.o) \leq D_0$.

Hence, fix $t>0$ such that $\gamma_{o, \xi}(t) \in \varphi(C_1) \equiv C$ for some $\varphi \in \Gamma$, where we let $\eta \equiv \varphi(\xi_0)$.
If we let $t_0$ be the entering time of $\gamma_{o, \xi}$ in $C$, that is, $\gamma_{o, \xi}(t_0) \in \partial C$,
then clearly by the above remark and since  $\gamma_{o, \xi}(t_0)$ belongs to some lift of $ K$,
we have 
\be
\nonumber
	d( \gamma_{o, \xi}(t), \Gamma.o) \leq d( \gamma_{o, \xi}(t), \gamma_{o, \xi}(t_0))  + D_0 \equiv \bar d + D_0.
\ee
Moreover, let $\tilde C$ be the horoball based at $\eta$ (and contained in $C$) such that $\gamma_{o, \xi}(t) \in \partial \tilde C$
and note that $\gamma_{o, \eta}( d(o, C) + D_t(\xi)) \in \partial \tilde C$.
It then follows from \cite{PaulinParkkonen2}, Lemma 2.9, that both $d( \gamma_{o, \xi}(t_0), \gamma_{o, \eta}( d(o, C)))$ and
$d( \gamma_{o, \xi}(t), \gamma_{o, \eta}( d(o, C) + D_t(\xi)) )$ are bounded above by the constant $2 \log(1+ \sqrt{2})$.
This shows
\bea
\nonumber
		\bar d &=& d( \gamma_{o, \xi}(t), \gamma_{o, \xi}(t_0)) 
			\\ \nonumber
		&\leq& d(\gamma_{o, \xi}(t), \gamma_{o, \eta}( d(o, C) + D_t(\xi)) )  + d(\gamma_{o, \eta}( d(o, C) + D_t(\xi)) , \gamma_{o, \xi}(t_0))
			\\ \nonumber
		&\leq&   2\log(1+ \sqrt{2}) + ( D_t(\xi) + d(\gamma_{o, \eta}( d(o, C) ) , \gamma_{o, \xi}(t_0)) )
			\\ \nonumber
		&\leq&   D_t(\xi) + 4\log(1+ \sqrt{2}).
\eea
Finally, since $o \not \in C$ (used in the first inequality) we have
\bea
\nonumber
		D_t(\xi) &\leq& d( \gamma_{o, \xi}(t), \Gamma.o) 
	\\ \nonumber
		&\leq&  \bar d + D_0 
		\\ \nonumber
		&\leq& D_t(\xi) + 4\log(1+ \sqrt{2}) + D_0,
\eea
proving the claim.
\end{proof}

In the following, let $\mu= \mu_o$ be the Patterson-Sullivan measure given at the base point $o$.
By the above lemma, we can  reformulate the \emph{Global Measure Formula} due to \cite{StratmannVelani}, Theorem 2, to the following.

\begin{theorem} 
\label{GMF}
There exist positive constants $c_1, c_2>0$ and $t_0>0$ such that
for all $\xi \in \Lambda\Gamma$ and for all $t>t_0$, we have that
\be
\nonumber
		c_1  e^{- \delta t} \cdot e^{-(\d - m ) D_t(\xi) } \leq \mu(B_{d_o}(\xi, e^{-t})) \leq c_2 e^{- \delta t} \cdot e^{-(\d -m ) D_t(\xi) }.
\ee
In particular, if $\Gamma$ is convex-cocompact, then $\mu$ satisfies a power law with respect to $\delta$.%
\footnote{ The same is true if $\delta$ equals $m$ and in particular if $\Gamma$ is of the first kind in which case $\mu$ is equivalent to the Lebesgue measure on $S^n$. }
\end{theorem}

For the second case, let again  $o=0$ 
and note that, since $\Gamma_2$ is almost malnormal in $\Gamma$, $C_2$ can be of dimension at most $m \leq n$.
Moreover, since  $C_2$ is an $m$-dimensional hyperbolic subspace, 
the boundary  $\partial_{\infty}C_2  = \Lambda \Gamma_2 \subset \Lambda\Gamma$ of $C_2$ is an $(m-1)$-dimensional  sphere (with respect to $d_0$).
Hence, every image $\varphi. \Lambda\Gamma_2$, $\varphi \in \Gamma$, 
is contained in the set $\cal{H}(\Gamma) \equiv \{ S \cap \Lambda \Gamma : S$ is a sphere in $S^n$ of codimension at least $1$$\}$.
A finite Borel measure $\mu$ on $S^n$ is called \emph{$\cal{H}(\Gamma)$-friendly},
if $\mu$ is Federer and $(\Lambda\Gamma \times (t_0, \infty), B_1, \mu)$ is absolutely $(\tau, c_{\tau})$-decaying with respect to $\cal{H}(\Gamma)$.

\begin{theorem}[\cite{StratmannUrbanski}, Theorem 2]
\label{SU}
For every non-elementary convex-cocompact discrete group $\Gamma \subset I(\H^{n+1})$ (without elliptic elements),
such that $\Lambda\Gamma$ is not contained in a finite union of elements of $\cal{H}(\Gamma)$,
the Patterson-Sullivan measure $\mu$ at $o$ is $\cal{H}(\Gamma)$-friendly.
\end{theorem}

Note that if we consider only $0$-dimensional spheres, we can clearly choose $\tau=\delta$.


\subsubsection{The resonant sets.}
Let $\bar \Omega=\Omega= \Lambda\Gamma \times (t_0, \infty)$, where $t_0$ is sufficiently large as in Theorem \ref{GMF} and Theorem \ref{SU}  above
(as well as Lemma \ref{Dynamical} and \ref{ExistenceDirichlet} below).
We are given the discrete set of sizes $\{D_i([\varphi]) : [\varphi] \in \Gamma/\Gamma_i ,D_i([\varphi]) >t_0 \}$ which we relabel to $\{s^i_m\}_{m \in \N} \subset \R^+$
and reorder such that $s^i_m \leq s^i_k$ for $m\leq k$. 
For $m \in \N$ let 
\bea
\nonumber
	 R^i_m &\equiv& \{ \xi \in \varphi.\Lambda\Gamma_{i}  :  [\varphi] \in  \Gamma/\Gamma_{i} 
	\text{ such that } t_0< D_i([\varphi]) \leq s^i_m  \}
\\ \nonumber 
	&= &\{ \xi \in \varphi.\Lambda\Gamma_{i}  :  [\varphi] \in  \Gamma/\Gamma_{i} 
	\text{ such that } e^{-t_0}> e^{-D_i([\varphi])} \geq e^{-s^i_m } \}.
\eea
Since $\Gamma$ is discrete, for every metric ball $B=B(\xi, e^{-t})$, $(\xi, t)\in \Omega$, 
only finitely many sets $\varphi. \Lambda\Gamma_i$ with $D_i([\varphi])\leq t$ can intersect $B$
and it is readily checked that $(\Omega, B_1)$ is $d_*$-separating with respect to $\cal{F}_i$ where $d_*= \log(2)$.
Moreover, since $\Lambda\Gamma$ is compact, $(\Omega, B_1)$ is $\log(3)$-separating. Clearly we have $[\sigma]$ for $\sigma=1$.

For  $\cal{F}_i \equiv (R^i_m, s^i_m)$, since $\Lambda\Gamma_i\subset S^n$ is closed (hence compact),
we remark that
\be
\nonumber
    \textbf{Bad}(\cal{D}_i, e^{-c}) = \textbf{Bad}_{\Lambda\Gamma}^{B_1}(\cal{F}_i, c).
\ee

\subsubsection{The lower bound}

For the lower bound, note that the following is shown in the author's earlier work \cite{Weil2}, Section $3.6.5$, 
using that $C_i$ is $(2\delta_0,T_0)$-immersed:
For two different cosets $[\bar \varphi]$, $[ \varphi]\in \Gamma/\Gamma_{i}$
let $\eta \in \varphi. \Lambda\Gamma_{i}$ 
and $\bar \eta \in\bar \varphi. \Lambda\Gamma_{i}$. 
Then
\be
\label{A2}
	d_o(\eta, \bar\eta)  \geq   e^{- c_i}e^{- \max\{D_i([\varphi]), D_i([\bar \varphi])\}} ,
\ee
where
\be
\nonumber
	 c_1\equiv   \delta_0, \ \ \ \  c_2  \equiv  T(2 \delta_0) + 2\delta_0,
\ee
and $\delta_0$ is the hyperbolicity constant of $\H^{n+1}$ (and $i$ stands for the respective case).

For \paragraph{Case 2.} we obtain that, for $ l_* = c_2 + \log(3)$, for any formal ball $(\xi, t) \in \Omega$ we have
\be
\nonumber
    B(\xi, e^{-(t+\bar l_*)}) \cap \cR_i(t) = B(\xi, e^{-(t+\bar l_*)}) \cap S,
\ee
where $S$ is either empty or $S= \varphi.\Lambda\Gamma_2\in \cal{H}(\Gamma)$ for some $[\varphi] \in \Gamma/\Gamma_2$.
Thus, \eqref{ContainedInSpheres} is satisfied with $n_*=1$.
Proposition \ref{SphereDecayingMeasure} and Theorem \ref{SU} show that
$(\Omega, B_1,\mu)$ is  $\tau_l(c)$-decaying with respect to $\cal{F}_2$, where $\tau_l(c)=c_{\tau} e^{-\tau(c-2\log(3))}$,
for all $c \geq 2\log(3)$ and the parameters $(c,l_c)$, $l_c= T_0 + \delta_0 + 2\log(3)$.
We let $c_0\geq 2\log(3)$ large such that for all $c\geq c_0$ we have $\tau_l(c)<1$.
Recall that $(\Omega, B_1,\mu)$ satisfies a power law with respect to the parameters $(\delta, c_1, c_2)$.
Thus, Theorem \ref{ThmBounds} [LB] together with Proposition \ref{PropositionMeasure} and \eqref{Inequalities} (giving $[k_c, \bar k_c]$) establish the lower bound
\bea
\nonumber
	 \text{dim}(\textbf{Bad}(\cal{D}_2, e^{- (2c + l_c)}))
	 &\geq & \delta 	- \frac{ \log(2) + 2\delta(\log( \tfrac{c_2}{c_1}) + \log(3)) + \lvert \log(1- c_{\tau}e^{2 \tau \log(3)}\cdot e^{-\tau c}) \rvert  }{c}.
\eea


For \paragraph{Case 1.}, \eqref{A2} implies that \eqref{Distinct} is satisfied for $l_* = \delta_0 + \log(2)$.
Using the Global Measure Formula, we can determine the required constants.

\begin{proposition} 
\label{LParameters}
Given $c>0$ (such that $\tau_l(c) <1$), 
 $[k_c, \bar k_c]$ is satisfied and $(\Omega, B_1, \mu)$ is $\tau_l(c)$-decaying with respect to $\cF_1$ and the parameters $(c, l_c)$ for $l_c= \delta_0  + \log(2)$ and 
\bea
\nonumber
	k_c & \geq&  \tfrac{c_1}{c_2} e^{- \d \d_0}e^{ - (2\d - m)c }  \equiv \bar c_1 e^{- (2\d - m)c},
	\\ \nonumber
	\bar k_c & \leq& \tfrac{c_1}{c_2} e^{ -\delta (4d_* + \d_0 )}e^{ m ( 2 d_*+ \d_0) }e^{ m c } \equiv \bar c_2 e^{-mc} \leq \bar c_2 e^{(2\d-m)c},
	\\ \nonumber
	\tau_l(c) 
	&\leq&  \tfrac{c_2}{c_1} e^{2\delta d_* + m\d_0} e^{-(2\d - m)c}\equiv \bar c_3 e^{-(2\d - m)c}.
\eea 
\end{proposition}

\begin{proof}
For any $\eta \in B(\xi, e^{- t} )$ with $t$ sufficiently large, since $e^{-(\xi, \eta)_o}=d_o( \xi, \eta) \leq e^{-t }$ and $\H^{n+1}$ is a $\delta_0$-tripod-hyperbolic space, 
we have $d( \gamma_{o, \xi}(t), \gamma_{o, \eta}(t)) < \delta_0$.
Hence $\lvert D_t(\xi) - D_t(\eta) \rvert \leq \d_0$.
Moreover, we have $\lvert D_h(\eta) - D_s(\eta) \rvert \leq \lvert h-s \rvert$ for all $h, s$.
This shows that for $\eta \in B(\xi, e^{- t} )$ and $s, h \geq 0$,
\be
\label{DifferenceD}
	\lvert D_{t + s}(\xi) - D_{t+h}(\eta) \rvert \leq \d_0 + s+h.
\ee

Recall that $t_k = s^1_1 +  kc + l_c$ and let  $(\xi, t_{k}) \in \Omega$ be a given a formal ball. 
From the above \eqref{A2}, we know that $B(\xi, e^{- t_k} )\cap \cR(t_k-l_c)$ contains at most one point, say $\eta = \varphi. \Lambda\Gamma_1$.
By \eqref{DifferenceD},  $D_{t_k + d_*}(\xi)$ and $D_{t_k + c -d_*}(\eta)$ can differ by at most $ c + \delta_0  - 2d_*$.
Moreover, since $D_1([\varphi]) \leq t_k - l_c$, we have for the depth of $\eta$ that 
\be
\nonumber
	D_{t_k + c - d_*}(\eta) = t_k + c - d_* - D([\varphi]) \geq  c + l_c - d_*.
\ee
Assuming that $c + l_c  \geq c + \delta_0 +d_* $ (which is the case for $c\geq \log(2)$), we have  $D_{t_k + c - d_*}(\eta) \geq D_{t_k + d_*}(\xi) + c + \delta_0 $.
Using the Global Measure Formula, we obtain
\bea
\nonumber
	\mu(B(\xi, e^{-(t_k +d_*)})) &\geq& c_1  e^{- \delta (t_{k} + d_*)} \cdot e^{-(\d - m) D_{t_k + d_*}(\xi)} 
	\\ \nonumber
	&\geq&  c_2 e^{-\delta(t_k + c - d_*)}\cdot  \tfrac{c_1}{c_2} e^{\delta (c - 2 d_* ) }e^{-(\d - m)( D_{t_k  + c -d_*}(\eta) - (c + \delta_0 ))} 	
	\\ \nonumber
	&\geq&  c_2 e^{-\delta(t_k + c - d_*)} e^{-(\d - m) D_{t_k  + c -d_*}(\eta)}  \cdot  
			\tfrac{c_1}{c_2} e^{2 \d (c - d_* )   }e^{-m( c + \delta_0 )} 	
	\\ \nonumber
	&\geq&  \mu(B(\eta,e^{-(t_k + c - d_*)})) \cdot   \tfrac{c_1}{c_2} e^{-2\delta d_* - m\d_0 } e^{(2\d - m)c}
	\\ \nonumber
	&\geq&  \mu(B(\xi, e^{-(t_k + d_*)}) \cap B(\eta,e^{-(t_k + c - d_*)})) \cdot  \tau_l(c)^{-1}.  
\eea

As above, using \eqref{DifferenceD} for $\eta \in B(\xi, e^{-t_k})$ and the Global Measure Formula, we obtain
\bea
\nonumber
		\mu(B(\eta, e^{- t_{k+1}})) &\geq& c_1 e^{- \delta  t_{k+1}} \cdot e^{-(\d - m ) D_{ t_{k+1}}(\xi)}
		\\ \nonumber
		 &\geq&  \mu(B(\xi, e^{- t_{k}  } ))\cdot \tfrac{c_1}{c_2} e^{- \delta c}e^{-(\d - m ) (c +\d_0)  }
		  \\ \nonumber
		 &\equiv &  \mu(B(\xi, e^{- t_{k}} )) \cdot k_c ,
\eea
as well as 
\bea
\nonumber
		\mu(B(\xi, e^{- (t_{k}+ d_*)} )) &\geq& c_1 e^{- \delta  (t_{k} + d_*)} \cdot e^{-(\d - m ) D_{ t_{k}+d_*}(\xi)}
		\\ \nonumber
		 &\geq&  \mu(B(\eta, e^{-( t_{k+1}  - d_*)} )\cdot \tfrac{c_1}{c_2} e^{ \delta (c - 2d_* )}e^{-(\d - m) (c +  2d_* + \d_0) }
		  \\ \nonumber
		  &\geq&  \mu(B(\eta, e^{-( t_{k+1}  - d_*)} )\cdot \tfrac{c_1}{c_2} e^{ -\delta (4d_* + \d_0 )}e^{ m ( 2 d_*+ \d_0) }e^{ m c }
		  \\ \nonumber
		 &\equiv &  \mu(B(\eta, e^{- (t_{k+1} - d_*) } ) \cdot \bar k_c^{-1}.
\eea
This finishes the proof.
\end{proof}

Assuming that $c>c_0$, where $c_0$ is as in Lemma \ref{Dynamical} and such that $\tau_l({c_0})<1$, the following Lemma will finish determining the parameters for the lower bound.

\begin{lemma}
For any $\xi \in \textbf{Bad}(\cal{F}, 2c+l_c)$ we have $d_{\mu}(\xi) \geq \delta$.
\end{lemma}

\begin{proof}
If $\xi \in \textbf{Bad}(\cal{F}, 2c+l_c )$, then $d_o(\xi, \varphi. \Lambda\Gamma_1) > e^{-(D_1([\varphi])+ 2c + l_c )}$ 
for every $[\varphi] \in \Gamma/\Gamma_1$ with $D_1([\varphi])>t_0$.
Hence, Lemma \ref{Dynamical} states that the length of $\gamma_{o, \xi}(\R^+) \cap \varphi(C_1)$ is bounded by $2(2c+l_c + 2\log(\k_0))$
for every $[\varphi] \in \Gamma/\Gamma_1$.
In particular, the distance from $\gamma_{o, \xi} (t )$ to $\partial \varphi(C_1)$ is less than $2c +l_c+ 2\log(\k_0)$ for all $t>t_0$ 
and we see that $0 \leq D_{t}(\xi) \leq 2c + l_c + 2 \log(\kappa_0)$.
The Global Measure Formula yields that $c_1 e^{- \delta t}C^{-1}\leq \mu(B(\xi, e^{-t})) \leq c_2 e^{- \delta t} C$ for all $t>t_0$, for some $C=C(c)>0$.
In particular, $d_{\mu}(\xi) \geq \delta$.
\end{proof}

Finally,   
Theorem \ref{ThmBounds} [LB] together with Proposition \ref{PropositionMeasure} give the lower bound
\bea
	\nonumber
	\text{dim}(\textbf{Bad}(\cal{D}_1, e^{-(2c +\d_0 + \log(2))} )) 
	&\geq&  \delta  
	- \frac{  \log( 2\bar c_2 \bar c_1^{-1}) + \lvert \log(1- \bar c_3 e^{ - (2\d-m) c }) \rvert  }
	{ c },
\eea
where $\bar c_i$ are the constants from Proposition \ref{LParameters}.
This finishes the proof of the lower bounds of Theorem \ref{ThmGF}.
\\

\paragraph{The Special Case.}
Let $\Lambda\Gamma=S^n$.
Note that for any formal ball  $(\xi, t_0)$, $\xi \in S^n$,
we can take an isometry from the hyperbolic ball to the upper half space model (again denoted by $\H^{n+1}$) which maps $o$ to $(0, \dots, 0,1)\in \H^{n+1}$ and $\xi$ to $0 \in \R^n \subset  \partial_{\infty}\H^{n+1}$.
If $t_0>0$ is sufficiently large then $B(0, e^{-t_0})$ (with respect to the visual distance) is contained in the Euclidean unit ball $B \subset \R^n$
and we remark that the visual metric $d_o$ restricted to $B$ is bi-Lipschitz equivalent to the Euclidean metric on $B$;
let $c_B\geq 1$ be the bi-Lipschitz constant.

We let $c= \log(m)>c_0$  for some $m \in \N$ sufficiently large (such that $\bar \tau^i_l(c)<1$ below).
Up to modifying $l^i_c$ and $\tau^i_c$ to $\tilde l^i_c = l^i_c + \log(c_B)$ and $\tilde \tau^i_l(c) = c_B^n \tau^i_l(c)$  respectively,
we may use the same arguments as above and assume for any point $\xi \in B$ that \eqref{Decaying} is satisfied 
with respect to the Lebesgue measure and the function $B_1$ (which is with respect to the Euclidean metric).
Note also that we nowhere used the condition that $\xi \in L^{B_1}_{k}(c) $ 
so that the condition becomes obsolete in this setting.
Hence, Lemma \ref{TildeDecayingLemma} shows that \eqref{Decaying2} is satisfied for the parameters 
$\bar l^i_c = \tilde l^i_c  + a$, with $a \equiv 2\sqrt{n} + 3\log(2) $,
and $\bar \tau^i_l(c) \leq e^{n (a-\log(2))} \tilde \tau_l(c)^i$, where $i$ stands for the respective cases.
Recalling that $\tilde \tau_l(c)^1 =  \bar c_1 \ e^{-nc}$ and  $\tilde \tau_l(c)^2 = \bar c_2 \ e^{-\tau c}$,  
Theorem \ref{ThmBounds} [LB] together with Proposition \ref{PropositionCube} yield the lower bound (up to identifications)
\bea
\nonumber
	\text{dim}(\textbf{Bad}(\cal{D}_i, e^{-(2c + \bar l^i_c)} ))  &\geq & \text{dim}(\textbf{Bad}_{\R^n}^{B_{1}}(\cal{F}_i, 2c + \bar l_c^i ) \cap B) 
	\\ \nonumber
	& \geq&\text{dim}(\textbf{Bad}_{\R^n}^{Q_{1}}(\cal{F}_i, 2c + \bar l_c^i ) \cap B) 
	\geq 
	 n  	- \frac{ \lvert   \log(1- \bar \tau_l(c)^i) \rvert  }	{  c }.
\eea
Again, up to modifying $\bar \tau^i_l(c)$ to $\bar \tau^1_l(c) \equiv \bar k^1_l e^{- nc}$ and $\bar \tau^1_l(c) \equiv \bar k^i_l e^{- \tau c}$ 
for suitable constants $\bar k_l^i >0$ (independent on $c$), this gives the result for sufficiently large general $c\geq c_0$.
This finishes the proof of the lower bounds of Theorem \ref{TheoremBoundedGeodesics}.


\subsubsection{The upper bound}
We again distinguish between the cases and start with \paragraph{Case 2.} by showing a Dirichlet-type Lemma.
Recall that  $D_0$ denotes the diameter of the compact set $K$ covering the convex core $\cal{C}M$.

\begin{lemma}
\label{DirichletCompact}
There exists a constant $\kappa_1 \geq 0$ such that 
for all  $\xi \in \Lambda\Gamma$ and $t>t_0$, where $t_0> 2D_0$, there exists $[ \varphi] \in \Gamma/\Gamma_2$  with $D_2([ \varphi] ) \leq t$
such that 
\be
\nonumber
	d_o(\xi, \varphi. \Lambda\Gamma_2) < e^{ 2D_0 + \kappa_1 } e^{-t}.
\ee
\end{lemma}

\begin{proof}
Let $\tilde K$ be a lift of $K$ such that $o\in \tilde K$.
The geodesic ray $\gamma_{o, \xi}$ is contained in $\cal{C}\Gamma$, which is covered by images $\varphi(\tilde K)$, $\varphi\in \Gamma$.
Hence, let $\varphi \in \Gamma$ such that $\gamma_{o, \xi}(t-D_0) \in \varphi (\tilde K)$.
Since $C_2 \subset \cal{C}\Gamma$, some image of $C_2$ under $\Gamma$, say $C_2$ itself, intersects $\tilde K$. 
Thus, $\varphi(C_2)$ intersects $\varphi(\tilde K)$,
and we see that
\be
\nonumber
	D_2([\varphi]) = d(o, \varphi(C_2)) \leq d(o, \gamma_{o, \xi}(t-D_0)) + d( \gamma_{o, \xi}(t-D_0)), \varphi(C_2))  \leq t.
\ee
Moreover, there exists a geodesic line $\alpha$ contained in $\varphi(C_2)$ at distance at most $D_0$ to $\gamma_{o, \xi}(t-D_0)$.
Let $H$ be the hyperbolic half-space such that $ \gamma_{o, \xi}(t-2D_0) \in \partial H$, $H$ orthogonal to  $ \gamma_{o, \xi}$ and $\xi \in \partial_{\infty}H$. 
Hence, one of the endpoints of $\alpha$ (which belongs to $\varphi. \Lambda\Gamma_2$) must lie in the boundary $\partial_{\infty}H$ of $H$.
Remarking that $\partial_{\infty}H$ is a subset of $B(\xi, e^{- (d(o, H) - \kappa_1 )})$ for some universal constant $\kappa_1>0$,
yields the claim.
\end{proof}

Setting $u_*= 2D_0 + \kappa_1$, we see that $\cal{F}$ locally contains $\cal{S}$, which denotes the set of points of $\Lambda\Gamma$.
Moreover, since $\Gamma$ is convex-cocompact, $(\Omega, B_1,\mu)$ satisfies a power law with respect to the parameters $(\delta, c_{1}, c_{2})$,
Proposition \ref{DirichletProposition} shows that
$(\Omega, B_1,\mu)$ is $\tau_u(c)$-Dirichlet with respect to $\cal{F}_2$ and the parameters $(c, u_*)$, where 
$\tau_u(c)= \tfrac{c_{1}}{c_{2}} e^{-\delta(2d_*  + u_*)}\cdot e^{- \delta c} $.

Using \eqref{Inequalities} (giving $[K_c]$) and Proposition \ref{PropositionMeasure}, Theorem \ref{ThmBounds} [UB] establishes
\bea
\nonumber
	\text{dim}(\textbf{Bad}(\cal{D}_2, e^{-c} )) &\leq& \frac{  \log(1- \tfrac{c_{1}}{c_{2}} e^{-\delta(2d_* + u_*)} e^{- \delta c}) -\log(\tfrac{c_1}{c_2} e^{- \delta( c + u_* + 2d_* ) })  }{ c + u_*}
	\\ \nonumber
	&=& \delta -  \frac{ \lvert \log(1- \tfrac{c_{1}}{c_{2}} e^{-\delta(2 d_*+ u_*)} \cdot e^{- \delta c}) \rvert + \log(\tfrac{c_1}{c_2}) - 2\delta  d_*   }{ c + u_*}.
\eea

We are left with
\paragraph{Case 1.} 
We start again with the following Dirichlet-type Lemma that follows from \cite{StratmannVelani}, Theorem 1, which we reformulated in a version best suitable for us. 

\begin{lemma} 
There exists a $t_0\geq 0$ and a constant $\kappa_1>0$ (we may assume $\k_1\geq 1$) such that for any $\xi \in \Lambda\Gamma$, for any $t>t_0$ 
there exists $[\varphi] \in \Gamma/\Gamma_1$  with $D_1([\varphi] ) \leq t$, such that
\be
\label{HBDirichlet}
	d_o(\xi, \varphi. \Lambda\Gamma_1) \leq \kappa_1\ e^{-t/2}\ e^{-D_1([\varphi])/2}.
\ee
\end{lemma}

Fix $c>0$ and let $u_c \equiv c +  2\log(\kappa_1)$.
Recall that  $\bar t_k = s^1_1 + k(c+u_c) - u_c$. 
We need the following refinement of  the above lemma.

\begin{lemma} 
\label{ExistenceDirichlet}
For $\xi \in \Lambda\Gamma$ with $\xi  \in U_{k-1}(c)$ and $\bar t_k>t_0$, 
there exists $[ \varphi] \in \Gamma/\Gamma_1$  with $\bar t_{k-1} + u_c<D_1([ \varphi] ) \leq \bar t_k + u_c$
such that $ \varphi. \Lambda\Gamma_1 \in B(\xi, e^{-  \bar t_k})$.
\end{lemma}

\begin{proof}
Let $\xi \in \Lambda\Gamma$ and $\bar t_k >t_0$.
There exists $[\varphi] \in \Gamma/\Gamma_1$  with $D([\varphi] ) \leq \bar t_k + u_c$  such that \eqref{HBDirichlet} is satisfied.
If $D_1([\varphi] ) \leq  \bar t_{k-1} + u_c =\bar  t_k + u_c - (c + u_c) $, then
\bea
\nonumber
		d_o (\xi, \varphi. \Lambda\Gamma_1)
		&\leq& \kappa_1 e^{-(\bar t_k + u_c)/2}e^{-D_1([\varphi])/2}
		\\ \nonumber 
		&\leq & \kappa_1 e^{- (D_1([\varphi]) + 1/2 ( c+u_c)  )} 
		\\ \nonumber 
		&= & \k_1 e^{- (D_1([\varphi]) + 1/2 ( c+ c  + 2\log(\kappa_1) )    )} 
		= e^{- (D_1([\varphi]) + c )    } .
\eea
Thus, we see that 
\be
\nonumber
	\xi \in B(\varphi. \Lambda\Gamma_1, e^{-(D_1([\varphi]) + c )}) \subset \bigcup_{s_n \leq \bar t_{k-1} + u_c} \cN(R_n, s_n + c) = U_{k-1}(c)^C,
\ee
and we may assume that $\bar t_{k-1} + u_c <D_1([\varphi] ) \leq \bar t_k + u_c$.
In this case, we have
\bea
\nonumber
		d_o (\xi, \varphi. \Lambda\Gamma_1)
		&\leq& \kappa_1 e^{-(\bar t_k+u_c)/2}e^{-D_1([\varphi])/2}
		\\ \nonumber 
		&<& \kappa_1 e^{- (\bar t_k+u_c)/2 -  ( \bar t_{k-1} +u_c)/2  )} 
		\\ \nonumber 
		&= &  \k_1 e^{- \bar t_k  - (2u_c -(c+u_c) ) /2  } 
		= e^{- \bar t_k}
\eea
and hence, $\varphi. \Lambda\Gamma_1 \in B(\xi, e^{- \bar t_k})$ which finishes the proof.
\end{proof}

Combining the Global Measure Formula and the above lemma yield the parameters.

\begin{proposition} 
\label{UParameters}
Given $c>0$, 
$[K_c]$ is satisfied and $(\Omega, B_1, \mu)$ is $\tau_u(c)$-Dirichlet with respect to $\cF_1$ and the parameters $(c, u_c)$ for $u_c \equiv c + 2\log(\kappa_1)$ and
\bea
\nonumber
	K_c & \geq& \tfrac{c_1}{c_2}e^{-(2\d - m) (2d_*  + \d_0) } e^{-(2\d-m)(c+ u_c)}
		\equiv \bar c_1 e^{- (2\delta - m) (c+u_c) }
	\\ \nonumber
	\tau_u(c) &\geq&   \tfrac{c_1}{c_2}e^{ - \delta( 2c + u_c  +3d_*) + m  (c+d_*) } \equiv  \bar c_2 e^{ - (3\delta - m) c  }.
\eea
\end{proposition}

\begin{proof}
Let  $(\xi, \bar t_{k} - d_*) \in \Omega $ be a given a formal ball and $\eta \in B(\xi, e^{- \bar t_k}) \subset B(\xi, e^{- (\bar t_k - d_*)})$.
Using \eqref{DifferenceD} we obtain 
\be
\nonumber
	D_{\bar t_{k+1} + d_*}(\eta) \leq D_{\bar t_{k} - d_*}(\xi) + c+ u_c +  2d_*   + \d_0.
\ee
The Global Measure Formula shows that
\bea
\nonumber
		\mu(B(\eta, e^{-( \bar t_{k+1} +d_*)})) &\geq& c_1 e^{- \delta (\bar t_{k+1} + d_*)} \cdot e^{-(\d - m ) D_{\bar t_{k+1} + d_*}(\eta)}
		\\ \nonumber
		 &\geq& c_1 e^{- \delta (\bar t_{k} + c + u_c  +d_*)} e^{-(\d - m ) D_{\bar t_k - d_*}(\xi)} 
		 		\cdot e^{-(\d - m) ( c+ u_c +  2d_*   + \d_0) }
		 \\ \nonumber
		 &\geq&  \mu(B(\xi, e^{- (\bar t_{k} - d_*) } )  \cdot  e^{-(2\d-m)(c+ u_c)}
		 \\ \nonumber
		&\geq&   \mu(B(\xi, e^{- (\bar t_{k} - d_*) } )   \cdot   K_c.
\eea

Similarly, let  $(\xi, \bar t_{k} - d_*) \in \Omega$ be a given a formal ball such that 
$\xi \in U_{k-1}(c)$.
By Lemma \ref{ExistenceDirichlet}, there exists $[\varphi] \in \Gamma/\Gamma_1$ with $\bar t_{k-1} + u_c< D_1([\varphi]) \leq \bar  t_{k} + u_c$
and $\eta \in B(\xi, e^{- \bar t_k})$, where $\eta \equiv \varphi. \Lambda\Gamma_1$.
Moreover, since 
\be
\nonumber
	D_1([\varphi]) + c+d_*> \bar t_{k-1} + u_c + c + d_* \geq \bar t_k + d_*,
\ee
the ball $B(\eta, e^{-(D_1([\varphi]) + c+d_*) }) \subset B(\eta, e^{-(\bar t_{k} + d_*)})$
which in turn is contained in $B(\xi, e^{- (\bar t_k - d_*)})$.
Finally, we have $D_{D_1([\varphi]) + c +d_*} (\eta) = c +d_* $ and
the Global Measure Formula shows
\bea
\nonumber
		\mu(B(\xi, e^{- (\bar t_{k} -d_*)}) \cap B(\eta, e^{- (D_1([\varphi]) + c + d_*)}) ) &\geq& \mu(B(\eta, e^{- (D_1([\varphi]) + c + d_*)})		
		\\ \nonumber
		 &\geq&
		c_1 e^{- \delta (D_1([\varphi]) + c + d_*)} \cdot e^{-(\d - m ) (c+ d_*)}
		\\ \nonumber
		 &\geq& c_2 e^{- \delta (\bar t_{k} - d_*) } \cdot \tfrac{c_1}{c_2}e^{ - \delta(c + u_c + 2d_*)-(\d -m ) (c+d_*) }
		 \\ \nonumber
		 &\geq&   \mu(B(\xi, e^{- (\bar t_{k} - d_*) } ) \cdot  \tfrac{c_1}{c_2}e^{ - \delta( 2c + u_c  +3d_*) + m  (c+d_*) }
		  \\ \nonumber
		 &\equiv & \tau_u(c)  \cdot  \mu(B(\xi, e^{- (\bar t_{k} - d_*) } ).
\eea
This finishes the proof.
\end{proof}

Finally,  
Theorem \ref{ThmBounds} [UB] together with Proposition \ref{PropositionMeasure} give the upper bound
\bea
\label{UBParabolic}
	\text{dim}(\textbf{Bad}(\cal{D}_1, e^{-c} )) 
	& \leq & 
	\frac{ -\log(K_c) + \log(1-\tau_u(c)) }{ c + u_c}
	\\\nonumber
	& \leq & 
	\frac{  -\log(\bar c_1) + ( 2\delta  - m) (c+u_c)  + \log(1- \bar c_2 e^{- (3\delta - m)c } ) }{c+u_c}
	\\\nonumber
	& \leq & 
	(2\delta  - m) - \frac{  \lvert \log(1- \bar c_2 e^{- (3\delta - m) c } ) \rvert + \log(\bar c_1)  }{2c + 2\log(k_1)},
\eea
where $\bar c_i$ are the constants from Proposition \ref{UParameters}.
This finishes the proof of the upper bounds of Theorem \ref{ThmGF}.
\\


\paragraph{The Special Case.}
Let again $\Lambda\Gamma=S^n$.
For $c>0$ and $a=  \sqrt{n} + 3d_*$  let
$\bar u_c^i \geq u_{c}^i + 2a$ such that $c + \bar u_c^i = \log(m_i)$ (with $m_i \in \N$ minimal).
Moreover, let $\xi_j \in S^n$ be finitely many points such that $B_j = B(\xi_j, \bar t_0)$ cover $S^n$.
As for the lower bound, for each $\xi_j$ we can take again an isometry to the upper half space model which maps $o$ to $e_{n+1}=(0, \dots, 0,1)\in \H^{n+1}$ and $\xi$ to $0 \in \R^n \subset  \partial_{\infty}\H^{n+1}$
as well as  $B_j$ to a subset $\tilde B_j$  contained in the Euclidean unit ball $B$.
Recall that  $\bar t_0  \geq t_0$, a technical constant.
Up to increasing $t_0$, we may even assume that the cube $Q = Q_1(0,\bar t_0) \supset \tilde B_j$ is contained in $B$. 

Note that in the proof of $[\tau_u(c)]$ (or $[\tau_u(c+a)]$) in Proposition \ref{UParameters} it is nowhere necessary to require $t=\bar t_k$.
Hence condition \eqref{Bedingung} is satisfied (with respect to the Lebesgue measure on $\R^n$,  the function $B_1$ and the visual metric $d_{e_{n+1}}$).
Up to adding a multiplicative constant depending only on the ball $B$ and $n$, 
\eqref{Bedingung} is satisfied with respect to the Lebesgue measure on $\R^n$,  the function $B_1$ and the Euclidean metric.
Lemma \ref{TildeDecayingLemma} implies that $(B \times (t_0, \infty), Q_1, \mu)$ satisfies \eqref{Dirichlet2} for the parameters $(c, u_c + 2a)$ and
\be
\nonumber
	\tilde \tau_u(c)^i \geq e^{-n\sigma( a + 2d_*+ \sqrt{n}/\sigma)} \tau_u(c+a)^i \equiv \bar k_u^i e^{- n_ic },
\ee
where $n_1=2$ and $n_2=1$ and $\bar k_u^i$ depend only on $n$.
Finally, since $m_i$ above was chosen minimal there exists a constant $k^i_u \geq 0$ (independent on $c$) such that  
$ c + \bar u_c^i \leq n_i c + k^i_u$.
Thus, Theorem \ref{ThmBounds} [UB] and Proposition \ref{PropositionCube} show (up to identifications)
\bea
\nonumber
	\text{dim}(\textbf{Bad}(\cal{D}_i, e^{- (c + \sqrt{n})  } ) \cap \tilde B_j) 
	&\leq& \text{dim}(\textbf{Bad}_{\R^n}^{Q_{1}}(\cal{F}, c )  \cap Q) 
	\\ \nonumber
	& \leq &  n - \frac{ \lvert   \log(1- \bar k_u^i \ e^{- n_ic} )\rvert}{ n_i c + k_u^i }.
\eea
This finishes the proof of the upper bounds of Theorem \ref{TheoremBoundedGeodesics}.


\subsection{Toral Endomorphisms}
\label{ToralEndo}

For the motivation of the following result, we refer to Broderick, Fishman and Kleinbock \cite{BFK} and references therein.
For $n\in \N$, let $\cal{M}=(M_k)$ be a sequence of real matrices $M_k\in GL(n,\R)$, with $t_k= \lVert M_k \rVert_{op}$ (the operator norm), 
and $\cal{Z}=(Z_k)$ be a sequence of $\tau_k$-separated%
\footnote{ That is, for every $y_1, y_2 \in Z_k$ we have $d(y_1,y_2) \geq \tau_k>0$.}
subsets of $\R^n$.
Define 
\be
\nonumber
	E_{\cal{M}, \cal{Z}} \equiv \{x \in \R^n  : \exists\  c=c(x) >0 \text{ such that } d(M_kx, Z_k) \geq c \cdot \tau_k \text{ for all } k \in \N_0\},
\ee
where $d$ is the Euclidean distance.
For $c>0$,  let $E_{\cal{M}, \cal{Z}}(c)$ be the elements $x\in E_{\cal{M}, \cal{Z}}$ with $c(x) \geq c$.

We assume that, independently of $t\in \R^+$, for all $c>0$  we have
\be
\label{Count}
	\lvert \{ k \in \N : \log(t_k/\tau_k) \in (t-c, t] \} \rvert \leq \varphi(c),
\ee
for some function $\varphi : \R^+ \to \R^+$.
The sequence $\cal{M}$ is \emph{lacunary} if $\inf_{k\in \N} \tfrac{t_{k+1}}{t_k} \equiv \lambda >1$ and
the sequence $\cal{Z}$ is \emph{uniformly discrete} if there exists $\tau_0 >0$ such that every set $Z_k$ is $\tau_0$-separated.
Note that if $\cal{M}$ is lacunary and $\cal{Z}$  is uniformly discrete, then \eqref{Count} holds and $\varphi$ is bounded by $\varphi(c) \leq c/\log(\lambda)$.

Let again $\cal{S}$ denote the set of affine hyperplanes in $\R^n$ and recall that the Lebesgue measure is absolutely $(1, c_n)$-decaying (with respect to $\cal{S}$ and the function $\psi= B_1$).
Using similar arguments for the proof as \cite{BFK, Weil2}, we want to show the following lower bounds.
An upper bound will be considered below.

\begin{theorem}
Let $X \subset \R^n$ be the support of an absolutely $(\tau, c_{\tau})$-decaying measure $\mu$ (with respect to $\cal{S}$ and $\psi=B_1$) 
which also satisfies a power law with respect to the exponent $\delta$ and constants $c_1$, $c_2$.
Let  $\cal{M}$ and  $\cal{Z}$ be as above satisfying \eqref{Count} with $\varphi(c) \leq e^{\bar \tau c}$, where $0<\bar \tau<\tau$.
Then, there exists $c_0>0$ such that for all $c>c_0$ we have
\bea
\nonumber
	\emph{dim}( E_{\cal{M}, \cal{Z}}(e^{- (2c + \log(12))})\cap X)
	 &\geq &  \delta - \frac{ \log( 2 \tfrac{c_1^2}{c_2^2} ) + 2\delta \log(2) + \lvert \log(1- c_{\tau}2^{\tau} \varphi(c)  e^{-\tau c }  \rvert  }{ c }.
\eea

In particular, (if $\mu$ denotes the Lebesgue measure) there exist constants $k_l, \bar k_l>0$ and $c_0>0$,
such that for all $c>c_0$ we have 
\bea
\nonumber
	\emph{dim}( E_{\cal{M}, \cal{Z}}(e^{- c}))
	&\geq &
	 n  	- \frac{ \lvert   \log(1- \bar k_l \ \varphi(c/2) e^{-c/2}) \rvert  }{  c/2 - k_l }.
\eea
\end{theorem}

Note that when $\cM = (M^k)$ where $M \in GL(n, \Z)$ and each $Z_k= y+ \Z^n$ for some $y \in \R^n$,
then, under certain assumptions on $M$, a similar lower bound for $\text{dim}( E_{\cal{M}, \cal{Z}}(e^{- c}))$ follows from Abercrombie, Nair \cite{AbercrombieNair}.

\begin{proof}
Let $v_k \in \R^n$ be the unit vector such that $\lVert M_k v_k \rVert = t_k$ 
and if $V_k \equiv \{M_kv_k\}^{\perp}$ is the subspace orthogonal to $M_kv_k$, let $W_k \equiv M_k^{-1}(V_k)$.
Then, for $k \in \N$ and $z \in Z_k$  we  define the subsets 
\be
\nonumber
	Y_k(z)\equiv (M_k^{-1}(z) + W_k)  \cap M_k^{-1}(B(z, \tau_k/4) ).
\ee
Set $s_k \equiv \log(\tau_k/t_k)$, which we reorder such that $s_k \leq s_{k+1}$, so that we obtain a discrete set of sizes.
For $k\in \Lambda \equiv \N$ let  the resonant set $R_k$ be given by
\bea
\nonumber
	R_k &\equiv& \{ x \in Y_l(z_l): z_l \in Z_l \text{ and }  \log(t_l / \tau_l)   \leq s_k \} 
	\\ \nonumber
		&=& \{ x \in Y_l(z_l): z_l \in Z_l \text{ and }  \frac{\tau_l}{t_l} \geq \frac{\tau_k}{t_k} \}, 
\eea
which gives a increasing and discrete family $\cal{F}=\{R_k, s_k\}$.

Note that for all $x\in \R^n$ we have $\lVert x \rVert \geq \lVert M_kx \rVert /t_k$.
Hence, for distinct points $z_1$, $z_2 \in Z_k$, $Y_k(z_1)$ and $Y_k(z_2)$ are subsets of parallel affine hyperplanes and  we have
\bea
\label{Disjoint3}	
 	\lVert Y_k(z_1) - Y_k(z_2) \rVert &\geq&
	\lVert M_k^{-1} (B(y_1, \tau_k/4)) -M_k^{-1} (B(y_2, \tau_k,/4)) \rVert 
	\\ \nonumber
	&\geq& \frac{\tau_k - 2 \tau_k/4}{t_k}  =  \frac{\tau_k }{2t_k}  \geq \tfrac{1}{2} e^{-s_k},
\eea
since $Z_k$ is $\tau_k$-separated.
Let $l_c= \log(4) + \log(3)$ independent of $c$.
Given a closed ball $B=B(x, 2 e^{-(t+l_c)}) \subset \R^n$ with $x\in X$, for every $k\in \N$ with $s_k \leq t$, 
it follows from \eqref{Disjoint3} that at most one of the sets $Y_k(y)$, $y\in Z_k$, can intersect $B$.
Moreover, for $c>0$, the number of $k \in \N$ with $s_k \in (t-c, t]$
is bounded by $\varphi(c)$ by  \eqref{Count}. 
Recall that $\cR(t, c) \equiv \cR(t) - \cR(t-c)$.
Thus, there exist at most $N=\lfloor \varphi(c) \rfloor$ affine hyperplanes $L_1, \dots, L_N \in \cal{S}$ such that
\be
\nonumber
	B(x, 2 e^{-(t + l_c)}) \cap  \cR(t,c)  \subset B(x, 2e^{-(t + l_c)}) \cap \bigcup_{i=1}^N L_i.
\ee
Since $(\Omega, B_1,\mu)$ is absolutely $(\tau, c_{\tau})$-decaying with respect to $\cal{S}$,  for $c\geq \log(3) + \log(2)$ and $B=B(x,e^{-(t+l_c + d_*)})$ we have
\bea
\nonumber
	\mu(B \cap \cal{N}_{e^{-(t+l_c +d_* + c - \log(2))}}\big(R(t,c) \big) )
		&\leq& \sum_{i=1}^N  \mu( B \cap \cal{N}_{e^{-(t + l_c +d_* + c - \log(2))}}(L_i)) 
		\\ \nonumber
		&\leq& \varphi(c) \cdot c_{\tau} e^{-\tau (c - \log(2))} \mu( B) \equiv \tau_l(c)\ \mu(B).
\eea
Note that, since $\varphi(c) \leq e^{ \bar \tau c}$ with  $\bar \tau <\tau$,  for all  $c>c_0 = \log(c_{\tau} 2^{\tau})/(\tau - \bar \tau) $  we have $\tau_l(c)<1$.
Using the remark after \eqref{Dirichlet}, we in fact showed that $(\Omega, B_1,\mu)$ is $\tau_l(c)$-decaying with respect to $\cal{F}$ and the parameters $(c, l_c)$.
Moreover, $B_1$ is $\log(2)$-separating with respect to the sets $R(t,c)$.

Finally, let $x \in $ \textbf{Bad}$_X^{B_1}(\cal{F}, c)$, that is, for every $k\in \N$ and $y\in Z_k$  we have
$	d(x, Y_k(y))\geq e^{-(s_k+c)} \geq   e^{-c}  \tau_k/t_k$.
Assume that $M_k x \in B(y, \tau_k/4)$.
Then,  $ x \in \cal{N}_{e^{-c} \tau_k/t_k}(M_k^{-1}(y) + W_k)^C \cap M_k^{-1}(B(y, \tau_k/4))  $
and we can write the vector $v = x- M_k^{-1}(y)$ as $v=w + \tilde c \tau_k/t_k v_k$ with $w \in W_k$ and $\tilde c \geq e^{-c}$.
Hence, since $M_kW_k=V_k$ is orthogonal to $M_kv_k$, 
\be
\nonumber
	d(M_k x, y) 
	= \lVert M_k v \rVert 
	= \lVert M_k w + \tilde c\ \tfrac{\tau_k}{t_k} M_k v_k \rVert \geq \tilde c\ \tfrac{\tau_k}{t_k} \lVert M_k v_k \rVert \geq e^{-c} \tau_k,
\ee
and we showed that \textbf{Bad}$_X^{B_1}(\cal{F}, c) \subset E_{\cal{M}, \cal{Z}}(e^{-c})\cap X$.

Using \eqref{Inequalities} (giving $[k_c, \bar k_c]$) and Proposition \ref{PropositionMeasure}, Theorem \ref{ThmBounds} [LB] shows for $c>c_0$,
\bea
\nonumber
	\text{dim}( E_{\cal{M}, \cal{Z}}(e^{- (2c + l_c)})\cap X)
	 &\geq &  \delta - \frac{ \log( 2 \tfrac{c_1^2}{c_2^2} ) + 2\delta \log(2)  + \lvert \log(1- c_{\tau}2^{\tau} \varphi(c)  e^{-\tau c }  \rvert  }{ c }.
\eea

For the second part, when $\mu$  denotes the Lebesgue measure and $\psi=Q_1$,
we let $c=\log(m) > c_0 $ be sufficiently large such that $\bar \tau_l(c) <1$ (as below).
The above arguments  apply analogously for the function $\psi=Q_1$ 
with possibly different parameters $\bar l_c = l_c + k$ and $\bar \tau_l(c) = \bar k \tau_l(c) \equiv \bar k_l \varphi(c) e^{- \tau c }$, 
where $\tau=1$, for some constants $k, \bar k, \bar k_l >0$ independent on $c$.
Hence,  Theorem \ref{ThmBounds} [LB] together with Proposition \ref{PropositionCube}  show
\bea
\nonumber
	\text{dim}( E_{\cal{M}, \cal{Z}}(e^{- (2c + \bar l_c)})
	&\geq &
	 n  	- \frac{ \lvert   \log(1- \bar k_l \ \varphi(c) e^{-c}) \rvert  }{  c },
\eea
since $\textbf{Bad}^{Q_1}_{\R^n}(2c + \bar l_c) \subset \textbf{Bad}^{B_1}_{\R^n}(2c + \bar l_c)$.
This finishes the proof.
\end{proof}


\subsubsection{A discussion on the upper bound.}

For a non-trivial upper bound
we restrict to the following example.
Let $Z=Z_k$ for all $k\in \N$ where $Z$ is a $\tau_0$-spanning%
\footnote{ That is, for any $x \in \R^2$, there exists $z \in Z$ such that $d(x, z)< \tau$.}
 set of $\R^2$.
Let $\cal{M}=(M_k)$ with $M_k = M^k$ where $M \in GL(2, \R)$ is a real diagonizable matrix with eigenvalues $\lambda \geq \beta >1$.
For simplicity, let $M=\text{diag}(\lambda, \beta)$, where $\lambda$ and $\beta$ are integers, and let  $Z=\Z^2$.
Under these assumptions we claim that there exist constants $c_0>0$ and $k_u, \bar k_u >0$ such that for $c>c_0$ we have
\bea
\label{ExampleUB}
	\text{dim}(E_{\cal{M}, \cal{Z}}(e^{-c}) \cap [-1,1]^2) &\leq&
		2  -  \tfrac{\log(\beta)}{\log(\lambda)} \frac{  \lvert \log(1- \bar k_u\ (\lambda \beta)^{-c/\log(\beta)} \rvert  }{  c + k_u} .
\eea

\begin{remark}
Note that when $\beta=1$ and $c$ is sufficiently large there exist $M$-invariant strips of $\R^2$, consisting of badly approximable elements $x$ with $c(x)\geq e^{-c}$. Hence \eqref{ExampleUB} fails.

If $M= D\text{diag}(\lambda, \beta)D^{-1}$ for $D\in GL(2, \R)$ we may consider $\tilde \psi(x,t) = x +  D R_{\bar \sigma}(0,t)$, for $R_{\bar \sigma}$  as in  \eqref{RectangleExampleUB} below,  in oder to obtain an upper bound.
\end{remark}

\begin{proof}[Sketch of the proof of \eqref{ExampleUB}.]
For $\sigma_1\equiv \log(\lambda) \geq  \log(\beta) \equiv \sigma_2>0$ let $\bar \sigma=(\sigma_1, \sigma_2)$.
On $\Omega=\R^2 \times \R^+$ we consider the rectangle function $R_{\bar \sigma}(x, t)$,
\be
\label{RectangleExampleUB}
	R_{\bar \sigma}(x, t) = x + B(0, e^{- \sigma_1 t}) \times B(0, e^{- \sigma_2 t}).
\ee
Note that  $(\R^2 \times \R, R_{\bar \sigma}, \mu$) satisfies a $(\sigma_1 + \sigma_2)$-power law, 
where $\mu$ denotes the Lebesgue measure, and  is $d_*=\log(2)/\sigma_2$ contracting.

Since $\lambda$, $\beta\in \N_{\geq 2}$ for every $\bar c \in \N$ we have $\lambda^{\bar c}  = p$, $\beta^{ \bar c}= q$ with $p,q\in \N_{\geq 2}$,
and we can partition $R_{\bar \sigma}(x, t)$ into $pq$ rectangles $R_{\bar \sigma}(x_i, t + \bar c)$ as in \eqref{CubeCover}.
So fix a parameter $c>0$ sufficiently large and let $\bar c\in \N$ be the minimal integer such that $\bar c  \geq c/\sigma_2 + \log(6)/\sigma_2 +1$. 

\begin{lemma}
\label{ProvideCover}
For any rectangle $R_0 = R_{\bar \sigma}(x, t)$ with $x \in [-1,1]^2$, $t\geq 0$,
we can cover $R_0 \cap E_{\cal{M}, \cal{Z}}(e^{-c})$ by $N(\bar c)$ rectangles $R_{\bar \sigma}(x_i, t + \bar c)$
with (for some $\bar k_u = \bar k_u(\bar \sigma)$)
\be
\nonumber
	N(\bar c) \leq e^{(\sigma_1  + \sigma_2) \bar c}  (1-  \bar k_u \ e^{-(\sigma_1 + \sigma_2) c/\sigma_2} ) .
\ee
\end{lemma} 

\begin{proof}
Let $ k \in \N$ be the minimal integer such that $ k \geq t +  \log(2)/\sigma_2$.
Then we have
\be
\nonumber
	R \equiv M^{k} R_0 
	= M^{ k}x + B(0, \lambda^{ k} e^{- \sigma_1 t}) \times B(0, \beta^{k} e^{- \sigma_2 t}) 
\ee
which is a rectangle of edge lengths in $[2, 2^{\sigma_1/\sigma_2} \lambda ] \times [2, 2\beta]$
since $ k$ was chosen minimal.
In particular, there exists an integer point $z \in \Z^2$ such that $Q_1(z, c)$ (the cube) is contained in $R$. 
Hence, 
\bea
\label{CoverR0}
	R_0 = M^{- k}R &\supset& M^{-  k}Q_1(z,c) 
	\\ \nonumber
	&=& M^{-  k}z + B(0, e^{- \sigma_1  k - c}) \times B(0, e^{-\sigma_2  k - c})
	\supset R_{\bar \sigma}( M^{- k} z,   k + c/\sigma_2).
\eea
This in particular shows that  $R_0 \cap E_{\cal{M}, \cal{Z}}(e^{-c}) \subset R_0 \cap \psi( M^{- k} z,  k + c/\sigma_2)^C$
and it suffices to cover the sup set.
Again since $ k$ is minimal, $t + c/\sigma_2 + \log(2)/\sigma_2 + 1 \geq  k + c/\sigma_2$,
so that  a rectangle $R_{\bar \sigma}(x, t + \bar c) \subset R_0$ that intersects $R_0 \cap R_{\bar \sigma}( M^{- k} z, k + c/\sigma_2))^C$ 
cannot intersect $R_{\bar \sigma}( M^{-  k} z,  k + c/\sigma_2 + \log(3)/\sigma_2)$.
Moreover, we have
\be
\nonumber
	\mu(R_0 \cap \psi( M^{- k} z, k + c/\sigma_2 + \log(3)/\sigma_2) 
	\geq \bar k_u \ e^{-(\sigma_1 + \sigma_2) c/\sigma_2} \mu(R_0) \equiv \tau_u(\bar c)\mu(R_0),
\ee
for some constant $\bar k_u>0$.
Thus, if $R_{i}= R_{\bar \sigma}(x_i, t + \bar c)$ are the $pq$ rectangles from the partition of $R_0$,
then we can bound the number of rectangles $R_{i}$ not intersecting $R_{\bar \sigma}( M^{- k} z, k + c/\sigma_2 + \log(3)/\sigma_2)$,
hence covering $R_0 \cap E_{\cal{M}, \cal{Z}}(e^{-c})$,
by $pq (1- \tau_u(\bar c)) = e^{(\sigma_1  + \sigma_2) \bar c}  (1- \tau_u(\bar c))  $. 
This finishes the proof.
\end{proof}

Let $R_0= \bar R_{\sigma}(0) = [-1,1]^2$.
Now,  given a rectangle $\tilde R=R_{0 i_1 \dots i_k} =  \bar R_{\sigma}(x_{0i_1 \dots i_k}, k \bar c)$,
cover of $\tilde R \cap E_{\cal{M}, \cal{Z}}(e^{-c})$ by $N(\bar c)$ rectangles  $\bar R_{\sigma}(x_{0i_1 \dots i_ki_{k+1}}, (k+1) \bar c)$ provided by Lemma \ref{ProvideCover}.
As in Lemma \ref{ConstructCover}, we obtain a covering $\cU_{k+1}$ of $E_{\cal{M}, \cal{Z}}(e^{-c})$
by  $N(\bar c)^k$ rectangles  $\bar R_{\sigma}(x_{0i_1 \dots i_k i_{k+1}}, (k+1) \bar c)$.
Similar to the proof of Lemma \ref{UBFormula2}, using Lemma \ref{RefineCover} (with $\hat \sigma=\sigma_1$),
we obtain a cover $\tilde \cU_{k+1}$ of $E_{\cal{M}, \cal{Z}}(e^{-c})$ by sets of diameter $2e^{-\sigma_1 (k+1)\bar c}$ with
$\lvert\tilde \cU_{k+1}\rvert \leq \lvert \cU_{k+1} \rvert \cdot 2e^{(\sigma_1 - \sigma_2 ) (k+1)\bar c}$.
Finally, similar to \eqref{Schranke3}, this shows 
\bea
	\nonumber
	\text{dim}(E_{\cal{M}, \cal{Z}}(e^{-c}) \cap R_0) &\leq&
	 \liminf_{k \to \infty} \frac{\log( N(\bar c)^k \cdot 2e^{(\sigma_1 - \sigma_2 ) k\bar c} )}{- \log(2 e^{- \sigma_1 k \bar c}) }
	\\ \nonumber
	&\leq& \frac{\sigma_1 - \sigma_2}{\sigma_1} +  \frac{ \log(e^{(\sigma_1  + \sigma_2) \bar c}  (1-  \bar k_u \ e^{-(\sigma_1 + \sigma_2) c/\sigma_2} ) )  }{\sigma_1 \bar c}
	\\ \nonumber
	&\leq& 2  -  \frac{ \lvert \log(1- \bar k_u\ e^{-(\sigma_1 + \sigma_2)  c/\sigma_2}) \rvert  }{\sigma_1 \bar c} 
	\\ \nonumber
	&\leq& 2  -  \tfrac{\log(\beta)}{\log(\lambda)} \frac{  \lvert \log(1- \bar k_u\ e^{-(\sigma_1+ \sigma_2)c /\sigma_2}) \rvert  }{  c + k_u} ,
\eea
where we used that $\bar c$ is chosen minimally, so that $\bar c \leq c/\sigma_2 + \log(6)/\sigma_2 +2 \leq c/\sigma_2 + k_u/\sigma_2$ for some $k_0\geq 0$.
This finishes the proof.
\end{proof}


\bibliographystyle         {plain}      
\bibliography{cup_ref.bib}

 \end{document}